\documentclass[reqno]{amsart}

\usepackage{graphicx}
\usepackage{amscd}
\usepackage{amssymb}
\usepackage{amstext}
\usepackage{amsmath}
\usepackage{enumerate}

\input epsf

\newtheorem{theorem}{Theorem}[section] % 1st argument is your name for it
\newtheorem{lem}[theorem]{Lemma}     % 2nd argument is what is printed
\newtheorem{sublem}[theorem]{Sublemma}

\newtheorem{prop}[theorem]{Proposition}
\theoremstyle{definition}
\newtheorem{defn}{Definition}[section]
      
\newtheorem{assumption}[defn]{Assumption}
\newtheorem{rem}[defn]{Remark}

\newtheorem{exm}[defn]{Example}
\newtheorem{exdefn}[defn]{Example-Definition}
\def\R{{\mathbb R}}
\def\C{{\mathbb C}}
\def\Z{{\mathbb Z}}
\def\Q{{\mathbb Q}}

\def\e{{\epsilon}}

\def\MM{{\mathcal M}}

\def\JJ{{\mathcal J}}

\title[Lagrangian Floer theory over integers]% end with percent
 {Lagrangian Floer theory over integers:\\
spherically positive symplectic manifolds} % This is the full title of the paper

\author{Kenji FUKAYA, Yong-Geun Oh,
Hiroshi Ohta, Kaoru Ono}

\dedicatory{Dedicate to Professor Dennis
Sullivan on his seventieth birthday}

\thanks{KF is supported partially by JSPS Grant-in-Aid for Scientific Research
No.18104001 and Global COE program G08, YO by US NSF grant \# 0904197, HO by JSPS Grant-in-Aid
for Scientific Research No.23340015, and KO by JSPS Grant-in-Aid for
Scientific Research, No.21244002.}
\begin{document}
\begin{abstract}
In this paper 
we study the Lagrangian Floer theory over $\Z$ or
$\Z_2$. Under an appropriate assumption on ambient symplectic
manifold, we show that the whole story of Lagrangian Floer theory in
\cite{fooo-book1}, \cite{fooo-book2} can be developed over $\Z_2$ coefficients, and over
$\Z$ coefficients when Lagrangian submanifolds are relatively spin.
The main technical tools used for the construction are the notion of
the sheaf of groups, and stratification and compatibility of the
normal cones applied to the Kuranishi structure of the moduli space
of pseudo-holomorphic discs.
\end{abstract}
\maketitle

\renewcommand{\thefootnote}{\arabic{footnote}}

\tableofcontents
\section{Introduction}

\label{intro}
In this paper we establish Floer theory over $\Z_2$
(resp. $\Z$) for a Lagrangian submanifold (resp. relatively spin
Lagrangian submanifold) under the assumption that the ambient
symplectic manifold is spherically positive. Here we define:
\begin{defn}\label{34.1}
Let $J$ be an almost complex structure on $M$. We call $J$  {\it
spherically positive} if every $J$-holomorphic sphere $v: S^2 \to M$
with $c_1(M)[v] \leq 0$ is constant.
\par
For a symplectic manifold $(M,\omega)$, we denote by $\mathcal
J_{(M,\omega)}^{c_1 > 0}$ the set of spherically positive almost
complex structures which are compatible with $\omega$.
\par
We call a symplectic manifold $(M,\omega)$ {\it spherically positive}
if there exists a compatible spherically positive almost complex structure $J$.
Sometimes we say $(M,\omega,J)$ is spherically positive.
\end{defn}
Throughout this paper all symplectic manifolds are assumed to be compact or tame. We always assume that Lagrangian submanifolds are compact.
\begin{exm}
\begin{enumerate}[(1)]
\item
We recall that a symplectic manifold $(M,\omega)$ is (positively
spherically) {\it monotone}
if there exists $c>0$ such that $[\omega](\alpha) = cc_1(M) (\alpha)$
for any $\alpha \in \pi_2(M)$.
It is easy to see that any monotone symplectic manifolds
are spherically positive. For example, $\mathbb CP^n$, $\C^n$, $T^{2n}$
are spherically positive with respect to the standard complex structure.
\item
If  $M$ is a Fano manifold,
$(M,\omega)$ is spherically positive for any K\"ahler form
$\omega$.
\item
Any 4-dimensional symplectic manifold is spherically positive with
respect to generic almost complex structures.
\item
Any product of spherically positive symplectic manifolds is
spherically positive.
\end{enumerate}
\end{exm}
\begin{rem}
The set $\JJ_{(M,\omega)}^{c_1 > 0}$
may be neither path connected nor dense in the set of almost complex structures
compatible with $\omega$, except when $\dim M = 4$
or $M$ is monotone.
\par
If $\psi: (M,\omega) \to (M',\omega')$ is a symplectic
diffeomorphism, it induces a bijection
$\psi_*:\JJ_{(M,\omega)}^{c_1 > 0} \to \JJ_{(M',\omega')}^{c_1 >
0}$ in an obvious way.
\end{rem}

Our main result in this paper is as follows.
For a commutative ring $R$ with unit
we define the {\it universal Novikov ring}
$\Lambda_{0,{\rm nov}}^R$ over $R$ by
\begin{align}\label{eq:nov}
 \Lambda^{R}_{{\rm nov}} & = \left.\left\{ \sum_{i=0}^{\infty} a_i T^{\lambda_i}e^{\mu_i}  \
\right\vert \ a_i \in R,
\ \mu_i \in \Z, \ \lambda_i \in \R, \
\lim_{i \to \infty}\lambda_i = +\infty \right\},
\\
 \Lambda^{R}_{0,{\rm nov}} & =
\left.\left\{ \sum_{i=0}^{\infty} a_i T^{\lambda_i}e^{\mu_i} \in \Lambda_{{\rm nov}}
\ \right\vert \ \lambda_i \geq 0 \right\}.
\label{eq:nov0}
\end{align}
Here $T,e$ are indeterminates.
We define
a filtration $F^{\lambda}\Lambda_{{\rm nov}}^R=
\{ \sum_i a_iT^{\lambda_i}e^{\mu_i} ~\vert~ \lambda_i \ge \lambda \}$ and a degree
$\deg (aT^{\lambda}e^{\mu})=2\mu$
for $aT^{\lambda}e^{\mu} \in \Lambda_{{\rm nov}}^R$.
Then $\Lambda_{0,{\rm nov}}^R$ and $\Lambda_{{\rm nov}}^R$
become filtered graded commutative rings.
We denote by
$\Lambda_{0,{\rm nov}}^{+,R}$ the ideal of
$\Lambda_{0,{\rm nov}}^R$ that consists of elements
such that $\lambda_i >0$.
\begin{theorem}\label{theoremA}
Let $(M,\omega,J)$ be a spherically positive symplectic manifold
and  $L$ its Lagrangian submanifold.
\begin{enumerate}
\item We can associate a structure of
filtered $A_{\infty}$ algebra $\{\mathfrak m_{k}^J\}_{k=0}^{\infty}$ on $H^*(L;\Lambda_{0,{\rm nov}}^{\Z_2})$,
which depends only on the connected component of
$\mathcal J_{(M,\omega)}^{c_1 > 0}$ containing $J$ up to isomorphism.
\item
If $\psi: (M,L) \to (M',L')$ is a symplectic diffeomorphism, we
can associate to it an isomorphism $\psi_* :
(H^*(L;\Lambda_{0,{\rm nov}}^{\Z_2}),\{\mathfrak m_{k}^J\})
\to (H^*(L';\Lambda_{0,{\rm nov}}^{\Z_2}),\{\mathfrak m_{k}^{\psi_*J}\})$ of
filtered $A_{\infty}$ algebras  whose homotopy class depends only on
the isotopy class of symplectic diffeomorphism $\psi: (M,L) \to
(M',L')$.
Moreover $(\psi\circ \psi')_{\ast}=\psi_{\ast} \circ \psi'_{\ast}$.
\item
The Poincar\'e dual $PD([L]) \in H^0(L;\Lambda_{0,{\rm nov}}^{\Z_2})$ of the fundamental class $[L]$
is the unit of our filtered $A_{\infty}$ algebra.
The homomorphism $\psi_*$ is unital.
\item
If $L$ is relatively spin and the homology group $H(L;\Z)$ is a torsion free
$\Z$ module, the above holds for the coefficient ring
$\Lambda_{0,{\rm nov}}^{\Z}$ in place of $\Lambda_{0,{\rm nov}}^{\Z_2}$.
The tensor product of this filtered $A_{\infty}$ algebra with $\Q$ becomes
unitally homotopy equivalent
to the filtered $A_{\infty}$ algebra of Theorem A \cite{fooo-book1}.
\end{enumerate}
\end{theorem}
\begin{rem}
We note that in Theorem \ref{theoremA} we put some assumption on
the symplectic manifold $(M,\omega)$ but not on its Lagrangian
submanifold $L$. For example, Theorem \ref{theoremA} applies to {\it any}
compact Lagrangian submanifold in $\C^n$.
\end{rem}
\begin{rem}\label{Zfreeremark}
In case $L$ is relatively spin but the homology group $H(L;\Z)$ is not free,
we can prove a slightly different but a similar result. See Section \ref{Zcoefficient}.
The same remark applies to Theorem \ref{theoremF}.
\end{rem}
Let $L_1,L_0$ be Lagrangian submanifolds of $(M,\omega)$.
We assume that they have clean intersection
in the following sense.
\begin{defn}
A pair of submanifolds $L_1,L_0$ of a manifold $M$ is said to have {\it
clean intersection} if $L_1 \cap L_0$ is a smooth submanifold and
$$
T_xL_1 \cap T_x L_0 = T_x(L_1\cap L_0)
$$
for each $x \in L_1 \cap L_0$.
\end{defn}
We decompose
$$
L_1 \cap L_0 = \bigcup_h R_h
$$
to the connected components.
Let $\mu_L(R_h)$ be the Bott-Morse version of Maslov index which is
defined
in Proposition 3.7.59 (3.7.60.1), also p.146,
\cite{fooo-book1}.
Actually we need to fix an extra data, (denoted by
$w$ there) to define the Bott-Morse version of Maslov index. $[
\cdot ]$ denotes the degree shift.
We put:
$$
C(L_1,L_0;\Lambda_{0,{\rm nov}}^{\Z_2}) = \bigoplus_h H(R_h;{\Z_2})[\mu_L(R_h)] \otimes \Lambda_{0,{\rm nov}}^{\Z_2}.
$$
In case $L_1$ and $L_0$ are oriented,
we defined in \cite{fooo-book2} Subsection 8.8 a local system $\Theta_{R_h}^-$ on each of the
connected component
$R_h$ of $L_1 \cap L_0$ by a homomorphism
$: \pi_1(R_h) \to \{\pm 1\} = \operatorname{Aut}(\Z)$.
We then put
$$
C(L_1,L_0;\Lambda_{0,{\rm nov}}^\Z) = \bigoplus_h H(R_h;\Theta_{R_h}^-)[\mu_L(R_h)] \otimes \Lambda_{0,{\rm nov}}^\Z.
$$
\par
\begin{theorem}\label{theoremF}
Let $L_1,L_0$ be a pair of Lagrangian submanifolds of $M$.
\begin{enumerate}
\item
$C(L_1,L_0;\Lambda_{0,{\rm nov}}^{\Z_2})$ has a
structure of unital filtered $A_{\infty}$ bimodule over the pair
$$
\left((H(L_1;\Lambda_{0,{\rm nov}}^{\Z_2}),\{\mathfrak m_k^J\}_{k=0}^{\infty}),
(H(L_0;\Lambda_{0,{\rm nov}}^{\Z_2}),\{\mathfrak m_k^J\}_{k=0}^{\infty})\right).
$$
\item
This depends only on the connected component of $J\in
\JJ_{(M,\omega)}^{c_1 > 0}$ up to an isomorphism of unital filtered
$A_{\infty}$ bimodule.
\item
A similar functoriality as Theorem \ref{theoremA} (2) holds.
\end{enumerate}
\par
If $(L_0,L_1)$ is a relatively spin pair and $H(L_1 \cap L_0;\Theta_{R_h}^-)$ is
a free $\Z$ module, we can take $\Lambda_{0,{\rm nov}}^{\Z}$ as a coefficient ring in place
of $\Lambda_{0,{\rm nov}}^{\Z_2}$.
This bimodule is homotopy equivalent to the ones in \cite{fooo-book1} Theorem F after taking $\otimes \Q$.
\end{theorem}
We can generalize various other results of \cite{fooo-book1}, \cite{fooo-book2} to
ones over $\Lambda_{0,{\rm nov}}^{\Z_2}$ or $\Lambda_{0,{\rm nov}}^{\Z}$ in a
similar way. See Section \ref{whichandhow} for the precise statement.
\par
In \cite{fooo-book1}, \cite{fooo-book2} we worked over $\Lambda_{0,{\rm nov}}^{\Q}$ because we
used multivalued perturbations (multisections) to define
(virtual) fundamental chains of various moduli spaces involved in
the construction. We need a multisection to perturb the moduli space
to achieve transversality since in general there is a nontrivial
automorphism of elements of the moduli spaces involved. More precisely, the moduli
space we use is the stable map compactification $\mathcal
M_{k+1}^{\text{\rm main}}(\beta)$ of the space of pseudo-holomorphic
discs with $k+1$ boundary marked points.
\par
In order to prove Theorems \ref{theoremA} and \ref{theoremF} we need to
use single valued sections in place of multisections that
satisfy some transversality properties.
The space $\mathcal M_{k+1}^{\text{\rm main}}(\beta)$ is regarded locally as
the zero set of a section of an orbi-bundle over an orbifold.
Transversality of the zero set in general fails even for a generic
section of an orbi-bundle over an orbifold.
Our main technical result to prove Theorems \ref{theoremA} and \ref{theoremF}
is stated in Section \ref{technicalresult}. (Theorem \ref{maintechnicalresult}.)
\par
We next mention some of applications of Theorems \ref{theoremA} and \ref{theoremF}.
Since many of them are straightforward generalization of the results
proved over $\Q$ coefficients in \cite{fooo-book1},\cite{fooo-book2}
we mention only a few of them below.
\begin{theorem}\label{theoremKs}
Let $L$ be a compact Lagrangian submanifold of $\C^n$ that satisfies
$H^2(L;\Z_2) = 0$.
Then its Maslov class $\mu_L \in H^1(L;\Z)$ is nonzero.
\end{theorem}
Theorem \ref{theoremKs} is proved in \cite{fooo-book1} under a
similar but different assumption that $H^2(L;\Q) = 0$ and $L$ is
relatively spin (Theorem K).
\begin{theorem}\label{theoremL} Let $L \subset (M,\omega)$ be a Lagrangian submanifold
such that the Maslov index homomorphism $ \mu_L:\pi_2(M,L) \to \Z $
is trivial. Assume  $H^2(L;\Z_2) = 0$ and $M$ is spherically
positive. Then for any Hamiltonian diffeomorphism $\phi: M \to M$,
we have
$$
L \cap \phi(L) \neq \emptyset.
$$
Moreover if $L$ is transversal to $\phi(L)$, there exists
$p \in L\cap \phi(L)$ whose Maslov index is $0$.
\end{theorem}
Theorem \ref{theoremL} is a $\Z_2$ coefficients  version of  Theorem L \cite{fooo-book1}.
\par
The main new part of this paper which consists of  Sections
\ref{technicalresult}-\ref{section35.4} is written in a way
independent of \cite{fooo-book1},\cite{fooo-book2}. Actually the  main result, Theorem
\ref{maintechnicalresult}, holds for an arbitrary space with
Kuranishi structure with tangent bundle, which may or may not be
related to the moduli space of pseudo-holomorphic curve. (We would
like to recall the readers that the theory of Kuranishi structures
itself that appeared in 1996 is not restricted to but independent of
the study of pseudo-holomorphic curves, although its main
application so far is aimed to the study of moduli spaces of
pseudo-holomorphic curves. )
\par
Sections \ref{moduli} and after are devoted to the
generalizations of the results of \cite{fooo-book1}, \cite{fooo-book2} to the $\Z$ or
$\Z_2$ coefficients. So they necessarily use  various results from
\cite{fooo-book1}, \cite{fooo-book2}. We, however, try to make them separately readable
without reading \cite{fooo-book1}, \cite{fooo-book2} modulo the details of the proofs, as much
as possible.
Especially Section \ref{AinftyconstrZ2}, which constructs a filtered
$A_{\infty}$ algebra over $\Z_2$, we repeat the construction of
compatible sections carried out in the transversality part (Section
7.2) of \cite{fooo-book2}. There we use the results on analysis and
homological algebra (construction of Kuranishi structure on moduli
space of holomorphic discs), which we refer readers to
\cite{fooo-book2}.
\par
This paper is one third of \cite{fooo-chap8}, that is  `Chapter 8' of the
2006 preprint version of \cite{fooo-book1}, \cite{fooo-book2}.
(`Chapter 8' is one of the two chapters which were removed
from \cite{fooo-book1}, \cite{fooo-book2} when it was finally published.
We removed them to keep the size of \cite{fooo-book1}, \cite{fooo-book2} within the
requirement of the publisher.)
During 5 years there are some progress. For example, the readers can see such progress in Subsections \ref{MClambda0}, \ref{Hochschild homology}.
\par
Another one third of  \cite{fooo-chap8} becomes \cite{fooo:inv}.
The paper containing the rest of  \cite{fooo-chap8}, which proves Arnold-Givental conjecture
in spherically positive case, is in preparation.

\section{Kuranishi structure: review}\label{kuranishisection}
In this section, we review the definition of Kuranishi structure.
There is nothing new in this section except Definitions
\ref{A1.21} and \ref{perturbed0kuranishi}, which we use in later
sections.
The main purpose of this section is a review and fixing notations. We refer \cite{FuOn99II} and
Section A1 of \cite{fooo-book2}.
Let $X$ be a compact metrizable space and $p \in X$.
\begin{defn}\label{A1.1}
A {\it Kuranishi neighborhood} of $p$ in $X$ is a quintet
$(V_p, E_p, \Gamma_p, \psi_p, s_p)$ such that:
\begin{enumerate}[{\rm (i)}]
\item $V_p$ is a
smooth manifold of finite dimension, which may or may
not have boundary or corner.
\item $E_p$ is a real vector space of finite dimension. \par
\item $\Gamma_p$ is a finite group acting
smoothly and effectively on $V_p$ and has a linear representation
on $E_p$.
\item
$s_p$ is a $\Gamma_p$ equivariant smooth map $V_p \to E_p$. \par
\item $\psi_p$ is a homeomorphism from
$s_p^{-1}(0)/\Gamma_p$ to a neighborhood of $p$ in $X$.
\end{enumerate}
We put $U_p = V_p/\Gamma_p$ and say that
$U_p$ is a {\it Kuranishi neighborhood}. We sometimes say that $V_p$ is a
Kuranishi neighborhood by an abuse of notation.
\par
We call $E_p \times V_p \to V_p$ the {\it obstruction bundle} and
$s_p$ the {\it Kuranishi map}.
For $x \in V_p$, denote by $I_x$ the isotropy subgroup at $x$, i.e.,
$$I_x = \{ \gamma \in \Gamma_p \vert \gamma x = x \}.$$
\end{defn}
Let us take a point $o_p \in V_p$ with $s_p(o_p) = 0$
and $\psi([o_p]) = p$. We may and will assume that $o_p$ is fixed by all elements of
$
\Gamma_p
$.
\begin{defn}\label{A1.3}
Let $(V_p, E_p, \Gamma_p, \psi_p, s_p)$,
$(V_q, E_q, \Gamma_q, \psi_q, s_q)$ be Kuranishi neighborhoods of
$p \in X$ and $q \in \psi_p(s_p^{-1}(0)/\Gamma_p)$, respectively.
We say a triple
$(\hat\phi_{pq},\phi_{pq},h_{pq})$ a {\it coordinate change}
if
\begin{enumerate}[{\rm (i)}]
\item
$h_{pq}$ is an injective homomorphism $\Gamma_q \to \Gamma_p$.
\par
\item $\phi_{pq}: V_{pq} \to V_p$ is an
$h_{pq}$ equivariant smooth embedding
from a $\Gamma_q$ invariant open neighborhood $V_{pq}$  of $o_q$ to $V_p$,
such that the induced map
$\underline{\phi}_{pq}:
V_{pq}/\Gamma_q \to V_p/\Gamma_p$ is injective.
\par
\item $(\hat\phi_{pq},\phi_{pq})$ is an $h_{pq}$
equivariant embedding of vector bundles $E_q  \times V_{pq} \to E_p \times V_p$.
\par
\item
$\hat\phi_{pq}\circ s_q = s_p\circ\phi_{pq}$. Here and hereafter we
sometimes regard $s_p$ as a section $s_p:  V_p \to E_p\times V_p$ of
trivial bundle $E_p \times V_p \to V_p$.
\par
\item
 $\psi_q =
\psi_p\circ \underline{\phi}_{pq}$ on $(s_q^{-1}(0) \cap V_{pq})/\Gamma_q$.
Here $\underline{\phi}_{pq}$ is as in (ii).
\item
$h_{pq}$ restricts to an isomorphism
$(\Gamma_q)_x \to (\Gamma_p)_{\phi_{pq}(x)}$ for any $x \in
V_{pq}$.
\end{enumerate}
\end{defn}
\begin{defn}\label{A1.5}
A {\it Kuranishi structure} on $X$ assigns a
Kuranishi neighborhood $(V_p, E_p, \Gamma_p, \psi_p, s_p)$
for each $p \in X$ and a coordinate change
$(\hat\phi_{pq},\phi_{pq},h_{pq})$ for each
$q \in \psi_p(s_p^{-1}(0)/\Gamma_p)$ such that the following holds.
\begin{enumerate}[{\rm (i)}]
\item $\dim V_p - \operatorname{rank} E_p$
is independent of $p$.
\par
\item If $r \in \psi_q((s_q^{-1}(0)\cap V_{pq})/\Gamma_q)$,
$q \in \psi_p(s_p^{-1}(0)/\Gamma_p)$, then
there exists $\gamma^{\alpha}_{pqr} \in
\Gamma_p$ for each connected component
$\alpha$ of $\phi_{qr}^{-1}(V_{pq}) \cap V_{qr} \cap V_{pr}$ such that
$$
h_{pq} \circ h_{qr} = \gamma^{\alpha}_{pqr}
\cdot h_{pr} \cdot (\gamma^{\alpha}_{pqr})^{-1}
, \quad
\phi_{pq} \circ \phi_{qr} = \gamma^{\alpha}_{pqr}
\cdot \phi_{pr}, \quad
\hat\phi_{pq} \circ \hat\phi_{qr} = \gamma^{\alpha}_{pqr}\cdot
\hat\phi_{pr}.
$$
Here the second equality holds on the
connected component $\alpha$ of $\phi_{qr}^{-1}(V_{pq}) \cap V_{qr} \cap V_{pr}$
and the third equality holds on the restriction of $E_r \times (\phi_{qr}^{-1}(V_{pq}) \cap V_{qr} \cap V_{pr})$ to $\alpha$.
In case $V_p$ has boundary or corners, we say that $(V_p, E_p, \Gamma_p, \psi_p,
s_p)$ is a {\it Kuranishi structure with boundary or corner}.
\end{enumerate}
We remark that (ii) is equivalent to the condition that
$$
\underline{\phi}_{pq} \circ \underline{\phi}_{qr} = \underline{\phi}_{pr}.
$$
(We can prove this equivalence by using the effectivity of the $\Gamma_p$
action.)
\par
We call $\dim V_p - \operatorname{rank} E_p$ the {\it virtual dimension}
(or dimension)
of the Kuranishi structure.
\par
In case $K \subset X$ we say $U$ is a {\it Kuranishi neighborhood} of $K$
if $U = \{V_{p_i}/\Gamma_{p_i}\}$,  $K \subset \bigcup_i
\psi_{p_i}(s_{p_i}^{-1}(0)/\Gamma_{p_i})$.
\par
An {\it orbifold} structure on $X$ is, by definition, a Kuranishi structure
on $X$ such that $E_p = 0$ for all $p$.
\end{defn}
We recall:
\begin{lem}\label{A1.11}{\rm([\cite{FuOn99II}, Lemma 6.3])}
Let $X$ be a space with Kuranishi structure.
Then there exists a finite set $P \subset X$, a partial order $<$ on $P$,
and a Kuranishi neighborhood $(V_p, E_p, \Gamma_p, \psi_p, s_p)$
of $p$ for each $p \in P$, with the following properties.
\par\smallskip
\begin{enumerate}[{\rm (i)}]
\noindent
\item
If $q < p$, $\psi_p(s_p^{-1}(0)/\Gamma_p) \cap
\psi_q(s_q^{-1}(0)/\Gamma_q) \ne \emptyset$, then there exists
$(V_{pq},\hat\phi_{pq},\phi_{pq},h_{pq})$ where:
\smallskip
\begin{enumerate}[{\rm(a)}]
\item
$V_{pq}$ is a $\Gamma_q$ invariant open subset of $V_q$ such that
$V_{pq}/\Gamma_q$ contains
$$\psi_q^{-1}(\psi_p(s_p^{-1}(0)/\Gamma_p) \cap \psi_q(s_q^{-1}(0)/\Gamma_q)),
$$
\item
$h_{pq}$ is an injective homomorphism $\Gamma_q \to \Gamma_p$ with
its image $(\Gamma_p)_{\phi_{pq}(q)}$,
\par
\item
$\phi_{pq}: V_{pq} \to V_p$ is an $h_{pq}$ equivariant smooth
embedding such that the induced map $V_{pq}/\Gamma_{q} \to
V_p/\Gamma_p$ is injective,
\par
\item
$(\hat\phi_{pq},\phi_{pq})$ is an $h_{pq}$
equivariant embedding of vector bundles $E_q\times{V_{pq}} \to E_p\times V_{p}$,
\par
\item
$\hat\phi_{pq}\circ s_q =
s_p\circ\phi_{pq}, \qquad \psi_q = \psi_p\circ\underline{\phi}_{pq}$.
\end{enumerate}
\smallskip
\item
If $r < q < p$,
$\psi_p(s_p^{-1}(0)/\Gamma_p) \cap \psi_q(s_q^{-1}(0)/\Gamma_q)
\cap \psi_r(s_r^{-1}(0)/\Gamma_r) \ne \emptyset$, then there exists $\gamma^{\alpha}_{pqr} \in \Gamma_p$
for each connected component $\alpha$ of
$\phi_{qr}^{-1}(V_{pq}) \cap V_{qr}
\cap V_{pr}$
such that
$$
h_{pq} \circ h_{qr} =
\gamma^{\alpha}_{pqr} \cdot
h_{pr}  \cdot
(\gamma^{\alpha}_{pqr})^{-1}
, \quad
\phi_{pq} \circ \phi_{qr} = \gamma^{\alpha}_{pqr} \cdot
\phi_{pr}, \quad
\hat\phi_{pq} \circ \hat\phi_{qr} = \gamma^{\alpha}_{pqr} \cdot
\hat\phi_{pr}.
$$
Here the second equality holds on the connected component $\alpha$ of $\phi_{qr}^{-1}(V_{pq}) \cap V_{qr}
\cap V_{pr}$, and the
third equality holds on the restriction to $\alpha$ of
$E_r\times({\phi_{qr}^{-1}(V_{pq}) \cap V_{qr}}\cap V_{pr})$.
\par\item
We have
$\displaystyle
\bigcup_{p\in P} \psi_p(s_p^{-1}(0)/\Gamma_p) = X.
$
\par\item
If
$\psi_p(s_p^{-1}(0)/\Gamma_p) \cap \psi_q(s_q^{-1}(0)/\Gamma_q) \ne \emptyset$,
then either $p<q$ or $q<p$.
\end{enumerate}
\end{lem}
We call the system of Kuranishi neighborhood and
coordinate change
$$
(\{(V_p, E_p, \Gamma_p, \psi_p, s_p) \mid p \in P\},
~
\{(V_{pq},\hat\phi_{pq},\phi_{pq},h_{pq})\mid p,q \in P, p<q\})
$$
 in Lemma \ref{A1.11}
{\it the good coordinate system}.
(Note we require {\it existence} of $\gamma_{pqr}^{\alpha}$ but do not
include it as a part of the structure. Compare Remark \ref{A1.81} (iv).)
\begin{defn}\label{A1.13}
Consider the situation of Lemma \ref{A1.11}. Let $Y$ be a
topological space. A family $\{f_p\}$ of $\Gamma_p$-equivariant
continuous maps $f_p: V_p \to Y$ is said to be a {\it strongly
continuous map} if
$$
f_p \circ \phi_{pq} = f_q
$$
on $V_{pq}$.  A strongly continuous map induces a continuous map $f: X \to
Y$.
We will ambiguously denote $f = \{f_p\}$ when the meaning is clear.
\par
When $Y$ is a smooth manifold, a strongly continuous map
$f: X \to Y$ is defined to be {\it smooth} if all $f_p:V_p \to Y$ are smooth.
We say that it is {\it weakly submersive} if each of $f_p$ is a submersion.
\par
When $Y$ is a simplicial complex, we say $f = \{f_p\}$ is  {\it strongly piecewise smooth}
if each of $f_p$ is  piecewise smooth.
\end{defn}
Consider the situation of Lemma \ref{A1.11}. We identify a neighborhood of
$\phi_{pq}(V_{pq})$ in $V_p$ with a neighborhood of the zero section
of the normal bundle $N_{V_{pq}}V_p \to V_{pq}$, using an
exponential map of an appropriate Riemannian metric. We take the
fiber derivative of the Kuranishi map $s_p$ along the fiber
direction and obtain a homomorphism
\begin{equation}\label{fiberdifferential}
d_{\text{fiber}}s_p: N_{V_{pq}}V_p \to E_p \times V_{pq}
\end{equation}
which is an $h_{pq}$-equivariant bundle homomorphism.
Note we need to take and fix a connection on the bundle $E_p$
to define (\ref{fiberdifferential}).
However (\ref{fiberdifferential}) is independent of the choice of the
connection on the zero set of $s_p$.
\begin{defn}\label{A1.14}
We say that the space with Kuranishi structure
$X$ {\it has a tangent bundle} if $d_{\text{fiber}}s_p$
induces a bundle isomorphism
\begin{equation}
N_{V_{pq}}V_p \cong \frac{E_p \times V_{pq}}{\hat{\phi}_{pq}(E_q\times {V_{pq})}}
\label{A1.15}\end{equation}
as $\Gamma_q$-equivariant bundles on $V_{pq} \cap s_p^{-1}(0)$.
\end{defn}
By definition, the following diagram commutes for
each $x \in {\phi}_{qr}^{-1}(V_{pq}) \cap V_{qr} \cap V_{pr}$.
\begin{equation}
\CD
(N_{V_{qr}}V_q)_x
@>{\gamma_{pqr}\cdot d\phi_{pq}}>> (N_{V_{pr}}V_p)_x
@> >> (N_{V_{pq}}V_p)_{\phi_{qr}(x)} \\
 @V{}VV  @V VV @V VV\\
\frac{E_q}{\hat{\phi}_{qr,x}(E_r)}
@>{\gamma_{pqr} \cdot \hat\phi_{pq}}>> \frac{E_p}{\hat{\phi}_{pr,x}(E_r)}
@> >> \frac{E_p}{\hat{\phi}_{pq,\phi_{qr}(x)}(E_q)}
\endCD
\label{A1.16}\end{equation} Here and hereafter $\hat{\phi}_{qp,x}:
E_q \to E_p$ is the restriction of the bundle map $\hat{\phi}_{pq}$
to the fiber of $x$.
And we take and fix a connection on the bundle $E_p$ so that the image of
$\hat\phi_{pq}$ is invariant under the parallel transport.
Then (\ref{fiberdifferential}) is well defined and (5) commutes.
\par
We may shrink $V_{pq}$ to a smaller subset if necessary so that
(\ref{A1.15}) is an isomorphism not only on $V_{pq} \cap s_p^{-1}(0)$
but also everywhere on $V_{pq}$.
So hereafter we always assume so.

\begin{defn}\label{A1.17}
Let $X$ be a space with Kuranishi structure which has a tangent bundle.
We say that the Kuranishi structure on $X$ is {\it oriented} if
we have a trivialization of
$$
\Lambda^{\text{\rm top}}E^*_p \otimes \Lambda^{\text{\rm top}}TV_p
$$
which is compatible with isomorphism (\ref{A1.15}).
\end{defn}
We next define the compatibility of sections of the obstruction bundles
over various Kuranishi charts as follows.
\begin{defn}\label{A1.21}
Let us consider the situation of Lemma \ref{A1.11}. We assume that our
Kuranishi structure has a tangent bundle. Suppose $p,q \in P$, $q<p$
and  suppose we have sections $s'_p$, $s'_q$ of $E_p\times
V_p$,  $E_q\times V_q$ respectively.
We consider the embedding $\phi_{pq}: V_{pq} \to V_p$ in Lemma
\ref{A1.11} (i) (c). We identify its normal bundle $N_{V_{pq}}V_p$
with a tubular neighborhood of  $\phi_{pq}(V_{pq})$. For each $x \in
V_{pq}$ we fix a splitting
\begin{equation}
E_p\cong \hat{\phi}_{pq,x}(E_q) \oplus \frac{E_p}{\hat{\phi}_{pq,x}(E_q)}.
\label{A1.22}\end{equation}
For each $y \in N_{V_{pq}}V_p$, we obtain an element $1(y)$ of
$\frac{E_p}{\hat{\phi}_{pq,\pi(y)}(E_q)}$ by using the isomorphism (\ref{A1.15}).
Then, on $N_{V_{pq}}V_p$, we define a section $s'_q \oplus 1$
by
$$\aligned
(s'_q \oplus 1)(y) = s'_{q}(\pi(y))\oplus
1(y) \in \hat{\phi}_{pq,x}(E_q) \oplus \frac{E_p}{\hat{\phi}_{pq,\pi(y)}(E_q)}
\cong E_p^n .
\endaligned$$
Here we use the splitting (\ref{A1.22}). Now we say that $s'_p$ is {\it compatible} with $s'_q$ if the restriction of
$s'_p$ to $N_{V_{pq}}V_p$ coincides with $s'_q \oplus 1$.
\par
We note that the section $s_p$ (that is a
Kuranishi map) is compatible with $s_q$ in the sense defined above.
This follows from the fact that (\ref{A1.15}) is induced by the fiber
derivative of the Kuranishi map.
\par
A {\it global section} $\mathfrak s' = \{s'_p\}_{p\in P}$ of
the obstruction bundle of a Kuranishi structure is a compatible
system of sections $(s'_p: V_p \to E_p)$.
A section $\mathfrak s' = \{s'_p\}_{p\in P}$ of
the obstruction bundle of a Kuranishi structure is also called
a {\it global section} if it is a compatible system of sections
$(s'_p: V'_p \to E_p)$ after shrinking Kuranishi neighborhood
$V_p$ to a relatively compact subspace $V'_p \subset V_p$ appropriately. 
Actually in the following definition we shrink Kuranishi neighborhoods and 
consider the zero set of $s'_p$ in the closure $\overline{V}'_p$ to construct a topological space $X'$.
\end{defn}
\begin{defn}\label{perturbed0kuranishi}
In case $\mathfrak s' = \{s'_p\}_{p\in P}$ is a global
section we define a topological space
$$
X' = (\mathfrak s')^{-1}(0)
$$
as follows.
We consider the disjoint union
\begin{equation}\label{inveresunion}
\bigcup_{p\in P} \frac{(s'_p)^{-1}(0)}{\Gamma_p} \times \{p\}.
\end{equation}
Let $\tilde x \in (s'_p)^{-1}(0) \subset \overline{V}'_p$ and
$\tilde y \in (s'_q)^{-1}(0) \subset \overline{V}'_q$.
We define $(x,p) \sim' (y,q) $ if
$q<p$, $x = [\tilde x]$,  $y= [\tilde y]$, $\tilde y \in V_{pq}$ and
$$
\phi_{pq}(\tilde y) = \tilde x.
$$
Let $\sim$ be the equivalence relation on the set
(\ref{inveresunion}) which is generated by $\sim'$. We define $X'$ as the set of
equivalence classes
of this equivalence relation.
We put quotient topology on it.
We have a Kuranishi structure on $X'$ whose Kuranishi map is $s'_p$
in an obvious way.
\par
If $f$ is a strongly continuous map from $X$ to $Y$,
it induces a strongly continuous map from $X'$ to $Y$.
The case of strongly smooth map or strongly piecewise smooth map
is similar.
\end{defn}
\section{Statement of the main technical result}\label{technicalresult}
Let $X$ be a space with Kuranishi structure. We fix its good coordinate system
$(\{(V_p, E_p, \Gamma_p, \psi_p, s_p) \mid p \in P\},\{(V_{pq},\hat\phi_{pq},\phi_{pq},h_{pq})\mid p,q \in P, p<q\})$.
\begin{defn}\label{defisotropy}
Let $x \in X$. We take $p \in P$ and $\tilde x \in V_p$ such that
$s_p(\tilde x) = 0$ and $\psi_p(\tilde x) = x$. We put
\begin{equation}\label{eqisotropy}
I_x = \{ \gamma \in \Gamma_p \mid \gamma \tilde x = \tilde x\},
\end{equation}
and call it the {\it isotropy group}.
It is easy to see that $I_x$ is independent of the choice of $p$ and $\tilde x$ and
depends only on $x$.
\par
For a finite group $\Gamma$,
we define
\begin{equation}\label{Gammasratum}
X^{\cong}(\Gamma) = \{ x \in X \mid I_x \cong \Gamma\}.
\end{equation}
\end{defn}
We decompse $X^{\cong}(\Gamma)$ into connected components
\begin{equation}
X^{\cong}(\Gamma) = \bigcup_i X^{\cong}(\Gamma;i).
\end{equation}
\begin{defn}\label{def:dXgamma}
We define the integers $d(X;\Gamma;i)$ as follows.
Let $x \in X^{\cong}(\Gamma;i)$.
We take $p$, $\tilde x$ as in Definition \ref{defisotropy}. The group
$I_x$ acts on $(E_p)_{\tilde x}$, the fiber of the obstruction bundle at $\tilde x$.
We then put
\begin{equation}\label{35.18pre}
E_{\tilde x}^{\Gamma} = \{v \in E_{\tilde x} \mid \forall \gamma \in \Gamma \,\,
\gamma v = v\}.
\end{equation}
Its dimension depends only on
$\Gamma,i$ but independent of $p$, $\tilde x$, $x$.
The group $\Gamma$ acts also on the tangent bundle
$T_{\tilde x}V_p$.
We put
$$
T_{\tilde x}V_p^{\Gamma} = \{T_{\tilde x}V_p \mid \forall \gamma \in \Gamma \,\,
\gamma v = v\}.
$$
Its dimension depends only on
$\Gamma,i$ but independent of $p$, $\tilde x$, $x$.
We now define
\begin{equation}
d(X;\Gamma;i) = \dim T_{\tilde x}V_p^{\Gamma} - \dim E_{\tilde{x}}^{\Gamma}.
\label{35.19}\end{equation}
\end{defn}
Now the main technical result used in the proof of Theorems
\ref{theoremA} and \ref{theoremF} is the following:
\begin{theorem}\label{maintechnicalresult}
Let $X$ be a space with Kuranishi structure that has a tangent bundle.
We take a good coordinate system.
Then there exists a strongly piecewise smooth global section $\mathfrak s'$, which is
arbitrarily close to the original Kuranishi map in $C^0$ sense, such that the following holds for
$X' = (\mathfrak s')^{-1}(0)$
defined in Definition \ref{perturbed0kuranishi} and its Kuranishi structure:
\begin{enumerate}[{\rm (i)}]
\item $X'$ has a triangulation.
\item Each of $X^{\prime \cong}(\Gamma)$ is a simplicial subcomplex of $X'$.
\item We have
$$
\dim X^{\prime \cong}(\Gamma) \le \max_i d(X;\Gamma;i).
$$
\end{enumerate}
Moreover if $f: X \to Y$ is a strongly smooth map from $X$ to a
manifold $Y$, we may choose the triangulation of $X'$ so that
$f$ induces a piecewise smooth map $X' \to Y$.
\end{theorem}
We note that if
$s'_p$ is a (single valued) section on $V_p/\Gamma_p$, then
at $\tilde x\in V_p$ with $I_{\tilde x} =\Gamma$, the
value $s'_p(\tilde x)$ is necessarily in the $\Gamma = \Gamma_{\tilde x}$
fixed point set of $(E_{p})_{\tilde x}$.
This implies that the dimension appearing in the right hand side of
(iii) above is optimal.
\par
In the situation where the space $X$ with Kuranishi structure is a moduli space of geometric origin,
the number $d(X;\Gamma;i)$ can be identified with an appropriate
equivariant index.
Thus we can estimate this number under certain
geometric assumption.
In case
$d(X;\Gamma;i)$ is strictly smaller than the virtual dimension of $X$ minus 1,
Theorem \ref{maintechnicalresult} implies that $X$ has a
virtual fundamental chain over $\Z$ or over $\Z_{2}$.
(To define virtual fundamental chain over $\Z$
we need to include orientation.)
This is the way, we prove Theorems \ref{theoremA}, \ref{theoremF}.
See Section \ref{AinftyconstrZ2}  for more detail.
\par
The proof of Theorem  \ref{maintechnicalresult} occupies
Sections \ref{sectionA1.6.1}-\ref{section35.4}.
A brief outline of the proof is in order.
\par
We use the theory of stratified sets for our  proof.
The collection of
$X(\Gamma;i)$'s for various $\Gamma$, $i$
defines a
natural stratification of $X$ and
the stratification extends to the Kuranishi neighborhood.
We construct our global section by an induction on the Kuranishi neighborhood.
We apply induction twice. For the first step we work on one
Kuranishi neighborhood and then we use the partial order $<$ of $P$
to inductively construct $s'_p$.
\par
When we restrict ourselves to a single Kuranishi neighborhood,
we deal with an orbifold and an orbi-bundle.
We use the stratification of an orbifold by its isotropy group
to construct $s'_p$ on each of the Kuranishi chart.
We use the normal bundle of the stratum and a conical
extension of the section on one stratum to its neighborhood.
In Section \ref{normalstack}, we study normal bundle of a stratum.
As we will observe in Example \ref{A1.68} Section \ref{sectionA1.6.1}, the stratum
(the set of $x$ with given $I_x$) does not have the normal bundle
in a usual sense. Actually it has some twisted stack structure and
the normal bundle is well-defined as a bundle over this stack.
We will explain this point in Sections \ref{sectionA1.6.1}-\ref{normalstack}.
\par
To construct the section $s'_p$ inductively we need to use
compatibility between the normal bundles of various components. The
story of compatible systems of normal bundles is classical and was
used in the study of triangulation of algebraic sets. Especially J.
Mather \cite{Math73} gave a nice definition of compatibility which
we use with minor modification. (See Section \ref{section35.2}.) We
use this in our construction of $s'_p$ and triangulation of its zero
set applied to each Kuranishi neighborhood in Section
\ref{section35.3}.
\par
The inductive construction over various Kuranishi charts is actually
standard using a good coordinate system.
We perform this construction in Section \ref{section35.4}, where we complete the proof of
Theorem \ref{maintechnicalresult}.

\section{Sheaves of group-category}
\label{sectionA1.6.1}

Let $X$ be an orbifold and $\Gamma$ a finite group. We define
$$
X^{\cong}(\Gamma) = \{ x \in X \mid I_x \cong \Gamma\}.
$$
In this paper, we need to consider a
normal bundle $N_{X^{\cong}(\Gamma)}X$ of $X^{\cong}(\Gamma)$ in $X$.
At first sight, one might expect that there
exists a vector bundle $N_{X^{\cong}(\Gamma)}X$ over the topological space
$X^{\cong}(\Gamma)$ together with a $\Gamma$ action on $N_{X^{\cong}(\Gamma)}X$
such that $N_{X^{\cong}(\Gamma)}X/\Gamma$ is
diffeomorphic to a neighborhood of $X^{\cong}(\Gamma)$ in $X$.
However such a vector bundle $N_{X^{\cong}(\Gamma)}X$
does not exist in general. In fact, we have the following counter example.

\begin{exm}\label{A1.68}
We consider $\C^n \times S^1$ with $\mathbb Z_p$ action defined by
\begin{equation}
[k] \cdot (z,[t]) = (\exp(2\pi\sqrt{-1}k/p)z,[t]).
\label{A1.69}\end{equation}
Here $[k] \in \Z \mod p$, $[t] \in \R/\Z = S^1$.
\par
We define an isomorphism
$$
F: (\C^n \times S^1)/\Z_p \to (\C^n \times S^1)/\Z_p
$$
by
$$
F([z,[t]]) = [\exp(2\pi\sqrt{-1}t/p)z,[t]].
$$
We take two copies $(\C^n \times D_{\pm}^2)/\Z_p$ of
$(\C^n \times  D^2)/\Z_p$ where $\Z_p$ action is similar to
(\ref{A1.69}). We identify
$$
(\C^n \times S^1)/\Z_p = (\C^n \times \partial D_{+}^2)/\Z_p
$$
with
$$
(\C^n \times S^1)/\Z_p = (\C^n \times \partial D_{-}^2)/\Z_p
$$
by $F$ and obtain an orbifold $X$.
\par
In this example $X^{\cong}(\mathbb Z_p) = S^2$.
The normal bundle $N_{X^{\cong}(\Z_p)}X$
does not exist since $F$ does not lift to a bundle isomorphism
$: \C^n \times  S^1 \to \C^n \times  S^1$.
\end{exm}

As Example \ref{A1.68} shows, the normal bundle of the singular
locus $X^{\cong}(\Gamma)$ does not exist in general as a global
quotient of a vector bundle with a $\Gamma$ action.  On the other
hand, in this paper we need to use the normal bundle of
$X^{\cong}(\Gamma)$ to define normally conical perturbations. For
this purpose we define the notion of a {\it normal bundle in the sense of
stack}. We restrict our discussion of the stack to the case we use
for this purpose. Related material is discussed in various
references such as \cite{Bry93}, \cite{Gir70}. The discussion here
is related to the phenomenon that occurs when we remove the
effectivity of the $\Gamma_p$ action from the definition of
orbifold. We feel that the results of Sections
\ref{sectionA1.6.1}-\ref{normalstack} are not really new. However it
is hard to find a reference that contains the results written in a
way we want to use. Also most of the references on the stack are
written in a very abstract way. To minimize our usage of many
abstract languages entering in the definition of stack-like objects,
we prefer to use more down-to-earth approach by explicitly writing
down all the formulas that we really need. This is the reason why we
include the materials in this paper.
\par
Let $G$ be a group. We consider the category $\underline G$ which
has only one object $*$ and morphism $\underline G(*,*) = G$. Let
$M$ be a topological space and $\mathcal U = \{U_i \mid i \in I\}$ be an
open covering of $M$. We assume that $U_{i_1} \cap \cdots \cap
U_{i_k}$ ($i_1,\dots,i_k \in I$) are either empty or contractible.
Namely we take a good covering.
\begin{defn}\label{A1.70}
\begin{enumerate}[{\rm (i)}]
\item A {\it sheaf of category of $\underline G$ on $(M,\mathcal U)$} consists of
pair
$(\{h_{ij}\}, \{\gamma_{ijk}\})$
of an isomorphism
$$
h_{ij} \in {\rm Aut}(G) \qquad \text{for each $U_i \cap U_j \ne \emptyset$}
$$
and an element
$$
\gamma_{ijk} \in G \qquad \text{for each $U_i \cap U_j \cap U_k \ne \emptyset$ }
$$
such that the following compatibility conditions (\ref{A1.71.1}) and
(\ref{A1.71.2}) hold for every $U_i \cap U_j \cap U_k \cap U_l \ne
\emptyset$:
\begin{subequations}\label{A1.71}
\begin{equation}
h_{ij} \circ h_{jk} = \gamma_{ijk} \cdot h_{ik} \cdot \gamma_{ijk}^{-1},
\label{A1.71.1}\end{equation}
\begin{equation}
\gamma_{ijk} \cdot \gamma_{ikl} = h_{ij}(\gamma_{jkl}) \cdot \gamma_{ijl}.
\label{A1.71.2}\end{equation}
\end{subequations}
Here the right hand side of
\eqref{A1.71.1} means
$\operatorname{ad} (\gamma_{ijk})\circ h_{ik}$.
(We will explain how (\ref{A1.71}) follows from
the definition of a stack in the categorical context in Remark \ref{A1.81} (v).)
\item
$(\{h_{ij}\},\{\gamma_{ijk}\})$ is said to be {\it isomorphic} to
$(\{h'_{ij}\},\{\gamma'_{ijk}\})$ if there exist
$$
\psi_i \in \text{Aut}(G), \qquad \mu_{ij} \in G
$$
for each $i$ and $U_i \cap U_j \ne \emptyset$ respectively, and
$$
h_{ij}^{\prime\prime} \in \text{Aut}(G), \qquad \gamma^{\prime\prime}_{ijk} \in G
$$
for each $i$, $U_i \cap U_j \ne \emptyset$, and  $U_i \cap U_j \cap U_k \ne \emptyset$
respectively such that
\begin{subequations}\label{A1.72}
\begin{eqnarray}
h_{ij}^{\prime\prime} & = &\psi_i\circ h_{ij} \circ \psi_j^{-1},
\label{A1.72.1}
\\
\gamma_{ijk}^{\prime\prime} &=& \psi_i(\gamma_{ijk}),
\label{A1.72.2}
\\
\mu_{ij}  \cdot h_{ij}^{\prime\prime} \cdot \mu_{ij}^{-1}&=& h'_{ij},
\label{A1.72.3}\\
\mu_{ij}  \cdot h_{ij}^{\prime\prime}(\mu_{jk})
\cdot \gamma^{\prime\prime}_{ijk} &=& \gamma'_{ijk}\cdot \mu_{ik}.
\label{A1.72.4}
\end{eqnarray}
\end{subequations}
We call a pair $(\{\mu_{ij}\},\{\psi_i\})$ an {\it isomorphism}
$: (\{h_{ij}\},\{\gamma_{ijk}\}) \to (\{h'_{ij}\},\{\gamma'_{ijk}\})$.
We will prove in Lemma \ref{A1.75} that we can compose isomorphism and `isomorphic'
defines an equivalence relation.
\item We denote by $Sh((M,\mathcal U);\underline G)$ the set of all isomorphism
classes of  sheaves of categories of $\underline G$ on $(M,\mathcal U)$.
\end{enumerate}\end{defn}
\par\medskip
To illustrate the meaning of (\ref{A1.71}) we show the following:
\begin{lem}\label{A1.73}
Let $\{Y_i\}_i$ be a collection of sets such that a group $G$ acts
effectively on each of $Y_i$.  Suppose that there is a point with trivial isotropy group
on each of $Y_i$. Let
$h_{ij}: G \to G$ be group isomorphisms, $\gamma_{ijk}$ elements of
$G$, and let $\phi_{ij}: Y_j \to Y_i$ be maps that are injective and
$h_{ij}$-equivariant. We assume
\begin{equation}
\phi_{ij} \circ \phi_{jk} = \gamma_{ijk}\cdot \phi_{ik}.
\label{A1.74}\end{equation}
Then $\gamma_{ijk}$ satisfies {\rm (\ref{A1.71})}.
\end{lem}
\begin{proof}
Let $g \in G$, and $y\in Y_k$ with trivial isotropy group. Then
$\phi_{ik}(y)$ also has a trivial isotropy group because $\phi_{ik}$ are
assumed to be injective and $h_{ik}$-equivariant.
\par
Again by the $h_{ij}$-equivariance of the map $\phi_{ik}$, we obtain
$$\aligned
\gamma_{ijk} \cdot h_{ik}(g) \cdot \phi_{ik}(y) &= \gamma_{ijk} \cdot \phi_{ik}(g\cdot y)
= \phi_{ij} (\phi_{jk}(g\cdot y)) \\
&= h_{ij}(h_{jk}(g)) \cdot \phi_{ij} (\phi_{jk}(y)) =
h_{ij}(h_{jk}(g)) \cdot \gamma_{ijk}\cdot \phi_{ik}(y).
\endaligned
$$
Because $\phi_{ik}(y)$ has a trivial isotropy group, (\ref{A1.71.1}) follows.
\par
Similarly for $y \in Y_l$ we calculate
$$\aligned
h_{ij}(\gamma_{jkl})\cdot \gamma_{ijl}\cdot \phi_{il}(y)
&= h_{ij}(\gamma_{jkl})\cdot\phi_{ij} (\phi_{jl}(y))
= \phi_{ij}(\gamma_{jkl}\cdot \phi_{jl}(y)) \\
&= \phi_{ij}(\phi_{jk}(\phi_{kl}(y)))
= \gamma_{ijk}\cdot \phi_{ik}(\phi_{kl}(y))
= \gamma_{ijk}\cdot\gamma_{ikl}\cdot\phi_{il}(y).
\endaligned$$
This implies (\ref{A1.71.2}).
\end{proof}
In the definition of Kuranishi structure the group $\Gamma_p$ at each point $p$
is assumed to be a finite group and the space $V_p$ is a smooth manifold.
One can show that effectivity of the action of $\Gamma_p$ automatically implies
existence of a point with trivial isotropy group. And we also assume that the
map $\phi_{pq}$ is an $h_{pq}$-equivariant embedding and in particular injective.
Therefore the same argument used in the proof of Lemma \ref{A1.73},
implies that $\gamma_{pqr}$ in Definition \ref{A1.5} satisfies (\ref{A1.71.2}).
\begin{lem}\label{A1.75}
The relation `isomorphism' in Definition {\rm \ref{A1.70}} is an equivalence
relation.
\end{lem}
\begin{proof}
We use notation of (\ref{A1.72}) and put
$$
(h^{\prime\prime}_{ij},\gamma^{\prime\prime}_{ijk}) = (1,\psi_i)_* (h_{ij},\gamma_{ijk}),
\quad
(h^{\prime}_{ij},\gamma^{\prime}_{ijk}) = (\mu_{ij},1)_*
(h^{\prime\prime}_{ij},\gamma^{\prime\prime}_{ijk}).
$$
We also put
$(\mu_{ij},1)_*\circ (1,\psi_i)_*=(\mu_{ij},\psi_i)_*$.
We note that
\begin{equation}
(1,\psi_i)_* \circ (1,\psi'_i)_* = (1,\psi_i\circ \psi'_i)_*.
\label{A1.76}\end{equation}
We next claim
\begin{equation}
(\mu_{ij},1)_*\circ  (\mu'_{ij},1)_* = (\mu_{ij}\cdot \mu'_{ij},1)_*.
\label{A1.77}\end{equation}
Let us prove (\ref{A1.77}). We put
$$
(\mu'_{ij},1)_*(h_{ij}^1,\gamma_{ijk}^1) = (h_{ij}^2,\gamma_{ijk}^2),
\quad
(\mu_{ij},1)_*(h_{ij}^2,\gamma_{ijk}^2) = (h_{ij}^3,\gamma_{ijk}^3).
$$
Then
$$\aligned
&h_{ij}^2 = \mu'_{ij} \cdot h_{ij}^1 \cdot (\mu'_{ij})^{-1}, \\
&\gamma^2_{ijk} = \mu'_{ij} \cdot h_{ij}^1(\mu'_{jk}) \cdot \gamma^1_{ijk} \cdot  (\mu'_{ik})^{-1}.
\endaligned$$
Therefore we have
$$
h_{ij}^3 = \mu_{ij} \cdot \mu'_{ij} \cdot h_{ij}^1 \cdot (\mu'_{ij})^{-1} \cdot \mu_{ij}^{-1}
$$
and
$$\aligned
\gamma^3_{ijk} &=  \mu_{ij} \cdot h_{ij}^2(\mu_{jk}) \cdot \gamma^2_{ijk} \cdot  \mu_{ik}^{-1} \\
&= \mu_{ij} \cdot  \mu'_{ij} \cdot  h_{ij}^1(\mu_{jk}) \cdot (\mu'_{ij})^{-1} \cdot \mu'_{ij}
\cdot h_{ij}^1(\mu'_{jk}) \cdot \gamma^1_{ijk} \cdot (\mu'_{ik})^{-1} \cdot \mu_{ik}^{-1} \\
&= \mu_{ij} \cdot  \mu'_{ij} \cdot h_{ij}^1(\mu_{jk} \cdot \mu'_{jk}) \cdot \gamma^1_{ijk} \cdot
( \mu_{ik} \cdot \mu'_{ik})^{-1}.
\endaligned$$
(\ref{A1.77}) is proved.
\par
We next claim
\begin{equation}
(1,\psi_i)_* \circ (\mu_{ij},1)_* = (\psi_i(\mu_{ij}),\psi_i)_* .
\label{A1.78}\end{equation}
Let us prove (\ref{A1.78}). We put
$$
(\mu_{ij},1)_*(h_{ij}^1,\gamma_{ijk}^1) = (h_{ij}^2,\gamma_{ijk}^2),
\quad
(1,\psi_i)_*(h_{ij}^2,\gamma_{ijk}^2) = (h_{ij}^3,\gamma_{ijk}^3).
$$
Then
$$\aligned
&h_{ij}^2 = \mu_{ij} \cdot h_{ij}^1 \cdot \mu_{ij}^{-1}, \\
&\gamma^2_{ijk} = \mu_{ij} \cdot h_{ij}^1(\mu_{jk}) \cdot \gamma^1_{ijk} \cdot  (\mu_{ik})^{-1}.
\endaligned$$
Therefore we have
$$
h_{ij}^3 = \psi_{i} \circ (\mu_{ij} \cdot h_{ij}^1 \cdot \mu_{ij}^{-1}) \circ \psi_{j}^{-1}
= \psi_i(\mu_{ij}) \cdot (\psi_{i} \circ h_{ij}^1  \circ \psi_{j}^{-1}) \cdot \psi_i(\mu_{ij})^{-1}
$$
and
$$
\gamma^3_{ijk} = \psi_i(\mu_{ij}) \cdot \psi_i(h_{ij}^1(\mu_{jk}))
\cdot \psi_i(\gamma^1_{ijk}) \cdot \psi_i(\mu_{ik})^{-1}.
$$
We next put
$$
(1,\psi_i)_*(h_{ij}^1,\gamma_{ijk}^1) = (h_{ij}^4,\gamma_{ijk}^4),
\quad
(\psi_i(\mu_{ij}),1)_*(h_{ij}^4,\gamma_{ijk}^4) = (h_{ij}^5,\gamma_{ijk}^5).
$$
Then
$$
h^4_{ij} = \psi_i \circ h^1_{ij} \circ \psi_j^{-1}, \quad
\gamma^4_{ijk} = \psi_i(\gamma^1_{ijk}).
$$
Therefore
$h_{ij}^5 = h_{ij}^3$ and
$$\aligned
\gamma_{ijk}^5 &= \psi_i(\mu_{ij}) \cdot h_{ij}^4(\psi_j(\mu_{jk})) \cdot \gamma^4_{ijk}
\cdot \psi_i(\mu_{ik})^{-1} \\
&= \psi_i(\mu_{ij}) \cdot \psi_i(h_{ij}^1(\mu_{jk})) \cdot  \psi_i(\gamma^1_{ijk})
\cdot \psi_i(\mu_{ik})^{-1}
= \gamma_{ijk}^3.
\endaligned$$
(\ref{A1.78}) is proved.
\par
(\ref{A1.76}), (\ref{A1.77}) and (\ref{A1.78}) imply that we can compose isomorphisms.
Hence the relation `isomorphic' is transitive.
\par
On the other hand, (\ref{A1.76}) and (\ref{A1.77}) imply that each
isomorphism has an inverse. Hence the relation `isomorphic' is symmetric.
\end{proof}
\begin{rem}\label{1.79}
Suppose $(h_{ij},\gamma_{ijk})$ satisfies (\ref{A1.71}). If we define
$(h'_{ij},\gamma'_{ijk})$ by (\ref{A1.72}), we can check
that $(h'_{ij},\gamma'_{ijk})$ satisfies (\ref{A1.71}), in a
similar way as the above calculation.
For example we consider the case $\psi_i =1$ and check
(\ref{A1.71.2}) as follows.
We have
$$\aligned
&\gamma'_{ijk} \cdot \gamma'_{ikl} \\
&= \mu_{ij} \cdot h_{ij}(\mu_{jk}) \cdot
\gamma_{ijk} \cdot  h_{ik}(\mu_{kl})
\cdot \gamma_{ikl} \cdot \mu_{il}^{-1} \\
&= \mu_{ij} \cdot h_{ij}(\mu_{jk}) \cdot
\gamma_{ijk} \cdot  h_{ik}(\mu_{kl}) \cdot
\gamma_{ijk}^{-1} \cdot h_{ij}(\gamma_{jkl}) \cdot \gamma_{ijl} \cdot \mu_{il}^{-1}.
\endaligned$$
On the other hand, we have
$$
\aligned
&h'_{ij}(\gamma'_{jkl}) \cdot \gamma'_{ijl} \\
&= \mu_{ij} \cdot h_{ij}(\mu_{jk} \cdot
h_{jk}(\mu_{kl}) \cdot \gamma_{jkl} \cdot \mu_{jl}^{-1})
\cdot \mu_{ij}^{-1}  \cdot
\mu_{ij} \cdot h_{ij}(\mu_{jl}) \cdot \gamma_{ijl}
\cdot \mu_{il}^{-1} \\
&= \mu_{ij} \cdot h_{ij}(\mu_{jk}) \cdot \gamma_{ijk}
\cdot h_{ik}(\mu_{kl}) \cdot \gamma_{ijk}^{-1}
\cdot h_{ij}(\gamma_{jkl})\cdot h_{ij}(\mu_{jl}^{-1})
\cdot\mu_{ij}^{-1} \\
&\quad \cdot \mu_{ij} \cdot h_{ij}(\mu_{jl})
\cdot \gamma_{ijl} \cdot \mu_{il}^{-1}.
\endaligned$$
Hence follows (\ref{A1.71.2}).
\end{rem}
\begin{defn}\label{A1.80}
Let $\mathcal U' = \{U'_j \mid j\in J\}$ be another covering of $M$
and let $i(\cdot): j \mapsto i(j)$ be a map $J \to I$ such that
$U'_j \subseteq U_{i(j)}$. We define a map:
$$
i(\cdot)^*: Sh((M,\mathcal U);\underline G) \to Sh((M,\mathcal
U');\underline G)
$$
by
$$
i(\cdot)^*([\{h_{i_1i_2}\},\{\gamma_{i_1i_2i_3}\}]) =
[\{h'_{j_1j_2}\},\{\gamma'_{j_1j_2j_3}\}]
$$
where
$$
h'_{j_1j_2} = h_{i(j_1)i(j_2)}, \quad \gamma'_{j_1j_2j_3} =
\gamma_{i(j_1)i(j_2)i(j_3)}.
$$
We thus obtain an inductive system $\mathcal U \mapsto
Sh((M,\mathcal U);\underline G)$. We take the inductive limit with
respect to this inductive system and define
$$
Sh(M,\underline G) = \underset{\longrightarrow}{\lim}\,\,Sh((M,\mathcal U);\underline G).
$$
An element of $Sh(M,\underline G)$ is said to be a {\it sheaf of category $\underline G$}
on $M$.
\end{defn}
\begin{rem}\label{A1.81}
\begin{enumerate}[{\rm (i)}]
\item
There is a more general notion, that is, a stack in the literature.
It was defined by Grothendieck (\cite{Grot62}, \cite{Grot71}). See
also \cite{Bry93}, \cite{Gir70}. We only consider the case we use for our
purpose where the stalk of the sheaf is the category $\underline G$
which is independent of the point.
\item
In case when $G$ is commutative, (\ref{A1.71.1}) implies
$$
h_{ij} \circ h_{jk} =  h_{ik}.
$$
Therefore it defines a $G$ local system $\mathfrak G$. Then (\ref{A1.71.2}) becomes
$$
\gamma_{ijk} + \gamma_{ikl} = h_{ij}(\gamma_{jkl}) + \gamma_{ijl}.
$$
Namely $\{\gamma_{ijk}\}$ defines a \v Cech cocycle in
$\text{\it \v C}\,^2(\mathcal U,\mathfrak G)$.
\par
Next we assume that $(\{h_{ij}\},\{\gamma_{ijk}\})$ is
isomorphic to $(\{h'_{ij}\},\{\gamma'_{ijk}\})$.
Then (\ref{A1.72.1}) and (\ref{A1.72.2}) imply that
the induced local system is isomorphic and $\{\gamma^{\prime\prime}_{ijk}\}$
is the same \v Cech cocycle as $\{\gamma_{ijk}\}$ under this isomorphism.
(\ref{A1.72.3}) and (\ref{A1.72.4}) imply that
$$
\gamma'_{ijk} - \gamma^{\prime\prime}_{ijk}
= \mu_{ij} + h^{\prime\prime}_{ij}(\mu_{jk}) - \mu_{ik}.
$$
Namely $\{\gamma^{\prime}_{ijk}\}$
is cohomologous to $\{\gamma^{\prime\prime}_{ijk}\}$.
Thus
$$
Sh(M,\underline G) \cong \bigcup_{\mathfrak G: \text{$G$ local
systems}}\text{\it \v H}\,^2(M;\mathfrak G)
$$
in the abelian case.
\par
\item
Usually (but not always) the effectivity of the (finite) group
$\Gamma_p$ action on $V_p$ is assumed when one defines the
notion of a
chart $(V_p,\Gamma_p,\psi_p)$ of an orbifold.
On the other hand, there is no such assumptions for stacks.
\par
Note (\ref{A1.71.1}) is the same formula as the first formula of Definition \ref{A1.5} (ii)
in the definition of Kuranishi structure.
In Definition \ref{A1.5} (ii) we assumed only the {\it existence} of $\gamma_{pqr}$.
Namely it is not a part of the structure. Also the
formula corresponding to (\ref{A1.71.2}) is not in Definition \ref{A1.5}.
On the other hand, in Definition \ref{A1.70} we include $\gamma_{ijk}$
as a part of the structure.
(Note in Definition \ref{A1.70}, $\gamma^{\alpha}_{ijk}$ may
depend on the connected component $\alpha$. Here $V_p \cap V_q \cap V_r$ is
connected since
we assume our cover is a good cover.)
\par
Actually, in the situation of Definition \ref{A1.5} where the $\Gamma_p$
action is assumed to be effective,
the element $\gamma_{pqr}$ satisfying Definition \ref{A1.5} (ii) is unique if it exists.
Moreover a formula corresponding to (\ref{A1.71.2}) can be proved.
(Lemma \ref{A1.73}.)
\par
In our situation where the $G$ action on $M$ is trivial,
$\gamma_{ijk}$ is not determined from the other data and
so we include it as a part of the structure. Also (\ref{A1.71.2}) is put
as a part of conditions.
\par
When the notion of orbifold was discovered by Satake \cite{Sat56},
he assumed the effectivity of the action of $\Gamma_p$. Later when
Thurston renamed Satake's V-manifold as an  orbifold, he did not
change its mathematical content and still assumed the effectivity of
$\Gamma_p$. Because of this, we include the effectivity of
$\Gamma_p$ as a part of the definition of orbifold in this paper.
\item Consider the situation of Definition \ref{A1.5}. Then, in Lemma \ref{A1.73}, we proved
the equality
\begin{equation}
\gamma_{pqr} \cdot \gamma_{prs} = h_{pq}(\gamma_{qrs}) \cdot \gamma_{pqs}
\label{A1.82}\end{equation}
where
$q \in \psi_p(s_p^{-1}(0)/\Gamma_p)$, $r \in \psi_q(s_q^{-1}(0)/\Gamma_q)$,
$s \in \psi_r(s_r^{-1}(0)/\Gamma_r)$.
Since (\ref{A1.82}) is automatic, we did not put it as a part of assumptions in
Definition \ref{A1.5}. In the situation where effectivity of the
$\Gamma_p$ action is not assumed, (\ref{A1.82}) will not be automatic.
\item Using the language of category theory, we can rewrite the definition
of $Sh(M,\underline G)$, as follows.
(Our discussion below is informal since we do {\it not}
use it in this paper.)
\par
We first define a category $\mathcal O(M)$.
Its objects are open sets of $M$.
There is no morphism from $U$ to $V$
if $U$ is not a subset of $V$. If $U \subset V$, there exists exactly
one morphism from $U$ to $V$.
\par
We next consider the 2-category $\underline{\underline G}$ as follows.
There is only one object  in it. The
category of morphism from this object to itself is $\underline G$.
\par
Then an element of $Sh(M,\underline G)$ is regarded as a
pseudo-functor from $\mathcal O(M)$ to $\underline{\underline G}$,
in the sense of \cite{Grot71} Expos\'e VI 8.
\par
Let us explain how a pseudo-functor $\mathcal O(M) \to
\underline{\underline G}$ is related to an element of
$Sh(X,\underline G)$. A pseudo-functor $\mathcal O(M) \to
\underline{\underline G}$ first assigns a functor $F_{UV}:
\underline G \to \underline G$ for each $U \subset V$. Such a
functor is nothing but a homomorphism $\phi_{UV}: G \to G$.
\par
If $U_3 \subset U_2 \subset U_1$, then the pseudo-functor associate
a natural transformation
$$
T_{U_3U_2U_1}: F_{U_3U_1} \to F_{U_3U_2}\circ F_{U_2U_1}
$$
which is (in our situation) automatically an equivalence.
By definition of the category $\underline G$, such a natural
transformation is given by an element $\gamma_{U_3U_2U_1} \in G$
such that
$$
\gamma_{U_3U_2U_1} \cdot \phi_{U_3U_1} = \phi_{U_3U_2} \circ \phi_{U_2U_1}.
$$
This formula corresponds to (\ref{A1.71.1}).
\par
For the pair $(F_{UV},\gamma_{U_3U_2U_1})$ to be a pseudo-functor we need
to assume a compatibility condition between them, that is
the commutativity of the following diagram for each
$U_4 \subset U_3 \subset U_2 \subset U_1$.
$$
\CD
F_{U_4U_3}\circ F_{U_3U_1} @<{T_{U_4U_3U_1}}<< F_{U_4U_1} @>{T_{U_4U_2U_1}}>> F_{U_4U_2}\circ F_{U_2U_1}\\
@VV{(F_{U_4U_3})_*(T_{U_3U_2U_1})}V  &&  @V{(F_{U_2U_1})^*(T_{U_4U_3U_2})}VV \\
F_{U_4U_3}\circ F_{U_3U_2}\circ F_{U_2U_1} & @= & F_{U_4U_3}\circ F_{U_3U_2}\circ F_{U_2U_1}.
\endCD
$$
(See Definition 3.10 (iv)(b) \cite{FGIKNV05} or \cite{Grot71}
Expos\'e VI Proposition 7.4.) The commutativity of this diagram is
equivalent to
$$
\gamma_{U_4U_3U_2} \cdot \gamma_{U_4U_2U_1} = \phi_{U_4U_3}(\gamma_{U_3U_2U_1})
\cdot \gamma_{U_4U_3U_1}.
$$
This formula is the same as (\ref{A1.71.2}).
(\ref{A1.71.1}) is a consequence of the fact that $T_{U_3U_2U_1}$ is a natural
transformation.
\par
We note that, in the definition of pseudo-functor in Definition
3.10 \cite{FGIKNV05}, there are other conditions (ii), (iv)(a). In
our situation, it will become
$$
h_{ii}(g) = \gamma_{iii}\cdot g\cdot \gamma_{iii}^{-1},
\quad
\gamma_{iij} = \gamma_{iii}, \quad \gamma_{ijj} = h_{ij}(\gamma_{jjj}).
$$
They follow from (\ref{A1.71}). In fact, the first formula follows from
(\ref{A1.71.1}) by putting $i=j=k$. The second formula
follows from (\ref{A1.71.2}) by putting $i=j=k$.
The third formula follows from (\ref{A1.71.2}) by putting $j=k=l$.
\par
We can continue and rewrite (\ref{A1.72}) using category theory.
We do not try to do it here.
In fact, the theory of stack which is well established is based on
category theory, and in this subsection we try to give a self-contained
account of the part thereof which we need,
{\it without} using category theory.
\item We did {\it not} assume
$$
h_{ii} = id, \qquad h_{ij} \circ h_{ji} = id,
$$
in Definition \ref{A1.70}.
There seems to be the version that assumes the above identities together with
$$
\gamma_{ijk} = \gamma_{ikj}^{-1}, \qquad \gamma_{ijk} = h_{ij}(\gamma_{jki}).
$$
In the abelian case, these conditions will become the condition
that $h_{ij\cdots i_k k}$ is anti-symmetric with respect to the change
of indices. It is well known that we have the same \v Cech cohomology,
whether or not we assume the anti-symmetricity.
\end{enumerate}
\end{rem}
\section{The stack structure of the singular locus
}
\label{subsecA1.6.2}
Now we apply the above discussion of stacks to the circumstance that arises in
later sections in relation to the study of normal bundles of the
singular locus $X^{\cong}(\Gamma)$.
\begin{exdefn}\label{A1.83}
Let $G$ be a finite group acting effectively and smoothly on a
smooth manifold $\widetilde X$. Consider the orbifold $X =
\widetilde X/G$.
Such an orbifold is said to be a {\it global quotient}.
Let $\Gamma$ be an abstract group. We put
$$
X^{\cong}(\Gamma) = \{ x \in \widetilde X \mid I_x \cong \Gamma
\}/G,
$$
where
$$
I_x = \{ g \in G \mid gx = x\}.
$$
It follows from (\ref{A1.84}) below that $X^{\cong}(\Gamma)$ is a smooth manifold.
\par
We now give a construction of an element of
$Sh(X^{\cong}(\Gamma),\underline{G})$.
We decompose
$X^{\cong}(\Gamma)$ as in \eqref{A1.84} and study each of them
separately.
Namely, for each subgroup $\Gamma_0 \subset G$ with
$\Gamma_0\cong\Gamma$, we consider
$$
\widetilde X^=(\Gamma_0) = \{x \in \widetilde X \mid I_x = \Gamma_0
\}.
$$
Denote the normalizer of $\Gamma_0$ by
$$
N(\Gamma_0) = \{ g\in G \mid g\Gamma_0g^{-1} = \Gamma_0\}.
$$
Then $H(\Gamma_0) = N(\Gamma_0)/\Gamma_0$ acts freely on $\widetilde
X^=(\Gamma_0)$ and by definition we have
\begin{equation}
X^{\cong}(\Gamma) = \bigcup_{\Gamma_0}\widetilde
X^=(\Gamma_0)/H(\Gamma_0),
\label{A1.84}\end{equation}
where the union is taken
over a complete system of representatives
of the conjugacy classes of the subgroups of $G$ isomorphic to
$\Gamma$. We have an exact sequence
\begin{equation}
1 \longrightarrow \Gamma_0 \longrightarrow N(\Gamma_0)
 \longrightarrow H(\Gamma_0) \longrightarrow 1.
\label{A1.85}\end{equation}
\par
We choose a sufficiently fine good
covering $\mathcal U = \{U_i \mid i
\in I\}$ of $X^{\cong}(\Gamma)$ and a lift $\widetilde U_i \subset
\widetilde X^=(\Gamma_0)$ of $U_i$ so that the projection
$\widetilde{X}\to X$ restricts to a homeomorphism from
$\widetilde{U}_i$ to $U_i$.
\par
For each $i,j$ with $U_i \cap U_j \ne \emptyset$ there exists a
unique $\underline h_{ij} \in H(\Gamma_0)$ such that $\widetilde U_i
\cap (\underline h_{ij} \widetilde U_j) \ne \emptyset$. We remark
that $\underline h_{ij} \cdot \underline h_{jk} = \underline h_{ik}$
in $H(\Gamma_0)$, if $U_i \cap U_j \cap U_k \ne \emptyset$.
\par
We choose lifts $\widetilde{\underline h}_{ij} \in N(\Gamma_0)$ of
$\underline h_{ij}$.
\par
We define an automorphism $h_{ij}: \Gamma_0 \to \Gamma_0$ by
$$
h_{ij}(g) = \widetilde{\underline h}_{ij} \cdot g \cdot
\widetilde{\underline h}_{ij}^{-1}.
$$
We define $\gamma_{ijk} \in \Gamma_0$ by
$$
\gamma_{ijk} \cdot \widetilde{\underline h}_{ik} =
\widetilde{\underline h}_{ij}\cdot \widetilde{\underline h}_{jk}.
$$
Then it is easy to check that $\{ \gamma_{ijk}\}$ satisfies (\ref{A1.71}).
\par
We can generalize the above construction and include the case of an orbifold
that is not necessarily a global quotient of a  manifold.
We do not discuss it here since we do not use it in the
main application (the proof of Theorem \ref{maintechnicalresult}).
\end{exdefn}
We note that we can choose $\gamma_{ijk} = 1$ if the exact sequence (\ref{A1.85})
splits. But this is not always the case.
\begin{exm}\label{A1.86}
Let us consider the orbifold $X$ given in Example \ref{A1.68}.
Then $X^{\cong}(\Gamma) = S^2$. The $\Z_p$ local system
is necessarily trivial on $S^2$. So
$$
Sh(S^2;\Z_p) \cong \text{\it \v H}\,^2(S^2;\Z_p) \cong \Z_p.
$$
The element thereof defined in Example-Definition \ref{A1.83} is the
generator of $\Z_p$ and hence is nonzero.
\end{exm}
\begin{lem}\label{A1.87}
The element of $Sh(X^{\cong}(\Gamma),\underline G)$
represented by $(\{h_{ij}\},\{\gamma_{ijk}\})$ is independent of various
choices involved in the construction.
\end{lem}
\begin{proof} We first fix $\mathcal U = \{U_i \mid i \in I\}$. We change
$\widetilde U_i$ to $\alpha_i \widetilde U_i$ where $\alpha_i \in
N(\Gamma_0)$. We also change $\widetilde{\underline h}_{ij}$  to
$$
\widetilde{\underline h}^{\prime\prime}_{ij} = \alpha_i \cdot
\widetilde{\underline h}_{ij} \cdot \alpha_j^{-1}.
$$
Then $h_{ij} \in \text{Aut}(G)$ is transformed to
$$
h_{ij}^{\prime\prime} = \text{ad}(\alpha_i) \circ
h_{ij} \circ \text{ad}(\alpha_j)^{-1},
$$
where $\text{ad}: N(\Gamma_0) \to \text{Aut}(\Gamma_0)$ is defined
by $\text{ad}(g)(g') = g \cdot g' \cdot g^{-1}$. We put $\psi_i =
\text{ad}(\alpha_i)$ and $\mu_{ij} = 1$. Then (\ref{A1.72.1}),
(\ref{A1.72.3}) are satisfied and hence
$(\{\mu_{ij}\},\{\psi_i\})$
defines an isomorphism.
\par
We also have
$$
\gamma_{ijk}^{\prime\prime} =
\alpha_i \cdot
\gamma_{ijk}
\cdot \alpha_i^{-1}
= \psi_i(\gamma_{ijk}).
$$
Therefore the isomorphism class is independent of the choice of
$\widetilde U_i$.
\par
We next fix $\widetilde U_i$ and change the lift
$\widetilde{\underline h}_{ij} \in N(\Gamma_0)$ of $\underline
h_{ij}$. We put
$$
\widetilde{\underline h}'_{ij} = \mu_{ij} \cdot
\widetilde{\underline h}_{ij}.
$$
Then we have
$$
\gamma'_{ijk} = \widetilde{\underline h}'_{ij} \cdot
\widetilde{\underline h}'_{jk} \cdot (\widetilde{\underline
h}'_{ik})^{-1} = \mu_{ij} \cdot h_{ij}(\mu_{jk}) \cdot \gamma_{ijk}
\cdot\mu_{ik}^{-1}
$$
as required. The invariance under the refinement of the covering is
easy to prove. \end{proof}
\begin{defn}\label{A1.88}
We call the structure defined by $(\{h_{ij}\},\{\gamma_{ijk}\}) \in Sh(X^{\cong}(\Gamma),\underline{G})$
in Example-Definition \ref{A1.83}, the {\it standard stack structure
on $X^{\cong}(\Gamma)$}.
\end{defn}

\section{The normal bundle of the singular locus
and $C^{\infty}$ local triviality}\label{normalstack}
Going back to the general topological space $M$, we next
define a vector bundle on the stack defined by an element of
$Sh(M,\underline G)$. Let $(\{ h_{ij}\},\{\gamma_{ijk}\}) \in
Sh((M,\mathcal U),\underline G)$.
\begin{defn}\label{A1.89}
A {\it vector bundle} on $(\{h_{ij}\},\{\gamma_{ijk}\})\in
Sh((M,\mathcal U),\underline G)$ is a pair $(\{F_i\},\{g_{ij}\})$
such that $F_i$ is a vector bundle on $U_i$ with $G$ action and
$g_{ij}$ is an $h_{ij}$-equivariant bundle isomorphism $ g_{ij}:
\left. F_j\right\vert_{U_i \cap U_j} \to \left. F_i\right\vert_{U_i
\cap U_j} $ such that:
\begin{equation}
g_{ij} \circ g_{jk}= \gamma_{ijk} \cdot g_{ik}.
\label{A1.90}\end{equation}
\par
We assume that $(\{\mu_{ij}\},\{\psi_i\})$ is an isomorphism
$(\{h_{ij}\},\{\gamma_{ijk}\}) \to (\{h'_{ij}\},\{\gamma'_{ijk}\})$.
An {\it isomorphism} from a vector bundle $\mathcal F =
(\{F_i\},\{g_{ij}\})$ on $(\{h_{ij}\},\{\gamma_{ijk}\})$ to a vector
bundle $\mathcal F' = (\{F'_i\},\{g'_{ij}\})$ on
$(\{h'_{ij}\},\{\gamma'_{ijk}\})$ is a family $\{\phi_i\}_{i\in I}$
of $\psi_i$-equivariant  isomorphisms of vector bundles $ \phi_i:
F_i \to F'_i $ such that
\begin{equation}
g'_{ij} \circ \phi_j = \mu_{ij} \cdot (\phi_i \circ g_{ij}).
\label{A1.91}\end{equation}
We say $\{\phi_i\}_{i\in I}$ is an {\it isomorphism $: \mathcal F \to \mathcal F'$ over
$(\{\mu_{ij}\},\{\psi_i\})$}.
\end{defn}
\begin{lem}\label{A1.92}
The relation `isomorphic' in Definition {\rm \ref{A1.89}} is an equivalence relation.
\end{lem}
\begin{proof}
Let $(\mu_{ij},\psi_i)_* (h^1_{ij},\gamma^1_{ijk}) =
(h^2_{ij},\gamma^2_{ijk})$ and $(\mu'_{ij},\psi'_i)_*
(h^2_{ij},\gamma^2_{ijk}) = (h^3_{ij},\gamma^3_{ijk})$. Let
$\mathcal F^c = (\{F_i^c\},\{g_{ij}^c\})$ be a vector bundle on
$(\{h^c_{ij}\},\{\gamma^c_{ijk}\})$ and $\{\phi_i\}_{i\in I}:
\mathcal F^1 \to \mathcal F^2$ and $\{\phi'_i\}_{i\in I}: \mathcal
F^2 \to \mathcal F^3$ be  isomorphisms over
$(\{\mu_{ij}\},\{\psi_i\})$, $(\{\mu'_{ij}\},\{\psi'_i\})$,
respectively. Then we can check easily by calculation that
$\{\phi'_i\circ \phi_i\}_{i\in I}$ is an isomorphism $: \mathcal F^1
\to \mathcal F^3$ over $(\{\mu'_{ij}\},\{\psi'_i\}) \circ
(\{\mu_{ij}\},\{\psi_i\}) = (\{\mu'_{ij} \cdot
\psi'_i(\mu_{ij})\},\{\psi'_i\circ \psi_i\})$.
\end{proof}
\begin{lem}\label{A1.93}
Let $\mathcal F = (\{F_i\},\{g_{ij}\})$ be a vector bundle on
$(\{h_{ij}\},\{\gamma_{ijk}\})$ and let $(\{\mu_{ij}\},\{\psi_i\})$
be an isomorphism $(\{h_{ij}\},\{\gamma_{ijk}\}) \to
(\{h'_{ij}\},\{\gamma'_{ijk}\})$. Let $F'_i$ be a $G$-equivariant
vector bundle on $U_i$ and let $\phi_i: F_i \to F'_i$ be a
$\psi_i$-equivariant bundle isomorphism. We {\bf define} $g'_{ij}$
by {\rm (\ref{A1.91})}.
Then $(\{F'_i\},\{g'_{ij}\})$ is a vector bundle on
$(\{h'_{ij}\},\{\gamma'_{ijk}\})$.
\end{lem}
\begin{proof}
It is easy to see that $g'_{ij}$ is $h'_{ij}$  equivariant.
So it suffices to check (\ref{A1.90}) for $g'_{ij}$ and $\gamma'_{ijk}$.
We may divide the cases into $(\mu_{ij},\psi_i) = (1,\psi_i)$ and
$(\mu_{ij},\psi_i) = (\mu_{ij},1)$.
\par
In case $(\mu_{ij},\psi_i) = (1,\psi_i)$ we have
$$\aligned
g'_{ij}\circ g'_{jk}\circ \phi_k &=
\phi_i \circ g_{ij} \circ g_{jk}
= \phi_i (\gamma_{ijk} \cdot g_{ik})\\
&= \psi_i(\gamma_{ijk})\cdot \phi_i  \circ g_{ik}
= \gamma'_{ijk} \cdot g'_{ik}\circ \phi_k
\endaligned$$
as required.
\par
In case $(\mu_{ij},\psi_i) = (\mu_{ij},1)$ we have
$$\aligned
g'_{ij} \circ g'_{jk}\circ \phi_k
&= g'_{ij} \circ (\mu_{jk}\cdot (\phi_j \circ g_{jk})) \\
&= h'_{ij}(\mu_{jk}) \cdot \mu_{ij} \cdot \phi_i\circ (g_{ij}
\circ g_{jk}) \\
&= \mu_{ij} \cdot h_{ij}(\mu_{jk}) \cdot \gamma_{ijk} \cdot \mu_{ik}^{-1}
\cdot (g'_{ik} \circ \phi_k) \\
&= \gamma'_{ijk} \cdot  (g'_{ik} \circ \phi_k),
\endaligned$$
as required.
The proof of the lemma is now complete.
\end{proof}
We next discuss how a vector bundle behaves under the refinement of
the covering. Let
$$
i(\cdot)^*([\{h_{i_1i_2}\},\{\gamma_{i_1i_2i_3}\}]) =
[\{h'_{j_1j_2}\},\{\gamma'_{j_1j_2j_3}\}].
$$
See Definition \ref{A1.80}.
Then a vector bundle $(\{F_i\},\{g_{i_1i_2}\})$ on
$(\{h_{i_1i_2}\},\{\gamma_{i_1i_2i_3}\})$ induces $(\{F'_j\},\{g'_{j_1j_2}\})$
on
$(\{h'_{j_1j_2}\},\{\gamma'_{j_1j_2j_3}\})$ by
$$
g'_{j_1j_2} = g_{i(j_1)i(j_2)}\vert_{U'_{j_1j_2}}.
$$
\par
Therefore we can define the notion of a vector bundle on a pair
$(M,[\{h_{ij}\},\{\gamma_{ijk}\}])$ where
$[\{h_{ij}\},\{\gamma_{ijk}\}] \in Sh(M,\underline G)$.
\par
We consider the situation of Example-Definition \ref{A1.83} and use the
notations there.
\begin{defn}\label{A1.94}
We define the normal bundle $N_{X^{\cong}(\Gamma)}X$ over $X^{\cong}(\Gamma)$ with the standard stack structure as follows:
We identify $U_i \subseteq X^{\cong}(\Gamma)$ with $\widetilde U_i
\subset \widetilde X^=(\Gamma_0)$ by the projection and put $F_i =
N_{\widetilde U_i}\widetilde{X}$, the normal bundle of $\widetilde
U_i$ in $\widetilde X$.
The $\Gamma_0$ ($\subset G$) action on
$\widetilde{X}$ induces one on $F_i$.
\par
We next define $g_{ij}$. We have
$$
(\widetilde{\underline h}_{ij}^{-1} \widetilde U_i) \cap \widetilde
U_j \subset \widetilde U_j \cong U_j.
$$
We identify $(\widetilde{\underline h}_{ij}^{-1}) \widetilde U_i
\cap \widetilde{U}_j$ with $U_i \cap U_j$. Then an open embedding
$$
\widetilde{\underline h}_{ij} \cdot: U_i \cap U_j \to \widetilde U_i
\cong U_i
$$
is induced. It extends to
a map $\widetilde X \to \widetilde X$. Then we have
$$
g_{ij}: = (\widetilde{\underline h}_{ij} \cdot)_*: F_j\vert_{U_i
\cap U_j} = N_{\widetilde{\underline h}_{ij}^{-1} \widetilde
U_i}\widetilde X \to F_i = N_{\widetilde U_i}\widetilde X.
$$
Since
$
\widetilde{\underline h}_{ij} \cdot \widetilde{\underline h}_{jk} =
\gamma_{ijk} \cdot \widetilde{\underline h}_{ik},
$
we obtain (\ref{A1.90}). Therefore
$(\{F_i\}, \{ g_{ij}\})$ is a vector bundle
on $(\{h_{ij}\},\{\gamma_{ijk}\}) \in Sh(X^{\cong}(\Gamma), \underline{G})$
in the sense of Definition \ref{A1.89}.
We put
$$N_{X^{\cong}(\Gamma)}X = (\{F_i\},\{ g_{ij}\})
$$
which we call
the {\it normal bundle} of $X^{\cong}(\Gamma)$.
\par
We can generalize this construction to
the case when $X$ is an orbifold,
not necessarily a global quotient of a manifold.
We do not discuss this here since it is not needed in our main
application (the proof of Theorem \ref{maintechnicalresult}).
\end{defn}
\begin{lem}\label{A1.95}
The vector bundle
$(\{F_i\},\{g_{ij}\})$ in Definition {\rm \ref{A1.94}} is independent of the
choices involved in the construction up to isomorphism.
\end{lem}
\begin{proof}
We use the notation of the proof of Lemma \ref{A1.87}.
Let $\alpha_i \in N(\Gamma_0)$.
\par
We first change $\widetilde U_i$ to $\alpha_i\cdot \widetilde U_i$.
Then $\alpha_i$ induces a map
$$
\alpha_i \cdot: N_{\widetilde U_i}\widetilde X \to
N_{\alpha_i\cdot\widetilde U_i}\widetilde X.
$$
It induces
$$
\phi_i = (\alpha_i \cdot)_*: F_i \to F'_i.
$$
Since $\psi_i = {\rm {ad}} (\alpha_i )$, it follows that
$\phi_i$ is $\psi_i$ equivariant.
The equality
$$
\widetilde{\underline h}^{\prime}_{ij} = \alpha_i \cdot
\widetilde{\underline h}_{ij} \cdot \alpha_j^{-1}
$$
implies that
$$
g'_{ij} \circ \phi_j = \phi_i \circ g_{ij}.
$$
Since $\mu_{ij} = 1$ in this case, we obtain the required
isomorphism.
\par
We next fix $\{\widetilde U_i\}$ and change the lift
$\{\widetilde{\underline h}_{ij}\}$ to
\begin{equation}
\widetilde{\underline h}'_{ij} = \mu_{ij}\cdot \widetilde{\underline
h}_{ij}. \label{A1.96}\end{equation}
In this case we have $F_i = F'_i$ and $\phi_i
=$ identity. (\ref{A1.96}) implies
$$
g'_{ij} = (\widetilde{\underline h}'_{ij})_* = \mu_{ij}\cdot
(\widetilde{\underline h}_{ij})_* = \mu_{ij} \cdot g_{ij}.
$$
The invariance under the refinement of the covering is easy to prove.
\end{proof}
\begin{exm}\label{A1.97}
In the situation of Example \ref{A1.68}, the normal
bundle does not exist as a (usual) vector bundle
on $S^2$. But it exists over
$S^2$ with nontrivial stack structure, which
corresponds to the generator of $\text{\it \v H}\,^2(S^2;\Z_p)$. See Example \ref{A1.86}.
\end{exm}
\begin{defn}\label{A1.98}
We consider the situation of Example-Definition \ref{A1.83} and let
$\widetilde E$ be a $G$ equivariant vector bundle on $\widetilde X$.
It induces an orbi-bundle $E$ on $X = \widetilde X/G$.
We define a vector bundle $E\vert_{X^{\cong}(\Gamma)}$ on $X^{\cong}(\Gamma)$ as follows:
We use the notation in Example-Definition \ref{A1.83}
and Definition \ref{A1.94}.
\par
We put
$$
E_i = \widetilde E\vert_{\widetilde U_i}
$$
and regard it as a vector bundle on $U_i$.
As in Definition \ref{A1.94}, we have
an open embedding
$$
\widetilde{\underline h}_{ij} \cdot: U_{i}\cap U_j \to \widetilde
U_i \cong U_i.
$$
Since $\widetilde{\underline h}_{ij}\in G$, it induces a bundle map
$$
g_{ij}: E_j\vert_{U_{i}\cap U_j} \to E_i.
$$
It follows that $g_{ij}$ satisfies the required relation in the same
way as in Definition \ref{A1.94}. We define
$$
E\vert_{X^{\cong}(\Gamma)} = (\{E_i\},\{g_{ij}\})
$$
and call it the {\it restriction} of $E$ to $X^{\cong}(\Gamma)$.
\par
We can also generalize this construction to the case of an orbifold
which is not necessarily a global quotient.
\par
In the same way as Lemma \ref{A1.95}, we can prove that
$(\{E_i\},\{g_{ij}\})$ is independent of the choices
up to isomorphism.
\end{defn}

Before proceeding further, we review the definitions of diffeomorphisms
between orbifolds and of locally trivial $C^{\infty}$ fiber bundles with orbifolds as fibers.
Those notions will be used in Sections \ref{good orbifold}, \ref{section35.2} and \ref{section35.3}.
They are of course well established. However there are several delicate
points in its definition. Since those points are related to the arguments of later sections,
we state a precise definition that we use in this paper.
\par
Let $\widetilde{X}$ be a smooth manifold and $G$ a finite group acting
effectively on it. The quotient $X=\widetilde{X}/G$ defines an orbifold.
Let $X' = \widetilde{X}'/G'$ be another global quotient and
$W \subset X$, $W' \subset X'$ open subsets.
\begin{defn}\label{diffeo}
A homeomorphism $F : W \to W'$ is said to be a
{\it diffeomorphism} if the following holds
for every $p\in W$ and $p' = F(p)$.
\par
Let $\tilde p \in \widetilde{X}$, $\tilde p' \in \widetilde{X}'$ be lifts of $p$ and $p'$, respectively.
Then there exist open neighborhoods $\widetilde U$, $\widetilde U'$ of them,
an isomorphism $\phi : I_{\tilde p} \to I_{\tilde p'}$ and
a $\phi$-equivariant  diffeomorphism $\tilde F_p :  \widetilde U \to \widetilde U'$
such that $\widetilde U/ I_{\tilde p}$, $\widetilde U'/I_{\tilde p'}$ are open subsets
of $X$, $X'$ and that $\widetilde F$ induces the restriction of the map $F$ to $\widetilde U/ I_{\tilde p}$.
\end{defn}
\begin{defn}\label{fiber}
Let $N$ be a smooth manifold, $X=\widetilde X/G$, $X' = \widetilde X'/G'$ global quotients,
and $W \subset X$, $F \subset X'$ open subsets. A continuous map
$\pi : W \to N$ is said to be a {\it locally trivial fiber bundle of $C^{\infty}$-class
with fiber $F$} if the following holds.
\par
For each $p \in N$ there exist an open neighborhood $U_p$ of it
and a diffeomorphism $\phi_p : F \times U_p \to \pi^{-1}(U_p)$  in the sense of
Definition \ref{diffeo}, such that the following diagram commutes.
\begin{equation}\label{fiberdiagram}
\begin{CD}
F \times U_p @>>{\phi_p}> \pi^{-1}(U_p) \\
@VVV  @VV{\pi}V \\
U_p &=&  U_p
\end{CD}
\end{equation}
\end{defn}
\begin{rem}
We require the commutativity of the Diagram (\ref{fiberdiagram}) set
theoretically. In general, the notion of morphisms between orbifolds
must be carefully defined. In fact, the systematic study of it
requires to use the notion of 2 category. In this paper, we  assume effectivity of the finite group
action in the definition of orbifold and in Definition \ref{fiber} we assumed the base space $N$ is a
manifold. So it suffices to require the commutativity of
Diagram (\ref{fiberdiagram}) set
theoretically, in our situation.
\end{rem}

Now let $\mathcal F = (\{F_i\},\{g_{ij}\})$ be a vector bundle over
$(\{h_{ij}\},\{\gamma_{ijk}\}) \in Sh(M,\underline G)$. Assume that
the $G$ action on the fibers of $F_i$ is effective. We are going to
define an orbifold structure on $\mathcal F/G$. We define an equivalence
relation $\sim$ on $ \bigcup_i F_i $ as follows. (Here $F_i$ also
denotes the total space of the vector bundle $F_i$ over $U_i$.) Let $x \in F_i$
and $y \in F_j$. Then $x\sim y$ if and only if one of the following
holds.
\begin{enumerate}
\item  $i=j$, $x = \gamma y$ for some $\gamma \in G$.
\par
\item  $\pi(x) \in U_{i} \cap U_j$, $y = \gamma \cdot g_{ji}(x)$ for some $\gamma \in G$.
Here $\pi: F_i \to U_i$, $\pi: F_j \to U_j$ are the projections.
\end{enumerate}
It is easy to see that $\sim$ is an equivalence relation.
We put
$$
\vert \mathcal F /G\vert = \bigcup_i F_i/\sim
$$
and define a quotient topology on it.
The projection $\pi: F_i \to U_i \subset M$
induces a map
$$
\pi: \vert \mathcal F /G\vert \to M.
$$
Let $F$ be the fiber of the vector bundle $F_i$. $F$ is a vector space on
which $G$ acts effectively.  We assume that $M$ is a smooth manifold.
\begin{lem} \label{A1.99}
$\vert\mathcal F /G\vert$ has a structure of an orbifold.
We denote it by $\mathcal F /G$.
If $\mathcal F$ is isomorphic to $\mathcal F'$, then $\mathcal F /G$
is diffeomorphic to $\mathcal F'/G$ as an orbifold.
\par
Moreover,
$\pi: \mathcal F/G \to M$ is a locally trivial fiber bundle of $C^{\infty}$-class, whose
fiber is $F/G$.
\end{lem}
\begin{proof}
Let $U_{i_1 \dots i_k}=\bigcap_{j=1}^k U_{i_j}$.
The bundle isomorphism $g_{ij}$ induces an $h_{ij}$-equivariant embedding $F_j\vert_{U_{ij}} \to F_i$,
which we also denote by $g_{ij}$.
By the definition of a vector bundle over $(\{h_{ij}\},\{\gamma_{ijk}\})$, we have
$$g_{ij} \circ g_{jk}=\gamma_{ijk} g_{ik} \quad \text{on } F_k\vert_{U_{ijk}}.$$
Therefore we find that $F_i/G$'s are glued to an orbifold.
Hence $\vert {\mathcal F}/G \vert$ has a structure of an orbifold.
\par
Let ${\mathcal F}$ and ${\mathcal F}'$ be vector bundles over $(\{h_{ij}\},\{\gamma_{ijk}\})$ and
$(\{h'_{ij}\},\{\gamma'_{ijk}\})$, respectively.
Suppose that $(\{\mu_{ij}\}, \{\psi_i\}):(\{h_{ij}\},\{\gamma_{ijk}\}) \to (\{h'_{ij}\},\{\gamma'_{ijk}\})$ is
an isomorphism and $\{\phi_i\}:{\mathcal F} \to {\mathcal F}'$ is an isomorphism over $(\{\mu_{ij}\},\{\psi_i\})$.
Then $\phi_i$ is a diffeomorphism between global quotient orbifolds $F_i/G$ and $F'_i/G$.
Clearly they are glued to a diffeomorphism between orbifolds ${\mathcal F}/G$ and ${\mathcal F}'/G$.

If necessary, we take a refinement of $\{U_i\}$ such that $F_i$ is $G$-equivariantly isomorphic to
$U_i \times F$.  Note that $G$ acts trivially on $U_i$ and effectively on $F$.
The open subset $\pi^{-1}(U_i) \subset \vert{\mathcal F}/G\vert$ is nothing but
$F_i/G \cong U_i \times F/G$.
Therefore $\pi: \mathcal F/G \to M$ is a locally trivial fiber bundle whose
fiber is $F/G$.
\end{proof}

\begin{exm}\label{normalbundleasloctrivial}
Let $\mathcal F = (\{F_i\},\{g_{ij}\})$ be a vector bundle over
$(\{h_{ij}\},\{\gamma_{ijk}\}) \in Sh(M,\underline G)$.
By Lemma \ref{A1.99}
$\pi : \mathcal F/G \to M$
is a locally trivial fiber bundle of $C^{\infty}$ class with fiber $\R^k/G$, where $g_{ij} : \R^k \to \R^k$.
\par
We remark that here we regard $M$ as a manifold and not a stack.
\par
We assume that $g_{ij} \in O(k)$. Namely we assume that it preserves a metric
of the fiber. We consider the set of the equivalence classes of
all $x \in F_i$ with $\Vert x\Vert \le r$ and denote it by
$\mathcal F/G(r)$. Then the restriction
$\pi : \mathcal F/G (r)\to M$ is also  a locally trivial fiber bundle of $C^{\infty}$ class,
that is a `ball bundle'.
Similarly when we denote the set of all the equivalence classes of $x \in F_i$ with $\Vert x\Vert = r$
by $S(\mathcal F/G)$,  it is a locally trivial fiber bundle of $C^{\infty}$ class over $M$
with fiber $S^{n-1}/G$. This is  a `sphere bundle'.
\par
In particular, if $X = \widetilde{X}/G$ is a global quotient, then
the map $N_{X^{\cong}(\Gamma)}X/\Gamma \to X^{\cong}(\Gamma)$
defines a  locally trivial fiber bundle of $C^{\infty}$ class.
We can define `ball bundle' and `sphere bundle', that are
locally trivial fiber bundles of $C^{\infty}$ class, also in this case.
\end{exm}
Example \ref{normalbundleasloctrivial} says the normal bundle of $X^{\cong}(\Gamma)$ is locally trivial in $C^{\infty}$ sense.
We note that in the theory of stratification
of analytic set or of Whitney stratification
one important point (observed by Whitney) is that
the normal cone to the stratum is locally trivial
only in $C^0$ sense.
\begin{exm}\label{35.3}
Let
$
C_a = \{ (tx,ty,t) \in \R^3 \mid -1 \le x \le 1, 0 \le y \le
1+a-a\vert x\vert, t\ge 0\}
$
and we put
$$
X = \{(x,y,z,w) \in \R^4 \mid (x,y,z) \in C_w, 0 < w < 1\}.
$$
Using the fact that $C_a$ is not affine isomorphic to $C_b$  for
$a \ne b$, we can prove that a neighborhood of $w$ axis in $X$ is
not diffeomorphic to the product $\R \times Z$ for any $Z \subset
\R^3$.
(Here we say $X$ is diffeomorphic to $\R \times Z$ if there exists a
homeomorphism which extends to a diffeomorphism to its sufficiently small open neighborhoods.)
\end{exm}

For the stratification $\{X^{\cong}(\Gamma)\}$ of the orbifold $X=\widetilde X/G$, we have the following:

\begin{lem}\label{A1.100}{\rm (Tubular neighborhood theorem)}
Consider the situation of Example-Definition {\rm \ref{A1.83}}.
Denote by $B(N_{X^{\cong}(\Gamma)}X)$ the unit ball bundle of $N_{X^{\cong}(\Gamma)}X$.
Then
$(B(N_{X^{\cong}(\Gamma)}X))/\Gamma$ is diffeomorphic to a neighborhood of
$X^{\cong}(\Gamma)$ in $X$ as an orbifold.
\end{lem}
\begin{proof}
We use the notation of Definition \ref{A1.94}.
We recall that
$$N_{X^{\cong}(\Gamma)}X =  (\{E_i\},\{g_{ij}\})$$
where $E_i = N_{\widetilde U_i}\widetilde X$.
Note that $E_i$ is the restriction of $N_{\widetilde{X}^=(\Gamma)}\widetilde{X}$ to
$\widetilde{U}_i \subset \widetilde{X}$.
\par
We pick open neighborhoods $W_1 \supset W_2 \supset \cdots$ of
$\bigcup_{\Gamma' \supset \Gamma}X^{\cong}(\Gamma')$ in $X$
such that $\bigcap_k W_k \cap X^{\cong}(\Gamma)=\emptyset$.
We may assume that $X^{\cong}(\Gamma)_k=X^{\cong}(\Gamma) \setminus W_k$ is
a closed submanifold in $X \setminus W_k$.
Denote by $\widetilde{X}^=(\Gamma)_k=\widetilde{X}^=(\Gamma) \setminus \pi^{-1}(W_k)$,
where $\pi:\widetilde{X} \to X$ is the projection.
\par
We take a $G$ invariant Riemannian metric on $\widetilde X$ and use
the exponential map to identify $E_i=N_{\widetilde
U_i}\widetilde{X}$ with a neighborhood of $\widetilde U_i$ in
$\widetilde X$ as follows.

Denote by $B_{\delta}(N_{\widetilde{X}^=(\Gamma)}\widetilde{X})$
the ball bundle of $N_{\widetilde{X}^=(\Gamma)}\widetilde{X}$ of radius $\delta$.
For $k=1,2, \dots$, we take $\delta_k>0$ such that the exponential map
$$
\exp: v \in B_{\delta_k} (N_{\widetilde{X}^=(\Gamma)}\widetilde{X} \vert_{\widetilde{X}^{=}(\Gamma)_k} )
\mapsto \exp (v) \in \widetilde{X},
$$
which is a diffeomorphism to an open neighborhood of $\widetilde{X}^{=}(\Gamma)_k$.

We pick an $H(\Gamma)$-invariant smooth positive function $\chi:\widetilde{X}^{=}(\Gamma) \to \R$ such that $\chi < \delta_{k+1}$ on
$\widetilde{X}^{=}(\Gamma) \cap \pi^{-1}(W_k)$.
Denote by $p_i:E_i \to \widetilde{U}_i$ the projection and
set $B_{\chi}(E_i)=\{v \in E_i \vert \| v \| < \chi(p_i(v)) \}$.
Then the exponential map $\exp$ restricted to
$B_{\chi}(E_i)$ induces a diffeomorphism from
$B_{\chi}(E_i)/G$ to an open neighborhood of $U_i$ in $X$ as orbifolds.
Clearly, we find that
$\underline{h}_{ij} \circ \exp\vert_{B_{\chi}(E_j) }= \exp\vert_{B_{\chi}(E_i)} \circ g_{ij}$
on $E_j\vert{\underline{h}_{ij}^{-1}(\widetilde{U}_i) \cap \widetilde{U}_j}$.
Hence they are glued together and we obtain a diffeomorphism
$\exp (\chi \cdot \bullet)$ from $(B(N_{X^{\cong}(\Gamma)}X))/\Gamma$ to an open neighborhood of
$X^{\cong}(\Gamma)$.

\end{proof}

Finally we remark that
various operations on vector bundles such as
Whitney sum, tensor product, Hom bundle, symmetric tensor product, etc.
can be generalized to the case of vector bundles  on
$[\{h_{ij}\},\{\gamma_{ijk}\}]
\in Sh(M;\underline G)$.
For example, if $\mathcal F = (\{F_i\},\{g^F_{ij}\})$ and
$\mathcal E  = (\{E_i\},\{g^E_{ij}\})$ are vector bundles on
$(\{h_{ij}\},\{\gamma_{ijk}\})$  then
$\operatorname{Hom} (\mathcal F,\mathcal E) =
(\{\operatorname{Hom} (F_i,E_i)\},\{g_{ij}\})$ where
$$
g_{ij}(u_j) = g_{ij}^E \circ u_j \circ (g_{ij}^{F})^{-1}.
$$
\par
We use the next lemma in Section \ref{section35.3}.
\begin{lem}\label{A1.101}
Let $\mathcal E  = (\{E_i\},\{g_{ij}\})$ be a vector bundle on
$(\{h_{ij}\},\{\gamma_{ijk}\})
\in  Sh(M,\underline G)$.
We put
$$
E_i^G = \{ v \in E_i \mid \forall \gamma\in G\,\,\gamma v = v\}.
$$
Then they are glued by $g_{ij}$ to define a
vector bundle on the topological space $M$ $($in the usual sense$)$.
\end{lem}
\begin{proof}
Let $g^G_{ij}: E_i^G\vert_{U_{ij}} \to E^G_{j}$ be the restriction
of $g_{ij}$. Since $G$ action on $E^G_i$ is trivial, it
follows from \eqref{A1.90} that
$$
g^G_{ij} \circ g^G_{jk} = g^G_{ik}.
$$
The lemma follows.
\end{proof}
\begin{rem}\label{A1.102}
In some situation we need to consider subspaces $Y_p \subset V_p$ of
a space $X$ with Kuranishi structure $(V_p,E_p,
\Gamma_p,\psi_p,s_p)$ and put a Kuranishi structure on the subspace.
When the action of $\Gamma_p$ is not effective, we need to add  some
$\Gamma_p$ vector space $F_p$ to both $Y_p$ and obstruction bundle
$E_p$ of each Kuranishi neighborhood, in order to define a Kuranishi
structure on the subspace $\cup_p Y_p$. We need some care to carry
this out. Namely we should glue those  vector spaces by a family of
linear isomorphisms $g_{pq}:  U_{pq} \to
\operatorname{Hom} (F_q,F_p)$ so that
$$
g_{pq}\circ g_{qr} = \gamma_{pqr} \cdot g_{pr}
$$
where $\gamma_{pqr}$ is the one appearing in Definition \ref{A1.5}
(ii). For example, in the case of Kuranishi structure with corners,
we define Kuranishi structure on codimension $k$ corner $S_kX$ by
taking the normal bundle of $S_kX$ and adding the normal bundle
to
both $S_kV_p$ and the obstruction bundle in order to make the
$\Gamma_p$ action effective.
\end{rem}

\section{Single valued piecewise smooth section of orbi-bundle: statement}\label{good orbifold}

Let $X = \widetilde{X}/G$ be a global quotient.
We consider the situation of Example-Definition
\ref{A1.83}.
We decompose $X^{\cong}(\Gamma)$ into the
connected components
\begin{equation}
X^{\cong}(\Gamma) = \bigcup_i X^{\cong}(\Gamma;i). \label{35.15}\end{equation}
Let $[p] \in X^{\cong}(\Gamma;i)$ where
$p \in \widetilde{X}$.
Let $E \to \widetilde{X}$
be a $G$ equivariant vector bundle. It is, by
definition, an orbi-bundle $E/G$ on $X$. A {\it single valued
section} of this orbi-bundle is, by definition, a $G$-equivariant
section of $E \to \widetilde{X}$.
\par
We have a $\Gamma$ action on the fiber $E_p$ of our
vector bundle $E$.
In this situation we recall from
Definition \ref{def:dXgamma} :
\begin{eqnarray}
E_p^{\Gamma} & = & \{v \in E_p \mid \forall \gamma \in \Gamma \,\,
\gamma v = v\}. \label{35.18} \\
d(X;\Gamma;i) & = & \dim X^{\cong}(\Gamma;i) - \dim E_p^{\Gamma}.
\label{35.19b}
\end{eqnarray}
In Sections
\ref{section35.2}-\ref{section35.3} we will prove the following:
\begin{prop}\label{35.20} For each $C^0$-section $s$ of the
orbi-bundle $E/G \to X$, there exists a sequence of single valued
piecewise smooth sections $s_{\epsilon}$ converging to $s$ in
$C^0$-sense such that the following holds:
\begin{enumerate}[{\rm (i)}]
\item
$s_{\epsilon}^{-1}(0)$ has a smooth triangulation.
Namely for each simplex the embedding $\Delta \to X$ locally lifts to a map
to $\widetilde{X}$ which is a smooth embedding.
\par
\item $s_{\epsilon}^{-1}(0) \cap
X^{\cong}(\Gamma)$ is a PL manifold such that each simplex is
smoothly embedded into $X^{\cong}(\Gamma)$.
\par
\item If $\Delta$ is a simplex of
$(i)$ whose interior intersects with $X^{\cong}(\Gamma)$,
the intersection of $\Delta$ with $X^{\cong}(\Gamma)$ is
$\Delta$ minus some faces and is smoothly embedded in
$X^{\cong}(\Gamma)$.
\par
\item
$$
\dim s_{\epsilon}^{-1}(0) \cap X^{\cong}(\Gamma;i) = \dim X^{\cong}(\Gamma;i) - \dim E_p^{\Gamma}.
$$
\end{enumerate}
\end{prop}
\begin{rem}\label{35.22}
We note that if $s$ is a single-valued section and if $[p] \in
X^{\cong}(\Gamma)$ then $s(p) \in E_p^{\Gamma}$. So the dimension
given in (iv) is optimal.
\end{rem}
We will use Proposition \ref{35.22} in the proof of Theorem \ref{maintechnicalresult} in Section \ref{section35.4}.

\section{System of tubular neighborhoods}\label{section35.2}
The proof of Proposition \ref{35.20} is closely related to the proof
of existence of a triangulation of the space with Whitney
stratification. (See \cite{Gor78}). Especially, we use the notion of
\emph{a system of tubular neighborhoods} introduced by Mather
\cite{Math73}. Mather used this notion to prove the famous
first isotopy lemma (see 8 in Section II \cite{Math73}). The first isotopy lemma implies that Whitney
stratification has a $C^0$ locally trivial normal cone. (See Theorem
8.3 \cite{Math73}). Note that existence of smooth triangulation of
an orbifold is well-known which is not what we intend to prove. In
order to show that the zero set of our section has a smooth
triangulation, we use various constructions appearing in the proof
of existence of a $C^0$-triangulation on the Whitney stratified
space.
\par
As we explained in Example \ref{normalbundleasloctrivial}, we have a $C^{\infty}$
locally trivial tubular neighborhood in our situation. This property
makes the system of tubular neighborhoods (or the system of normal
cones) in this case carry some properties better than that of
\cite{Math73}.
\par
We follow Section II in \cite{Math73} in the next definition.
\begin{defn}\label{35.23}
A tubular neighborhood of the stratum $X^{\cong}(\Gamma)$ in $X$ is
a quadruple $(\pi_{\Gamma},N_{X^{\cong}(\Gamma)}X,\sigma,\phi)$
satisfying
\begin{enumerate}[{\rm (i)}]
\item
$\pi_{\Gamma}:N_{X^{\cong}(\Gamma)}X \to X^{\cong}(\Gamma)$ is a
vector bundle in the sense of stack. (See Definition \ref{A1.89}.)
\par
\item $\sigma: X^{\cong}(\Gamma) \to \R_+$
is a smooth positive function.
\par
\item $\phi: B_{\sigma}(\Gamma)/\Gamma
\to U^{\cong}(\Gamma)$ is a diffeomorphism
(in the sense of Definition \ref{diffeo}) onto a neighborhood
$U^{\cong}(\Gamma)$ of $X^{\cong}(\Gamma)$ in $X$. Here
$
B_{\sigma}(\Gamma) =
\{ v \in N_{X^{\cong}(\Gamma)}X \mid \Vert v\Vert <
\sigma(\pi_{\Gamma}(v))\}.
$
\end{enumerate}
\end{defn}
We define $ \pi_{\Gamma}: U^{\cong}(\Gamma) \to X^{\cong}(\Gamma) $
as the composition $\pi_{\Gamma} \circ \phi^{-1}$. We also define $
\rho'_{\Gamma}: U^{\cong}(\Gamma) \to  \R $ by
$$
\rho'_{\Gamma}(\phi(v)) = \Vert v\Vert^2.
$$
We call $\rho'_{\Gamma}$ the {\it tubular distance function}.
We note that both  maps are smooth and $\pi_{\Gamma}$ is a
submersion. Moreover the pair
$$
(\pi_{\Gamma},\rho'_{\Gamma}): U^{\cong}(\Gamma) \setminus
X^{\cong}(\Gamma) \to X^{\cong}(\Gamma) \times \R_{>0}
$$
defines
a locally trivial fiber bundle of $C^{\infty}$ class.
\par
We need to adjust them so that they become compatible for
different $\Gamma$'s.
\begin{rem}\label{35.24}
In the situation of \cite{Math73}, the maps
$\pi_{\Gamma}$, $\rho'_{\Gamma}$ are smooth only in the interior of
the stratum.
\end{rem}
\begin{defn}\label{35.25}
{\it A system of tubular neighborhoods} of our stratification
$\{X^{\cong}(\Gamma) \mid \Gamma\}$ is a family
$(\pi_{\Gamma},\rho'_{\Gamma})$ such that
\begin{subequations}\label{35.26}
\begin{eqnarray}
\pi_{\Gamma^{\prime}} \circ \pi_{\Gamma} &=& \pi_{\Gamma^{\prime}}
\label{35.26.1}
\\
\rho'_{\Gamma^{\prime}} \circ \pi_{\Gamma} &=& \rho'_{\Gamma^{\prime}}
\label{35.26.2}
\end{eqnarray}
\end{subequations}
holds for $\Gamma^{\prime} \supset \Gamma$. Here we assume the
equalities $(\ref{35.26.1})$, $(\ref{35.26.2})$ whenever both sides are
defined.
\end{defn}
\begin{prop}\label{35.27} There exists a system of tubular
neighborhoods.
\end{prop}
The proof is actually the same as that of Corollary 6.5 of
\cite{Math73}. Mather proved the existence of a system of tubular
neighborhoods for the space with Whitney stratification. In his case
the situation is less tame than our case since the normal cone
exists only in $C^0$ sense.
(See example \ref{35.3}.)  In our case, the proof is easier since
the normal cone we produce by Lemma \ref{A1.100} is already smooth.
For the sake of completeness, we give the proof of Proposition
\ref{35.27} later in this subsection (Proposition \ref{35.33}).
\par
We next define the notion of a family of lines following Goresky
\cite{Gor78}. For $\epsilon>0$ we put:
\begin{eqnarray}
S^{\Gamma}(\epsilon) &=& \{ p \in U^{\cong}(\Gamma)
\mid \rho_{\Gamma}(p) = \epsilon^2\} \label{35.28}\\
U^{\cong}(\Gamma)(\epsilon) &=&
\{ p \in U^{\cong}(\Gamma)
\mid \rho_{\Gamma}(p) < \epsilon^2\} .
\label{35.29}
\end{eqnarray}
Here we modify $\rho'_{\Gamma}$ to $\rho_{\Gamma}$ in the
following way. (See the lines 2-5 from the bottom of \cite{Gor78} p
193. We warn that our notation $\rho_{\Gamma}$ corresponds to
Goresky's $\rho'_{X}$ and $\rho'_{\Gamma}$ to Goresky's
$\rho_{X}$.) We take $\rho'_{\Gamma}(p) = \Vert
v\Vert^2$ for $p = \phi(v)$ where $\Vert\cdot\Vert$ is an
appropriate norm induced by an inner product. We take a function $f_{\Gamma}: X^{\cong}(\Gamma)
\to \R_+$ that goes to infinity on the boundary. Then we define
$\rho_{\Gamma}(x)=(f_{\Gamma}\circ
\pi_{\Gamma}(x))\rho'_{\Gamma}(x)$.
(Note $f$ is related to $\sigma$ in Definition \ref{35.23} by $\sigma = \epsilon^2/f^2$.)
\par\medskip
\hskip2cm
\epsfbox{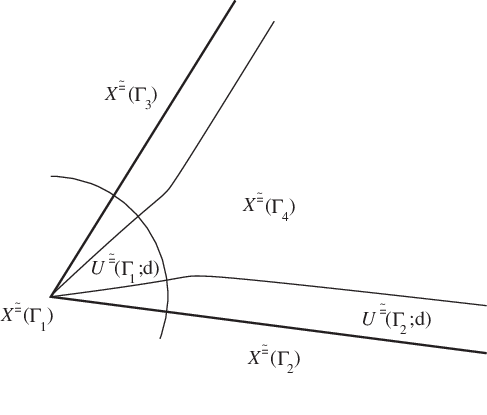}
\par
\centerline{\bf Figure 1}
\bigskip
\begin{defn}\label{35.30}
A family of smooth maps
$$
r_{\Gamma}(\epsilon): U^{\cong}(\Gamma) \setminus X^{\cong}(\Gamma)
\to S^{\Gamma}(\epsilon)
$$
is said to be a {\it family of lines} if the following holds for
$\Gamma' \supset \Gamma$:
\begin{enumerate}[{\rm (i)}]
\item
$r_{\Gamma'}(\epsilon') \circ r_{\Gamma}(\epsilon)
= r_{\Gamma}(\epsilon) \circ r_{\Gamma'}(\epsilon') \in S^{\Gamma}(\epsilon) \cap
S^{\Gamma'}(\epsilon')$ for all $\epsilon',\epsilon > 0$.
\par
\item
$\rho_{\Gamma'} \circ r_{\Gamma}(\epsilon) = \rho_{\Gamma'}$.
\par
\item
$\rho_{\Gamma} \circ r_{\Gamma'}(\epsilon) = \rho_{\Gamma}$.
\par
\item
$\pi_{\Gamma'} \circ r_{\Gamma}(\epsilon) = \pi_{\Gamma'}$.
\par
\item
If $0 < \epsilon < \epsilon' < \delta$ then
$r_{\Gamma}(\epsilon') \circ r_{\Gamma}(\epsilon) = r_{\Gamma}(\epsilon')$.
\par
\item
$
\pi_{\Gamma} \circ r_{\Gamma}(\epsilon) = \pi_{\Gamma}.
$
\par
\item
We define
$$
h_{\Gamma}:  U^{\cong}(\Gamma)(\epsilon) \setminus  X^{\cong}(\Gamma)
\to S^{\Gamma}(\epsilon) \times (0,\epsilon)
$$
by
$$
h_{\Gamma}(p) = (r_{\Gamma}(\epsilon)(p),\sqrt{\rho_{\Gamma}(p)})
$$
and extend it to a map
from  $U^{\cong}(\Gamma;\epsilon)$ to the
mapping cone of
$$
\pi_{\Gamma}\vert_{S^{\Gamma}(\epsilon)}: S^{\Gamma}(\epsilon)  \to
X^{\cong}(\Gamma),
$$
by setting
$h_\Gamma(p) = (p, 0)$ on $X^{\cong}(\Gamma)$. Then $h_{\Gamma}$
becomes a diffeomorphism.
\end{enumerate}
\end{defn}
Note the
mapping cone of
$
\pi_{\Gamma}\vert_{S^{\Gamma}(\epsilon)}: S^{\Gamma}(\epsilon)  \to
X^{\cong}(\Gamma)
$
is identified with
$
\{ p \in U^{\cong}(\Gamma)
\mid \rho_{\Gamma}(p) \le \epsilon^2\}.
$
So it is an orbifold with boundary. (It is a total space of the ball bundle discussed in
Example  \ref{normalbundleasloctrivial}.)
We remark that the diffeomorphism in (vii) is one in the
sense of Definition \ref{diffeo}.
\par\medskip
\hskip1cm
\epsfbox{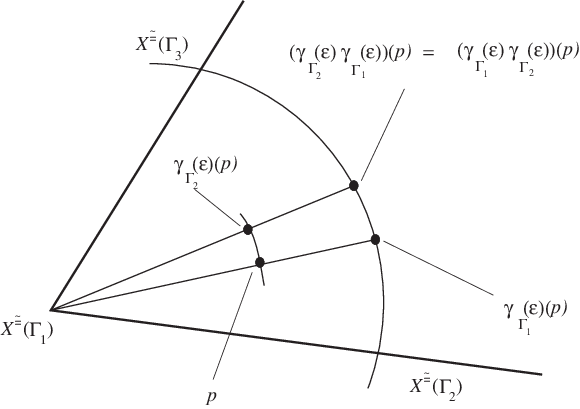}
\par
\centerline{\bf Figure 2}
\par\bigskip
The above definition except (vii) exactly coincides
with that of Goresky \cite{Gor78}. The condition (vii) is stronger
than the corresponding one from \cite{Gor78}. This is because in our
situation the normal cone is smooth and diffeomorphic to a
neighborhood $U^{\cong}(\Gamma)$ of $X^{\cong}(\Gamma)$.
\begin{prop}\label{35.32} There exists a family of lines.
\end{prop}
We can prove Proposition \ref{35.32} in the same way as in \cite{Gor78}
except that we need some extra argument to check (vii). Instead of
working this out, we give a slightly different self-contained proof
of Propositions \ref{35.27} and \ref{35.32} below, which exploits
the special case of orbifolds
(or the space with a stratification whose strata
have locally trivial normal bundles of
$C^{\infty}$-class).
The proof below is simpler than
those by Mather or Goresky. This is because we have already
proved that there exists a normal cone which is $C^\infty$ locally
trivial. For the cases studied by Mather or Goresky, proving
existence of $C^0$ trivial normal cone is one of the main goals of
their study. So our proof here rather goes in the direction opposite
to their study.
\par
Now we will prove the relative versions of Propositions \ref{35.27} and
\ref{35.32} below which include the propositions themselves.
Hereafter we write $(\pi,\rho,r)$ in place of
$\{(\pi_{\Gamma},\rho_{\Gamma},r_{\Gamma}) \mid \Gamma \subset
G\}$ for simplicity. We also write $\pi \circ r = \pi$ etc. in
place of Definition \ref{35.30} (vi) etc. by an abuse of notations.
\par
\begin{prop}\label{35.33} Let $X$ be a global quotient and $K$
a compact subset of $X$. Assume that there exist a system of tubular
neighborhoods $(\pi,\rho)$ and a family of lines $r$ in a
neighborhood $U$ of $K$. Then there exist $\pi$, $\rho$ and $r$ on
$X$ that coincide with the given ones in a neighborhood of $K$
respectively.
\end{prop}
For the proof of Proposition \ref{35.33}, we generalize it to the
following relative version. (Compare this with 6 in Section II \cite{Math73} where a similar procedure of the proof is applied.)
\begin{prop}\label{35.34} Let $X$ be a global quotient and $K$
a compact subset of $X$. Let $U_1, U_4$ be open subsets of $X$ such that
$U_4 \supset K$ and $U_1 \supset \overline{U}_4$. Assume that there
exist a system of tubular neighborhoods $(\pi,\rho)$ and a family of lines $r$ on $U_1$.
We also assume that there exists a locally
trivial fiber bundle of $C^{\infty}$ class
$pr_N: X \setminus U_4 \to N$ where $N$ is a
manifold. We assume   $pr_N \circ \pi = pr_N$, and $pr_N \circ r =
pr_N$ on $U_1 \setminus U_4$.
\par
Then there exist open sets $U_2$, $U_3$ with $U_j \supset
\overline{U}_{j+1}$ for $j=1,2,3$ and there exist $\pi$, $\rho$
and $r$ on $X$ which coincide with the given ones on a neighborhood of $K$ in
$\overline U_4$. In
addition, $pr_N \circ r = pr_N$ and $pr_N \circ \pi = pr_N$ hold on
$U_2 \setminus U_3$.
\end{prop}
Again we write $pr_N \circ \pi = pr_N$ etc. in place of $pr_N
\circ \pi_{\Gamma} = pr_N$ etc. by an abuse of notations.
\begin{proof}
By shrinking $U_1$ if necessary,
we may assume that $X^{\cong}(\Gamma;i) \subset U_4$
if and only if $X^{\cong}(\Gamma;i) \subset U_1$ for each
$\Gamma$. We put
\begin{eqnarray}
D & = & \dim X - \inf\{\dim X^{\cong}(\Gamma;i) \mid
\text{$X^{\cong}(\Gamma;i)$ is not contained in $U_1$}\}\nonumber
\\ & = & \sup\{\operatorname{codim} X^{\cong}(\Gamma;i) \mid
\text{$X^{\cong}(\Gamma;i)$ is not contained in $U_1$}\} .
\label{35.35}
\end{eqnarray}
The proof is given by an induction over $D$. If $D$
is $0$, there is nothing to prove.
\par
We assume that Proposition \ref{35.34} is proved when
$(\ref{35.35})$ is $D-1$ or smaller and prove the case of $D$. Let
$X^{\cong}(\Gamma;i)$ be a stratum of codimension $D$. Since this is
the stratum of smallest dimension in $X\setminus U_4$, it follows
that it is a smooth manifold outside $U_4$. We put
$X_0^{\cong}(\Gamma;i) = X^{\cong}(\Gamma;i)\setminus U_{3.5}$. Here
$U_{3.5}$ is a neighborhood  of $K$, which is slightly bigger than
$U_4$.
\par
By Example \ref{normalbundleasloctrivial},
we have the locally
trivial fiber bundle of $C^{\infty}$ class
\begin{equation}
\partial N_{X_0^{\cong}(\Gamma;i)}X/\Gamma
\longrightarrow X_0^{\cong}(\Gamma;i).
\label{35.36}
\end{equation}
By Lemma \ref{A1.100} (see also
Mather's version of tubular neighborhood theorem
in p.213 \cite{Math73})
we may take the projection $\pi_{\Gamma}:
N_{X_0^{\cong}(\Gamma;i)}X/\Gamma \to X_0^{\cong}(\Gamma;i)$ so that
${pr}_N \circ \pi_\Gamma = pr_N$, by choosing the Riemannian
metric we use to prove Lemma \ref{A1.100} so that each of the fibers
of ${pr}_N$ is totally geodesic.
\par
We will now modify the tubular neighborhood and the family of
lines on $(U_1 \setminus U_{3.5}) \,\cap\, \partial
N_{X_0^{\cong}(\Gamma;i)}X/\Gamma$ (which is given by assumption)
so that we can apply the induction hypothesis to $\partial
N_{X_0^{\cong}(\Gamma;i)}X/\Gamma$ and the fibration
\begin{equation}\label{pigamma}
\pi_{\Gamma} :
\partial N_{X_0^{\cong}(\Gamma;i)}X/\Gamma
\longrightarrow X_0^{\cong}(\Gamma;i).
\end{equation}
Note when we apply the induction hypothesis
$X_0^{\cong}(\Gamma;i)$ plays the role
of $N$ and
$\pi_{\Gamma}$ plays the role of
$pr_{N}$.
Moreover $U_1 \cap
\partial N_{X_0^{\cong}(\Gamma;i)}X/\Gamma$
plays the role of $U_1$.
\par
To apply the induction hypothesis to
(\ref{pigamma}), we need to check
that the assumption is satisfied.
In particular we need to find a tubular neighborhood and a
family of lines on
 $U_1 \cap
\partial N_{X_0^{\cong}(\Gamma;i)}X/\Gamma$
that is compatible with $\pi_{\Gamma}$.
\par
By the assumption of Proposition \ref{35.34},
there exist a tubular neighborhood and a
family of lines $(\pi',\rho',r')$
on the set $U_1$ that is compatible with $pr_N$.
It induces   a tubular neighborhood and a
family of lines $(\pi',\rho',r')$ on the set $\partial
N_{X_0^{\cong}(\Gamma;i)}X/\Gamma \cap U_1$.
This is compatible with $pr_N$
but not necessarily compatible with $\pi_{\Gamma}$.
We claim that we can modify  $\pi_{\Gamma}$
so that it is compatible with $(\pi',\rho',r')$ on the set $\partial
N_{X_0^{\cong}(\Gamma;i)}X/\Gamma \cap U_1$.
\par
Since there exists $(\pi',\rho',r')$
on the set $U_1$, there
exists
$$\pi'_{\Gamma}: \partial N_{X_0^{\cong}(\Gamma;i)}X/\Gamma
\to X_0^{\cong}(\Gamma;i)
$$
such that if $\Gamma \supset
\Gamma'$, the map $\pi'_{\Gamma}$ is consistent with
$\pi'_{\Gamma'}$, $\rho'_{\Gamma'}$ and
$r'_{\Gamma'}$ in the sense of (\ref{35.26}) and Definition \ref{35.30} (iv).
\par
We note that $\pi'_{\Gamma}$ may not coincide with $\pi_{\Gamma}$
given by Lemma \ref{A1.100}. But we can modify and glue them as
follows: The difference between two projections ($\pi'_{\Gamma}$
above and $\pi_{\Gamma}$) can be chosen to be arbitrarily small (in
$C^1$ sense), by taking the tubular neighborhood small. Then we can
use the minimal geodesic of a Riemannian metric on
$X_0^{\cong}(\Gamma;i)$, to find an isotopy between them. Hence by a
standard argument we can glue them.
We thus proved the claim.
\par
Thus we can apply our induction hypothesis to
(\ref{pigamma})
and obtain the system $(\pi, r, \rho)$ on $\partial
N_{X_0^{\cong}(\Gamma;i)}X/\Gamma$. Since
$N_{X_0^{\cong}(\Gamma;i)}X/\Gamma \cong U(X_0^{\cong}(\Gamma)) $
is a cone of $
\partial N_{X_0^{\cong}(\Gamma;i)}X/\Gamma$, the
system $(\pi, r, \rho)$ on $\partial
N_{X_0^{\cong}(\Gamma;i)}X/\Gamma$ induces one on
$N_{X_0^{\cong}(\Gamma;i)}X/\Gamma$ in an obvious way. It commutes
with the projection $U(X_0^{\cong}(\Gamma)) \to N$, since
$U(X_0^{\cong}(\Gamma)) \to N$ factors through $\pi_{\Gamma}$.
\par
Recall that we have $(\pi, r, \rho)$ on $U_1$ by assumption. On
$U_1 \cap U(X_0^{\cong}(\Gamma))$, this system may not coincide
with the one we have just constructed above. We now explain how we adjust
this system to carry out the gluing process.
\par
We first note that the projection $\pi_{\Gamma}: \partial
N_{X_0^{\cong}(\Gamma;i)}X/\Gamma \to X_0^{\cong}(\Gamma;i)$
coincides for the two systems on $U_1$ since they are already
arranged so when we apply the induction hypothesis above. Since both
systems on $U_1$ are the cone of the same system $(\pi,\rho,r)$ on
$\partial N_{X_0^{\cong}(\Gamma;i)}X/\Gamma$ by the same map, it
follows that they are the same as an abstract structure.
\par
However, the diffeomorphism from the cone of $\partial
N_{X_0^{\cong}(\Gamma;i)}X/\Gamma$ to
$U(X_0^{\cong}(\Gamma;i))$
(which exists by (vii)) may
not coincide. (Note this map is defined by $r$, the family of
lines of each of the structures.)
But we can show that they are isotopic by the same method as
before. Namely, we go to a branched covering of the cone of
$\partial
N_{X_0^{\cong}(\Gamma;i)}X/\Gamma$ and join the two
diffeomorphisms by the minimal geodesic of a $G$-equivariant
Riemannian metric that is totally geodesic along the fiber of
$pr_N$. Therefore we can glue them by the standard method.
\par
Now we have extended the systems $(\pi,\rho,r)$ to a neighborhood
of $X^{\cong}(\Gamma;i)$. We repeat the same construction for each
$X^{\cong}(\Gamma;i)$ with $\dim X^{\cong}(\Gamma;i) = \dim X -D$.
\par
We thus reduce the problem to the case when $D$ is strictly
smaller. The proof of Proposition \ref{35.34} is now finished by
induction. \end{proof}
\section {Single valued piecewise smooth section of
orbi-bundle: proof}\label{section35.3}
\begin{proof}[Proof of Proposition \ref{35.20}]
Let us fix a sufficiently small $d>0$ and put
\begin{equation}
\text{\rm Int} X^{d} (\Gamma)
= X^{\cong} (\Gamma)  \setminus \bigcup_{\Gamma' \supset \Gamma}
\text{\rm Int}\,\,U^{\cong}(\Gamma')(d).
\label{35.37}
\end{equation}
By definition of  system of tubular neighborhoods
we observe the following:
\begin{equation}
U^{\cong}(\Gamma_1) \cap U^{\cong}(\Gamma_2) \ne \emptyset \quad
\Rightarrow \quad \Gamma_1 \subset \Gamma_2 \,\,\text{or} \,\,
\Gamma_2 \subset \Gamma_1. \label{35.38}
\end{equation}
It follows from (\ref{35.38})
that $\text{\rm Int} X^{d} (\Gamma)$ is a smooth manifold with
corners. The codimension $k$ corner of $\text{\rm Int} X^{d}
(\Gamma)$ is a union of
\begin{equation}
X(\Gamma;\Gamma_1, \Gamma_2 ,\dots , \Gamma_k) = X^{\cong}(\Gamma) \cap \bigcap_{i=1}^k S^{\Gamma_i}(d)
\label{35.39}\end{equation}
where
\begin{equation}
\Gamma_1 \supset \Gamma_2 \supset \dots \supset \Gamma_k \supset \Gamma.
\label{35.40}\end{equation}
\par\medskip
\hskip2.5cm
\epsfbox{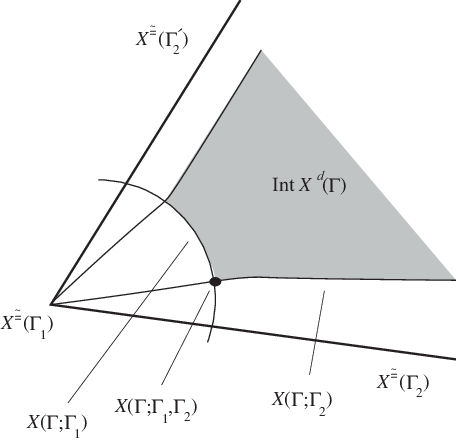}
\par
\centerline{\bf Figure 3}
\par\bigskip
Later in this section, we will first define our section
$s_{\epsilon}$ on $\bigcup_{\Gamma}\text{\rm Int} X^{d}
(\Gamma)$ and then extend it so that its zero set is a cone with
respect to the family of lines. Thus our proof is an analog to the
proof given in \S3-5 \cite{Gor78}.
\par
Let $E_p^\Gamma$ be as in (\ref{35.18pre}).
We put:
\begin{equation}
E_p = E_p^\Gamma \oplus E_p^{\bot} \label{35.41}
\end{equation}
where $E_p^{\bot}$ is the complement of $E_p^{\Gamma}$ in $E_p$.
\par
On $X(\Gamma;\Gamma_1, \Gamma_2 ,\dots , \Gamma_k)$
we define $E_p(\Gamma)$ inductively on $\# \Gamma$ and decompose $E_p^\Gamma$ into
\begin{equation}\label{35.41-2}
E_p^{\Gamma} = \bigoplus_{i=1}^k E_p(\Gamma_i) \oplus E_p(\Gamma)
\end{equation}
as follows.
\par
If $\Gamma$ is maximal, we set $E_p(\Gamma)=E_{p}^{\Gamma}$
on $X^{\cong}(\Gamma)$.
Using local triviality of $E(\Gamma)$, we
extend our subbundle $E(\Gamma)$ to the neighborhood
$U^{\cong}(\Gamma)(d)$ for a sufficiently small $d$  so that
$$
E_p(\Gamma)\subset E_p^{\Gamma'}
$$
is satisfied for $p \in U^{\cong}(\Gamma)(d) \cap X^{\cong}(\Gamma')$,
and
$\Gamma \supset \Gamma'$
as follows:
\par
We take a $\Gamma$-invariant
connection $\nabla$ of $E$ on $U^{\cong}(\Gamma)(d)$ so that each of
$E^{\Gamma}$ is a totally geodesic subbundle and that the curvature
of $\nabla$ is zero on each fiber of $U^{\cong}(\Gamma)(d) \to
X^{\cong}(\Gamma)$. Then we can use the parallel transport with
respect to $\nabla$ along the path contained in the fiber of
$U^{\cong}(\Gamma)(d) \to X^{\cong}(\Gamma)$ to extend $E(\Gamma)$ to
$U^{\cong}(\Gamma)(d)$.
\par
We next consider $p \in X(\Gamma;\Gamma_1)$. We may assume that
$E_p(\Gamma_1)$ is defined. Then we define $E_p(\Gamma)$ as the
orthonormal complement of $E_p(\Gamma_1)$ in $E_p^{\Gamma}$. We
extend them to its neighborhood. We thus obtain
$$
E_p \cong E_p(\Gamma_1)\oplus E_p(\Gamma) \oplus E_p^{\bot}.
$$
We can continue by a downward induction on $\# \Gamma$, $k$ and
obtain the decomposition (\ref{35.41-2}).
\par
We use (\ref{35.41-2})
to perform our construction of $s_{\epsilon}$. We also need the
following lemma.
\par
\begin{lem}\label{35.42}
Let $f: M \to N$ be a locally
trivial fiber bundle of $C^{\infty}$ class  between smooth manifolds and
$F$ a vector bundle on $M$. We fix a smooth triangulation of $N$.
Let $s$ be a section of $F$. Then there exists a family
$s^{\epsilon}$ of piecewise smooth sections of $F$ such that
\begin{enumerate}[{\rm (i)}]
\item $s^{\e}$ converges to $s$ in $C^0$
topology.
\par
\item $s^{\e}$ is of general position to $0$.
\par
\item
$f: (s^{\e})^{-1}(0) \to N$ is piecewise linear with respect to some
smooth triangulation of $(s^{\e})^{-1}(0)$ and a subdivision of
the given triangulation of $N$.
\end{enumerate}
\end{lem}
\begin{proof} Since $f$ is a locally
trivial fiber bundle of $C^{\infty}$ class,  we may choose a
triangulation of $M$ and a subdivision of the given one on $N$ with
respect to which $f$ is piecewise linear. (See \cite{Mun66},
\cite{Whi40}.) In other words, there exist simplicial complexes
$K_M$, $K_N$ and homeomorphisms $i_M: \vert K_M\vert \to M$, $i_N:
\vert K_N\vert \to N$ with the following properties:
\begin{enumerate}[{\rm (i)}]
\item The restrictions of the homeomorphisms $i_M$ and
$i_N$ to each simplex are diffeomorphisms onto their images.
\par
\item
$i_N^{-1} \circ f \circ i_M$ is induced
by a simplicial map.
Namely it sends a simplex of  $K_M$ to a simplex
of $K_N$ and is affine on each simplex.
\end{enumerate}
We next take a smooth triangulation of the total space $F$ of our
vector bundle so that the projection $F \to M$ is piecewise linear.
By taking an appropriate subdivision of the simplicial
decomposition, we may approximate our section $s$ by a section
$s^{\e}: M \to F$ which is piecewise linear, $C^0$ close to $s$, and
of general position to the zero section. (Existence of such $s^{\e}$
is a standard result of piecewise linear topology. See, for example,
\cite{Hud69}.) Then (i) and (ii) are satisfied. Since $s^{\e}$ is
piecewise linear, which is affine on each simplex, it follows that
the intersection of $(s^{\e})^{-1}(0)$ with each simplex is affine.
Hence we can find a subdivision of $K_M$ and $K_N$ such that
$(s^{\e})^{-1}(0)$ is a subcomplex and the restriction of $f$ to
$(s^{\e})^{-1}(0)$ is piecewise linear.
\end{proof}
We note that Lemma \ref{35.42} (iii) implies that the mapping cone
of $f: (s^{\e})^{-1}(0) \to N$ has a smooth triangulation.
\par
Now we start the construction of our section $s_{\epsilon}$.
We will put
\begin{equation}
s_{\epsilon} = \bigoplus_i s_{\epsilon}^{\Gamma_i} \oplus
s_{\epsilon}^{\Gamma} \oplus 0
\label{35.45}\end{equation}
according to our
decomposition \eqref{35.41}, \eqref{35.41-2}. Note the $E_p^{\bot}$-component is
necessarily zero because of the $G$-invariance. (In other words it
is zero since $s_{\epsilon}$ is single-valued.)
\par
We will construct $s_{\epsilon}^{\Gamma}$ by the downward
induction over the order of $\Gamma$.
\par\medskip
Let $\Gamma$ be maximal. We consider the vector bundle $E^{\Gamma}
\to X^{\cong}(\Gamma)$. We note that this is a vector bundle on
a manifold and is not only an orbi-bundle.
So for each given section $s$
we can take a smooth
section $s_{\epsilon}^{\Gamma}$ which is transversal to $0$ and $C^0$ close to $s$.
We extend $s_{\epsilon}^{\Gamma}$ to a section of $E(\Gamma)$ on $U^{\cong}(\Gamma)$ so
that it is covariantly constant along the fiber of $\pi_{\Gamma}$.
\par
We next
consider the case of $\Gamma$ that is not necessarily maximal but
has a property that there is no $\Gamma_1 \supset \Gamma_2 \supset \Gamma$ with
$X(\Gamma;\Gamma_1,\Gamma_2) \ne \emptyset$.
We consider $\text{\rm Int} X^{d}(\Gamma) \cap S^{\Gamma_1}(d)$. By assumption
\begin{equation}
\partial\,\, \text{\rm Int}\, X^d(\Gamma) =
\bigcup_{\Gamma_1 \supset \Gamma} X(\Gamma;\Gamma_1),
\label{35.47}
\end{equation}
where the right hand side are disjoint each other and $\Gamma_1$ are maximal.
By induction hypothesis we have defined
$s^{\Gamma_1}_{\epsilon}$ already.
We now apply Lemma \ref{35.42} to
$$
\pi_{\Gamma}: (s_{\epsilon}^{\Gamma_1})^{-1}(0) \cap X(\Gamma;\Gamma_1) \to
(s_{\epsilon}^{\Gamma_1})^{-1}(0) \cap X(\Gamma_1)
$$
and the bundle $E(\Gamma) \to (s_{\epsilon}^{\Gamma_1})^{-1}(0) \cap
X(\Gamma;\Gamma_1) $.
We then obtain $s_{\epsilon}^{\Gamma}$ on
$(s_{\epsilon}^{\Gamma_1})^{-1}(0) \cap X(\Gamma;\Gamma_1)$. We extend it to
$X(\Gamma;\Gamma_1) \setminus (s_{\epsilon}^{\Gamma_1})^{-1}(0)$ in an
arbitrary way. (It does not matter how we extend since it will not
change the zero set.)
\par
We have thus defined
$s_{\epsilon} = s_{\epsilon}^{\Gamma_1} \oplus
s_{\epsilon}^{\Gamma}$ on (\ref{35.47}).
Note that on
$$
\text{\rm Int}\, X^d(\Gamma) \setminus \partial\,\,
\text{\rm Int}\, X^d(\Gamma)
$$
$E$ is decomposed to $E^{\Gamma} = E(\Gamma)$ and $E^{\bot}$. (We
remark that we decompose $E^{\Gamma}$ to $E(\Gamma_1) \oplus
E(\Gamma)$ only at their boundaries.) On the boundary we defined
the section of $E^{\Gamma} \cong E(\Gamma) \oplus E(\Gamma_1)$
already which is of general position relative to zero. We can then
extend it to $\text{\rm Int}\, X^d(\Gamma)$ so that it is of
general position to zero. (Note $E^{\bot}$ component is
necessarily zero again.) We then extend this to its neighborhood
so that it is covariantly constant in each of $\pi_{\Gamma}$ fibers.
\par
Now the main induction step goes as follows: Assuming $s_{\epsilon}^{\Gamma'}$
is defined for $\#\Gamma' > \# \Gamma$, we consider $\Gamma$. Take
the decomposition
\begin{equation}
\partial\, \text{\rm Int}\, X^d(\Gamma)
= \bigcup X(\Gamma;\Gamma_1, \Gamma_2 ,\dots , \Gamma_{k}).
\label{35.48}\end{equation}
We will define $s_{\epsilon}^{\Gamma}$ on $X(\Gamma;\Gamma_1,
\Gamma_2 ,\dots , \Gamma_{k})$ by a downward induction on $k$.
\par
We consider a chain of isotropy groups $\Gamma_1,\dots,\Gamma_k$
given as in (\ref{35.40}) for
which $X(\Gamma;\Gamma_1, \Gamma_2 ,\dots
, \Gamma_k)$ is nonempty. Let $k$ be maximal among such choices.
We now apply Lemma \ref{35.42} to
\begin{equation}\aligned
\pi_{\Gamma_k} &: (s_{\epsilon}^{\Gamma_1} \oplus \cdots \oplus s_{\epsilon}^{\Gamma_k})^{-1}(0) \cap X(\Gamma;\Gamma_1,
\Gamma_2 ,\dots ,
\Gamma_{k})\\ &\to (s_{\epsilon}^{\Gamma_1} \oplus \cdots \oplus
s_{\epsilon}^{\Gamma_k})^{-1}(0)
\cap X(\Gamma_k;\Gamma_1, \Gamma_2 ,\dots , \Gamma_{k-1}),
\endaligned
\label{35.49}\end{equation}
and $E(\Gamma) \to X(\Gamma;\Gamma_1, \Gamma_2
,\dots , \Gamma_{k})$. Here we note that the well-definedness
of (\ref{35.49}) is a consequence of compatibilities of $\pi$ and $r$
stated in Definitions \ref{35.25} and  \ref{35.30}.
\par
We thus obtain $s_{\epsilon}^{\Gamma}$ on
$(s_{\epsilon}^{\Gamma_1} \oplus \cdots \oplus
s_{\epsilon}^{\Gamma_k})^{-1}(0) \cap X(\Gamma;\Gamma_1, \Gamma_2 ,\dots ,
\Gamma_{k})$ which we extend to $X(\Gamma;\Gamma_1, \Gamma_2
,\dots , \Gamma_{k})$ in an arbitrary way.
\par
Now we can extend $s_{\epsilon}^{\Gamma}$ to various  $X(\Gamma;\Gamma_1,
\Gamma_2 ,\dots , \Gamma_{\ell})$ by a downward induction on
$\ell$ using an appropriate relative version of  Lemma \ref{35.42}.
Namely we assume $s_{\epsilon}^{\Gamma}$ is defined on $X(\Gamma;\Gamma_1,
\Gamma_2 ,\dots , \Gamma_{k})$ for $k > \ell$ then $s_{\epsilon}^{\Gamma}$
is defined on $\partial X(\Gamma;\Gamma_1, \Gamma_2 ,\dots ,
\Gamma_{\ell})$. Then we extend it to $X(\Gamma;\Gamma_1, \Gamma_2
,\dots , \Gamma_{\ell})$ by applying a relative version of Lemma
\ref{35.42} to
$$\aligned
&(s_{\epsilon}^{\Gamma_1} \oplus \cdots \oplus
s_{\epsilon}^{\Gamma_{\ell}})^{-1}(0)
\cap X(\Gamma;\Gamma_1, \Gamma_2 ,\dots ,
\Gamma_{\ell}) \\
&\to
(s_{\epsilon}^{\Gamma_1} \oplus \cdots \oplus
s_{\epsilon}^{\Gamma_{\ell}})^{-1}(0)
\cap X(\Gamma_{\ell};\Gamma_1, \Gamma_2
,\dots , \Gamma_{\ell-1})
\endaligned$$
and a bundle $E(\Gamma)$.
\par
Thus we have constructed $s_{\epsilon}^{\Gamma}$ on (\ref{35.48}). Again we extend
$s_{\epsilon}^{\Gamma}$ to $\text{\rm Int}\, X^d(\Gamma)$ so that it is of
general position to $0$.
\par\medskip
Therefore we have constructed $s_{\epsilon}$ on the union
\begin{equation}
\bigcup_{\Gamma}\text{\rm Int}\, X^d(\Gamma).
\label{35.50}\end{equation} We will extend this to $X$ as follows:
We first observe
$$
X \setminus \bigcup_{\Gamma}\text{\rm Int}\, X^d(\Gamma) =
\bigcup_{\Gamma} \left((U^{\cong}(\Gamma)(d) \cap
\pi_{\Gamma}^{-1}(\text{\rm Int}\, X^d(\Gamma) )) \setminus
X^{\cong}(\Gamma) \right).
$$
We will define $s_{\epsilon}$ on
$$
\left((U^{\cong}(\Gamma)(d) \cap
\pi_{\Gamma}^{-1}(\text{\rm Int}\, X^d(\Gamma) )) \setminus
X^{\cong}(\Gamma) \right)
$$
by an upward induction on $\#\Gamma$.
Let
$$
p \in (U^{\cong}(\Gamma)(d) \cap \pi_{\Gamma}^{-1}(\text{\rm Int}\,
X^d(\Gamma) )) \setminus X^{\cong}(\Gamma)
$$
and consider
\begin{equation}\label{imageofr}
r_{\Gamma}(d)(p) \in S^{\Gamma}(d) \subseteq \bigcup_{\Gamma' \subsetneq
 \Gamma}
U^{\cong}(\Gamma')(d) \cap \pi_{\Gamma'}^{-1}(\text{\rm Int}\, X^d(\Gamma'))
\end{equation}
and
$$
\pi_{\Gamma}(p) \in \text{\rm Int}\, X^d(\Gamma).
$$
We note that the right hand side of (\ref{imageofr}) contains
the case $\Gamma' =\{1\}$. In that case
$U^{\cong}(\Gamma')(d) \cap \pi_{\Gamma'}^{-1}(\text{\rm Int}\, X^d(\Gamma'))
= \text{\rm Int}\, X^d(\Gamma')$.
\par\medskip
\hskip2cm
\epsfbox{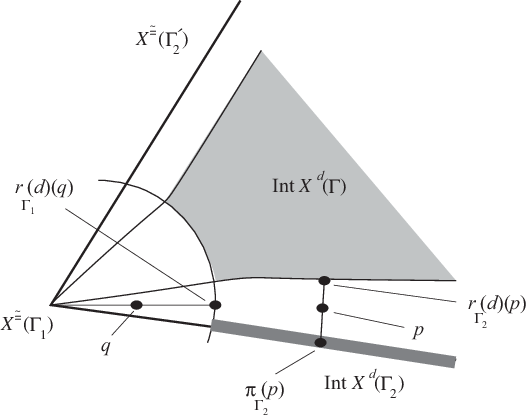}
\par
\centerline{\bf Figure 4}
\par\bigskip
By construction, $s_{\epsilon}^{\Gamma}(\pi_{\Gamma}(p))$ coincides with
$s_{\epsilon}^{\Gamma}(r_{\Gamma}(d)(p))$ under the identification
$$
E_{\pi_{\Gamma}(p)}(\Gamma) \cong E_{r_{\Gamma}(d)(p)}(\Gamma).
$$
We now put
\begin{equation}
s_{\epsilon}(p) = s_{\epsilon}^{\Gamma}(\pi_{\Gamma}(p))
+
\exp\left(\frac{1}{d^2} - \frac{1}{\rho_{\Gamma}(p)}\right)\sum_{\Gamma' \subsetneq \Gamma}
s_{\epsilon}^{\Gamma'}(r_{\Gamma}(d)(p)).
\label{35.51}
\end{equation}
This section coincides with previously defined one when
$\sqrt{\rho_{\Gamma}(p)} = 0$ or $d$.
(We set $\exp(1/d^2 - 1/0) = 0$ as definition.)
Hence
it defines a piecewise smooth section on $X$.
\par
We remark that by definition
$$
s_{\epsilon}^{-1}(0) \cap (U^{\cong}(\Gamma)(d) \cap \pi_{\Gamma}^{-1}(\text{\rm Int}\, X^d(\Gamma) ))
$$
is the cone of the map
$$
\pi_{\Gamma}: s_{\epsilon}^{-1}(0) \cap S^{\Gamma}(d) \to
X^{\cong}(\Gamma).
$$
Since we constructed our section applying Lemma \ref{35.42}
repeatedly, it follows that this cone has a smooth triangulation.
(We use Definition \ref{35.30} (vii) and the argument of section
3,4 \cite{Gor78} for this.) On (\ref{35.50}) we have a
transversality and hence Proposition \ref{35.20} (ii) holds.
Proposition \ref{35.20} (iii) and Proposition \ref{35.20} (iv) are
also obvious from construction. The proof of Proposition \ref{35.20}
is now complete.
\end{proof}

\section{Single valued perturbation of a space with Kuranishi structure: proof of Theorem \ref{maintechnicalresult}}
\label{section35.4}
In this section we prove Theorem \ref{maintechnicalresult}.
For the proof we need the following relative version of
Proposition \ref{35.20}.
\begin{defn}\label{normalcondef}
A piecewise smooth single-valued section of an orbi-bundle $E/G \to
X = M/G$ is said to be {\it normally conical} if the following holds:
\begin{enumerate}[{\rm (i)}]
\item There is a decomposition of $X= M/G$ into
$$
\bigcup_{\Gamma} \text{\rm Int} X^{d}(\Gamma) \cup
\bigcup_{\Gamma} (X^{\cong}(\Gamma) \setminus \text{\rm Int}
X^{d}(\Gamma))
$$
as in (\ref{35.37}).
\par
\item On $\text{\rm Int} X^{d}(\Gamma)$ the
$E^{\bot}$-component $s$ is of general position to $0$. (The
$E^{\Gamma}$ component is necessarily $0$.)
\par
\item On $ X^{\cong}(\Gamma) \setminus \text{\rm Int}
X^{d}(\Gamma)$, the section $s$ is given by (\ref{35.51}).
\end{enumerate}
\end{defn}
\begin{prop}\label{35.20'} Let $X$ be a global quotient, $K$
a compact subset of $X$ and $U$ a neighborhood
of $K$. Let $s$ be a
$C^0$-section of the orbi-bundle $E/G \to X$. We assume that $s$
satisfies Proposition {\rm \ref{35.20} (i)-(iv)} on $U$ and is normally conical in the above
sense. Then there exists a sequence of single-valued piecewise
smooth sections $s_{\epsilon}$ converging to $s$ in $C^0$ sense
satisfying Proposition {\rm \ref{35.20} (i)-(iv)} such that $s_{\epsilon} = s$ on $K$.
\end{prop}
The proof is the same as Proposition \ref{35.20} and is omitted.
\begin{proof}[Proof of Theorem \ref{maintechnicalresult}] The proof is by induction on $p \in P$ with respect
to the partial order $<$.
(See Lemma \ref{A1.11} for the finite set $P$ and  the partial order.)
If $p$ is minimal, we apply Proposition \ref{35.20} to
obtain $s'_p$. Assume that we have $s'_q$ for every $q<p$. We consider  $
s'_q$ and the image $\phi_{pq}(V_{pq})$. We restrict
$s'_q$ on the image $\phi_{pq}(V_{pq})$ and use
the embedding $\hat\phi_{pq}$ to obtain a section of
$E_{q}\vert_{\phi_{pq}(V_{pq})} \to V_{pq}$. We can extend it to
its neighborhood, so that the compatibility in the sense of
Definition \ref{A1.21} is satisfied.
\par
We note that the required properties (i) -
(ii) are satisfied on the tubular neighborhood $N_{\phi_{pq}(V_{pq})}$ if it is
satisfied by $s'_q$.
\par
Now we can use  Propositions \ref{35.20'} to obtain the section $s'_p$.  The proof of Theorem \ref{maintechnicalresult} is complete.
\end{proof}
\section{Moduli space of pseudo-holomorphic discs: review}\label{moduli}
We will apply Theorem \ref{maintechnicalresult}
to the moduli space of pseudo-holomorphic discs.
We review the definition and basic facts on the moduli space in this section.
\par
Let $(M,\omega)$ be a symplectic manifold and
$L$ a Lagrangian submanifold of $M$. We take a
compatible almost complex structure $J$ on $M$.
Let $\beta \in H_2(M,L;\Z)$.
We consider $(v;z_0,\dots,z_k)$ such that
\begin{enumerate}
\item $v: (D^2,\partial D^2) \to (M,L)$ is $J$-holomorphic.
\item $z_i \in \partial D^2$.  $z_i \ne z_j$ for $i\ne j$.
\item  $(z_0,\dots,z_k)$ respects the counter clock-wise cyclic order of
$\partial D^2$.
\item The homology class $v_*([D^2]) \in H_2(M,L)$ is $\beta$.
\end{enumerate}
\smallskip
We say $(v;z_0,\dots,z_k)$ and $(v';z'_0,\dots,z'_k)$ are equivalent if
there exists a biholomorphic map $\varphi: D^2 \to D^2$ such that
$\varphi(z'_i) = z_i$ and $v\circ \varphi = v'$. The set of the
equivalence classes is written as $\MM_{k+1}^{\text{\rm
main},\text{\rm reg}}(\beta)$.
We can compactify it as follows. We consider a pair $((\Sigma,\vec
z),v)$ where $v:\Sigma \to M$, $\Sigma$ is a bordered Riemann
surface of genus zero, $\vec z = (z_0,\dots,z_k)$ and $z_i \in
\partial \Sigma$ satisfies (2) and (3) above. We also assume $v$ is
a $J$-holomorphic map of homology class $\beta$.
\par
Moreover we assume that $((\Sigma,\vec z),v)$ is {\it stable}:
namely there exist only a finite number of biholomorphic maps
$\varphi: \Sigma \to \Sigma$ such that $\varphi(z_i) = z_i$ and $v
\circ \varphi = v$.
We say $((\Sigma,\vec z),v) \sim ((\Sigma',\vec z'),v')$ if there
exists a biholomorphic map $\varphi: \Sigma' \to \Sigma$ such that
$\varphi(z'_i) = z_i$ and $v\circ \varphi = v'$.
We denote by $\MM_{k+1}^{\text{\rm main}}(\beta)$ the set of $\sim$
equivalence classes of such $((\Sigma,(z_0,\dots,z_k)),v)$. (See
Subsection 2.1.2 \cite{fooo-book1} for detail.)
\par
We define a map $ ev = (ev_0,\dots,ev_k): \MM_{k+1}^{\text{\rm
main}}(\beta) \to L^k $ by
$$
ev_i((\Sigma,(z_0,\dots,z_k)),v)  = v(z_i).
$$
\begin{theorem}\label{modulimain}
We can define a topology on $\MM_{k+1}^{\text{\rm main}}(\beta)$
with respect to which it is compact and Hausdorff.
\par
$\MM_{k+1}^{\text{\rm main}}(\beta)$ has a Kuranishi structure with
corner. Its (virtual) dimension is given by:
\begin{equation}\label{dimformula}
\dim \MM_{k+1}^{\text{\rm main}}(\beta) = n + \mu_L(\beta) + k-2.
\end{equation}
Here $\mu_L(\beta)$ is the Maslov index of $\beta$.
Moreover
$\MM_{k+1}^{\text{\rm main}}(\beta)$ has a tangent space.
If $L$ is relatively spin, $\MM_{k+1}^{\text{\rm main}}(\beta)$
is orientable. The orientation is determined by the relative spin structure.
\par
The evaluation map $ev$ induces a strongly smooth and weakly submersive map.
\end{theorem}
This theorem is proved in Section 7.1 \cite{fooo-book2}.
\par
Let $P_i$ ($i=1,\dots,k$) be simplices and $f_i: P_i \to L$
smooth maps. We define:
\begin{equation}\label{modulifiberprod}
\MM_{k+1}^{\text{\rm main}}(\beta;P_1,\dots,P_k)
= \MM_{k+1}^{\text{\rm main}}(\beta)\,\, {}_{(ev_1,\dots,ev_k)}\!\times_{(f_1,\dots,f_k)}
(P_1\times \cdots \times P_k).
\end{equation}
(Here and hereafter we write $P_i$ in place of the singular chain
$(P_i,f_i)$.) Since $ev$ is weakly submersive, we can use Lemma A1.39
\cite{fooo-book2} to define a Kuranishi structure on
$\MM_{k+1}^{\text{\rm main}}(\beta;P_1,\dots,P_k)$. We will define
the filtered $A_{\infty}$ structure $\mathfrak m_k$ by
\begin{equation}\label{mkdefinition}
\mathfrak m_k(P_1,\dots,P_k)
= \sum_{\beta} T^{\omega[\beta]}e^{\mu_L(\beta)/2}ev_{0*}
\left([\MM_{k+1}^{\text{\rm main}}(\beta;P_1,\dots,P_k)]\right).
\end{equation}
Here $ev_{0*}
\left([\MM_{k+1}^{\text{\rm main}}(\beta;P_1,\dots,P_k)]\right)$ is the
virtual fundamental chain of this moduli space, which is regarded as a
chain on $L$ by the map $ev_0$.
In general, virtual fundamental chain depends on the choice of perturbation.
\par
In \cite{fooo-book1}, \cite{fooo-book2} we used multivalued perturbations (or
multisections) to define this virtual fundamental chain. So it is
defined over $\Q$. (We need the orientation, which is induced by the
relative spin structure.) In this paper we will use Theorem
\ref{maintechnicalresult} to define a virtual fundamental chain over
$\Z$ or $\Z_2$.
\par
In the next section we will also use a moduli space
$\MM_{k+1,\ell}^{\text{\rm main}}(\beta;P_1,\dots,P_k)$ which is
defined as follows: We consider
$((\Sigma,\vec z,\vec w),v)$, where $((\Sigma,\vec z),v)$ satisfies
the same assumption as in the definition of $\MM_{k+1}^{\text{\rm
main}}(\beta;P_1,\dots,P_k)$, except the stability, $\vec w =
(w_1,\dots,w_{\ell})$, $w_i \in \Sigma \setminus \partial\Sigma$,
and $w_i \ne w_j$ for $i\ne j$. We assume $((\Sigma,\vec z,\vec
w),v)$ is stable: namely there exists only a finite number of
biholomorphic maps $\varphi: \Sigma \to \Sigma$ such that
$\varphi(z_i) = z_i$, $\varphi(w_i) = w_i$ and $v \circ \varphi =
v$. We define an equivalence relation on the set of such
$((\Sigma,\vec z,\vec w),v)$ in a similar way. The set of
equivalence classes is denoted by $\MM_{k+1,\ell}^{\text{\rm
main}}(\beta;P_1,\dots,P_k)$.
\par
We can generalize Theorem \ref{modulimain} to
$\MM_{k+1,\ell}^{\text{\rm main}}(\beta;P_1,\dots,P_k)$,
where the formula (\ref{dimformula}) is replaced by
\begin{equation}\label{dimformula2}
\dim \MM_{k+1,\ell}^{\text{\rm main}}(\beta) = n + \mu_L(\beta) + 2\ell + k-2.
\end{equation}
Using the evaluation map $ev = (ev_0,\dots,ev_k):
\MM_{k+1,\ell}^{\text{\rm main}}(\beta) \to L^{k+1}$ we define the
fiber product $\MM_{k+1,\ell}^{\text{\rm
main}}(\beta;P_1,\dots,P_k)$ in the same way as in (\ref{modulifiberprod}).
\section{Estimate of the virtual dimension of the singular locus}\label{dimensionsec}
The main result of this section is:
\begin{prop}\label{dimprop}
Suppose $(M,\omega,J)$ is spherically positive.
Then for each $\Gamma \ne \{1\}$ and $i$, we have
\begin{equation}\label{dimform}
\dim \MM^{\text{\rm main}}_{k+1}(\beta; P_1, \dots, P_k)
- d( \MM^{\text{\rm main}}_{k+1}(\beta; P_1, \dots, P_k);\Gamma;i) \ge 2.
\end{equation}
\end{prop}
We recall that $d( \MM^{\text{\rm main}}_{k+1}(\beta; P_1, \dots, P_k);\Gamma;i)$
is defined in (\ref{35.19b}).
\begin{proof}
Let $((\Sigma,\vec z),v,\vec x)$ be an element of
$\MM^{\text{main}}_{k+1}(\beta; P_1, \dots, P_k)$: Namely $\Sigma$
is a genus zero bordered Riemann surface, $v: (\Sigma,\partial
\Sigma) \to (M,L)$ is pseudo-holomorphic, $\vec z =
(z_0,\dots,z_k)$ are boundary marked points of $\Sigma$, and
$$
\vec x = (x_1,\dots,x_k), \quad x_i \in P_i \quad \text{satisfying }\,
v(z_i) = f_i(x_i).
$$
Here $f_i: P_i \to L$ and $(P_i,f_i)$ are regarded as smooth
singular chains of $L$, which we write just as $P_i$ by an abuse of
notation.
\par
We recall that the genus of $\Sigma$ is zero and the disc
components cannot have non-trivial automorphism groups since it
has at least one special point, i.e., either marked or nodal
points, on the boundary. Therefore for every non-trivial element
$\varphi \in \text{Aut} ((\Sigma,\vec z),v)$ and any sphere
component $S_i^2 \cong S^2 \subset \Sigma$ preserved by $\varphi$,
the automorphism $\varphi$ acts as the multiplication by
$e^{2\pi\sqrt{-1}\ell/k}$,
$$
z \mapsto e^{2\pi\sqrt{-1}\ell/k}z
$$
with the identification $S^2_i = \C \cup \{\infty\}$. And $\varphi$
permutes other components: This is because any finite
subgroup of $PSL(2;\C) = \text{Aut} (S^2)$ that fixes $\infty$ is
conjugate to such a group.
\par
As a consequence, the quotient space $(\overline
\Sigma,\vec{\overline z}) = \Sigma/\text{Aut}((\Sigma,\vec z),v)$ is
again a (pointed) bordered Riemann surface of genus zero.
The pseudo-holomorphic map $v$ induces a map $\overline v:
\Sigma/\text{Aut}((\Sigma,\vec z),v) \to M$.
\begin{defn}\label{def:reducedmodel}
We call
$((\overline \Sigma,\vec{\overline z}),\overline v,\vec x)$ the {\it
reduced model} of $((\Sigma,\vec z),v,\vec x)$.
\end{defn}
We note that the
reduced model could be unstable. Namely there may appear a sphere
component with two singular points where the map $\overline v$ is
trivial. Even if the reduced model is stable, it may have a
nontrivial automorphism.
\begin{rem}\label{35.54}
We note that the notions of ``trivial automorphism'' and
``somewhere injective'' are two different notions: Somewhere
injectivity implies triviality of the automorphism group, but not
the other way around. For example, there is a branched covering $S^2
\to S^2$ with no nontrivial automorphism. {\it For the abstract
perturbation $\mathfrak s^{\epsilon}$, its transversality to zero is
related to the existence of nontrivial automorphism but not to the
somewhere injectivity.} (Somewhere injectivity is essential if one
uses perturbations only of $J$ to achieve transversality.)
\end{rem}
We now compare the virtual dimension of $((\Sigma,\vec z),v,\vec x)
\in \MM_{k+1}^{\text{main}}(\beta ; P_1,\dots, P_k)$ with that of
its reduced model. We begin with a discussion of the deformation
complex of a multiple sphere. Let $\alpha \in \pi_2(M)$ and
$\widetilde{\mathcal M}^{\text{reg}}(M;\alpha)$ be the set of
pseudo-holomorphic maps $u: S^2 \to M$ with $[u] = \alpha$. For $u
\in\widetilde{\mathcal M}^{\text{reg}}(M;\alpha)$ we define
$\mathfrak R_m(u)\in \widetilde{\mathcal M}^{\text{reg}}(M;m\alpha)$
by $\mathfrak R_m(u)(z) = u(z^m)$.
\par
For each $v \in \widetilde{\mathcal M}^{\text{reg}}(M;m\alpha)$, we consider
the linearization
$$
D_{v}\overline\partial: \Gamma(S^2;v^*TM) \to\Gamma(S^2;
\Lambda^{0,1} (v^*TM))
$$
of the Cauchy-Riemann section $\overline \partial$.
We denote by $C(v)=(C_0(v),C_1(v),D_{v}\overline\partial)$
the elliptic complex given by
$$C_0(v) =  \Gamma(S^2;v^*TM),\quad C_1(v) =
\Gamma(S^2;\Lambda^{0,1}(v^*TM)).$$
We consider the assignment of the pull-back complex $C(\mathfrak
R_m(u))$ to $u \in \widetilde{\mathcal M}^{\text{reg}}(M;\alpha)$ on
which the group $\Z_m$ acts. We regard this assignment of
$\Z_m$-modules as a $\Z_m$-equivariant family index over
$\widetilde{\mathcal M}^{\text{reg}}(M;\alpha)$.
\begin{lem}\label{35.55} The index of $C(\mathfrak R_m(u))$ as a
$\Z_m$-module is
$$2c_1(M)(\alpha) \operatorname{Reg}_{\Z_m} \oplus 2n \text{\bf 1}.$$
Here $\operatorname{Reg}_{\Z_m}$ is the regular representation of
$\Z_m$ and \text{\bf 1} is its trivial representation.
\end{lem}
\begin{proof}
Let $\gamma$ be an element of $\Z_m$ with $\gamma \neq \text{unit}$.
We use the Lefschetz fixed point formula by Atiyah-Bott
\cite{AtBo67} to obtain
$$\sum_{*=0,1}(-1)^*\operatorname{Tr}\left(\gamma: H^{*}(C(\mathfrak R_m(u))
\to H^{*}(C(\mathfrak R_m(u))\right)= 2n.
$$
Note there are only two
fixed points of $\gamma$ and we take the trace over $\R$. On the
other hand, the numerical index of $D_{\mathfrak
R_m(u)}\overline\partial$ is $2n + 2mc_1(M)(\alpha)$, which
coincides with the super-trace of the unit element $e \in \Z_m$. The
lemma follows immediately. \end{proof}
(We note that we can also
prove Lemma \ref{35.55} by directly calculating the kernel and
cokernel without using  \cite{AtBo67}.)
\par
In particular, Lemma \ref{35.55} implies that the $\Z_m$-invariant part of the
index
of $C({\mathfrak R}_m(u))$ for the $\Z_m$-cover of a holomorphic sphere
is equal to the index of $C(u)$ for its reduced model.
We will use this fact in the proof of Proposition \ref{35.63} coming later.
\par
Let $((\Sigma,\vec z),v,\vec x) \in
\MM_{k+1}^{\text{main}}(\beta:P_1,\dots, P_k)$ and
$((\overline\Sigma,\vec{\overline z}),\overline v,\vec x)$
be its reduced model.
We first recall the definition of the deformation complex of
$((\Sigma,\vec z),v,\vec x)$ which is an elliptic complex acted
upon by the group $\Gamma$ of automorphisms of $((\Sigma,\vec z),v)$.
\par
We decompose $\Sigma$ into irreducible components $\Sigma=
\bigcup_a \Sigma_{a}$ where $\Sigma_a$ is a sphere or a disc and
put $v_a = v\vert_{\Sigma_a}$ and consider the elliptic complex
$$C(v_a) = (C_0(v_a),C_1(v_a),D_{v_a}\overline\partial)$$ where
$$
D_{v_a}\overline\partial ~:~
C_0(v_a)=\Gamma(\Sigma_{a},\partial \Sigma_{a};v_a^*TM,v_a^*TL)
\to C_1(v_a)=\Gamma(\Sigma_{a};\Lambda^{0,1}(v_a^*TM)).
$$
(The boundary condition $v_a^*TL$ is empty if $\Sigma_a =S^2$.)
For each singular point $z^{\text{sing}}_i$ we take
$z^{\text{sing}}_{i,1} \in \Sigma_{a(i,1)}$,
$z^{\text{sing}}_{i,2} \in \Sigma_{a(i,2)}$ which are
$z^{\text{sing}}_i$ in $\Sigma$. We put
\begin{equation}
\left\{\aligned
&\widetilde C_0((\Sigma,\vec z),v)^+ = \bigoplus_a C_0(v_a), \\
&\widetilde C_0((\Sigma,\vec z),v) =\left.\left\{(W_a) \in \bigoplus_a C_0(v_a)
\,\right\vert \,W_{a(i,1)}(z^{\text{sing}}_{i,1})
=W_{a(i,2)}(z^{\text{sing}}_{i,2})\right\}.
\endaligned\right.\label{35.56}\end{equation}
We put $\widetilde C_1((\Sigma,\vec z),v) = \bigoplus_a C_1(v_a)$.
The operators $D_{v_a}\overline\partial$ induce
$$D_v\overline\partial: \widetilde C_0((\Sigma,\vec z),v)\to
\widetilde C_1((\Sigma,\vec z),v).
$$ \par Let
$\operatorname{Aut}(\Sigma,\vec z)$ be the group of all
automorphisms of $(\Sigma,\vec z)$. We have a canonical homomorphism
of its Lie algebra $\operatorname{aut}(\Sigma,\vec z)$ into
$\widetilde C_0((\Sigma,\vec z),v)$: Note that by the definition of
$\operatorname{Aut}(\Sigma,\vec z)$ any element of
$\operatorname{aut}(\Sigma,\vec z)$ has its value zero at the
singular points. The stability condition implies that this
homomorphism is injective and so we may regard
$\operatorname{aut}(\Sigma,\vec z)$ as a subspace of $\widetilde
C_0((\Sigma,\vec z),v)$. Moreover the image of
$\operatorname{aut}(\Sigma,\vec z)$ lies in the kernel of
$D_v\overline\partial$. Therefore we have the following complex
\begin{equation}
0 \to  \operatorname{aut}(\Sigma,\vec z) \to\widetilde
C_0((\Sigma,\vec z),v) \to \widetilde C_1((\Sigma,\vec z),v).
\label{35.57}
\end{equation}
We put
$$\aligned \widetilde C_0((\Sigma,\vec z),v,\vec x) & =
\Big\{\Big((W_a),(\text{\rm v}_i)\Big)
\,\Big\vert \, (W_a)\in \widetilde{C}_0((\Sigma,\vec z),v), \,\text{\rm v}_i \in
T_{x_i}P_i,\\&
\hskip1.5in
W_{a_i}(z_i) = (d_{x_i}f_i)(\text{\rm v}_{i})\Big\}.
\endaligned$$
(We note that $f_i : P_i \to L$ is our smooth singular simplex.)
Since $\operatorname{Aut}(\Sigma,\vec z)$ fixes the marked points $\vec z$,
it induces an action
on $\MM_{k+1}^{\text{main}}(\beta:P_1,\dots, P_k)$  and so its Lie algebra
$\operatorname{aut}(\Sigma,\vec z)$  injects to $\widetilde C_0((\Sigma,\vec z),v,\vec x)$.
This leads us to define
$$\aligned
C_0((\Sigma,\vec z),v,\vec x) &: = \widetilde C_0((\Sigma,\vec
z),v,\vec x)/ \operatorname{aut}(\Sigma,\vec z)\\C_1((\Sigma,\vec
z),v,\vec x) &: = \widetilde C_1((\Sigma,\vec z),v).
\endaligned$$
Here $z_i \in \Sigma_{a_i}$.
The operator $D_v\overline\partial$ also induces a homomorphism
$C_0((\Sigma,\vec z),v,\vec x) \to C_1((\Sigma,\vec z),v,\vec x)$.
We denote  by
$$
C((\Sigma,\vec z),v,\vec x), \quad \widetilde C((\Sigma,\vec
z),v,\vec x)$$
the complexes
$$
\aligned D_v\overline\partial &: C_0((\Sigma,\vec z),v,\vec x) \to
C_1((\Sigma,\vec z),v,\vec x), \\
D_v\overline\partial &:  \widetilde C_0((\Sigma,\vec z),v,\vec x)
\to \widetilde C_1((\Sigma,\vec z),v,\vec x),
\endaligned
$$
respectively.
The group $\Gamma = \operatorname{Aut}(\Sigma,\vec z)$ acts on
these complexes in an obvious way.
\par
To describe the relation between the deformation complex of
$((\Sigma,\vec z),v,\vec x)$ and that of its reduced model, we need one more notation.  \par
\begin{defn}\label{35.58}
We call a point $z \in \Sigma$  a {\it free fixed point} if
the following holds:
\begin{enumerate}[{\rm (i)}]
\item $z$ is not a singular point. \par
\item $z$ is on a sphere component $S^2_a$ of $\Sigma$
such that there exists $\gamma \in \Gamma$ which preserves $S_a^2$
and acts nontrivially on $S^2_a$.
Moreover $\gamma(z) = z$.
\end{enumerate}
\smallskip
For $\Gamma' \subset \Gamma$ we denote by $F(\Gamma')$ the set of
all $z$ satisfying (i), (ii) above for $\gamma \in \Gamma'$.
\end{defn}
For $\Gamma' \subset \Gamma$ we define
$$
C(F(\Gamma')) = \bigoplus_{z\in F(\Gamma')} \C [z]
$$
the vector space whose basis is identified with $F(\Gamma')$.
\par
In case
$\Gamma^{\prime\prime}$ normalizes a subgroup $\Gamma' (\subset
\Gamma)$, $\Gamma^{\prime\prime}$ acts on $F(\Gamma^{\prime})$ in an
obvious way and so induces an action on $C(F(\Gamma^{\prime}))$.
We put
$$
C(F(\Gamma^{\prime}))^{\Gamma^{\prime\prime}} = \{v \in
C(F(\Gamma')) \mid \forall \gamma \in \Gamma^{\prime\prime}
,\,\,\gamma v = v\}.
$$
It is easy to see that
$$
\dim_{\C} C(F(\Gamma^{\prime}))^{\Gamma^{\prime\prime}}  = \#
\left(F(\Gamma^{\prime})/\Gamma^{\prime\prime}\right)
$$
and in particular
\begin{equation}
\dim_{\C} C(F(\Gamma))^{\Gamma}
= \#\left(F(\Gamma)/\Gamma\right).
\label{35.60}\end{equation}
\begin{defn}\label{35.61}
For each sphere component $S^2_a$, we define its {\it distance
from the disc components} as the {\it minimal edge distance} of
the vertex corresponding to the component $S^2_a$ from the
vertices corresponding to the disc components in the dual graph of
$\Sigma$.
\end{defn}
We recall that the minimal edge distance between two vertices in a
graph is defined to be the minimum number of edges in all
connected paths between the two in the graph. We denote by
$\Sigma_d$ the union of all disc components and the sphere
components whose distance from the disc components are $\le d$.
By definition $\Sigma_0$ is the union of all the disc components.
\par
Let $(S_a^2,\vec{p}_a, o_a)$ be a sphere component of $((\Sigma,
\vec{z}),v)$, whose distance from the disc components is $d$. Here
$o_a$ is the point where $S_a^2$ is attached to $\Sigma_{d-1}$ and
$\vec{p}_a$ is the set of other singular points, i.e., those to
which some sphere components of distance $d+1$ from the disc
components are attached. Put $n_a = \# \vec{p}_a$ and let $\Gamma_a$
be the group of automorphisms on $(S_a^2,\vec{p}_a, o_a)$ consisting
of the restrictions of some elements of $\Gamma$ to $S_a^2$. Denote
by $\operatorname{Conf}_{m+1}(\C P^1)$ the moduli space (i.e.,
divided by the action of $PSL_2(\C)$) of $m+1$ ordered points on $\C P^1$.
Denote by $\operatorname{Conf}_{m + 1}^{\Gamma_a}(\C P^1)$ the
moduli space of distinct $m+1$ points on $\C P^1$ with the symmetry
group $\Gamma_a$,
(that is the set of $(\C P^1,(z_0,\dots,z_m))$ such that $\gamma z_i \in
\{z_0,\dots,z_m\}$ for each $i=1,\dots,m$ and $\gamma \in \Gamma_a$.)
Let $\operatorname{Aut}(S_a,\vec{p}_a,o_a)$ be the group of all
biholomorphic maps $\varphi : S_a \to S_a$ such that
$\varphi(\vec p_a)  = \vec p_a$ and $\varphi(o_a) = o_a$.
We put
$$
\operatorname{Aut}(S_a,\vec{p}_a,o_a)^{\Gamma_a}
= \{\varphi \in \operatorname{Aut}(S_a,\vec{p}_a,o_a)
\mid \forall \gamma \in \Gamma_a,\, \varphi\circ\gamma = \gamma \circ \varphi\}.
$$
We define
$$
\rho^{\Gamma_a}(S_a,\vec{p}_a,o_a) =
\begin{cases} -\dim_{\R}
\operatorname{Aut}(S_a,\vec{p}_a,o_a)^{\Gamma_a} \quad &
\text{~if~} n_a+1 <3
\\
\dim_{\R} \operatorname{Conf}_{n_a+1}^{\Gamma_a}(\C P^1) \quad &
\text{~if~} n_a+1 \geq 3.
\end{cases}
$$
Let $(\overline{S}_b, \vec{\overline{p}}_b, \overline{o}_b)$ be a
sphere component of the reduced model $((\overline
\Sigma,\vec{\overline z}),\overline v)$. Here $\overline{o}_b$ is
the point of $\overline{S}_b$ at which it is attached to
$\overline{\Sigma}_{d-1}$, and $\vec{\overline{p}}_b$ are those at
which some sphere components of distance $d
+1$ are attached. Put
$\overline{n}_b= \# \vec{\overline p}_b$. We define
$$
\overline{\rho}(\overline{S}_b, \vec{\overline{p}}_b,
\overline{o}_b) =
\begin{cases}
- \dim_{\R}
\operatorname{Aut}(S_b,\vec{\overline{p}}_b, \overline{o}_b)
\quad \text{~if~} \overline{n}_b +1 < 3 \\
\dim_{\R} \operatorname{Conf}_{\overline{n}_b+1}(\C P^1) \quad
\text{~if~} \overline{n}_b+1 \geq 3.
\end{cases}
$$
Let $(D_a^2;\vec p_a^{\,D},\vec p_a^{\,S})$ be a disc component of
$\Sigma$. Here $\vec p_a^{\,D} =
(p_{a,1}^{\,D},\dots,p_{a,m_a}^{\,D})$ is the boundary marked
points and $\vec p_a^{\,S} =
(p_{a,1}^{\,S},\dots,p_{a,n_a}^{\,S})$ is the interior marked
points. We put
$$
\rho(D_a^2;\vec p_a^{\,D},\vec p_a^{\,S}) = m_a + 2n_a - 3.
$$
This number is the negative of the dimension of the automorphism
group if $m_a + 2n_a \le 3$ and is the dimension of appropriate
moduli space (of $\C P^1$ with marked points) if $m_a + 2n_a \ge
3$. (We remark that the $\Gamma$-action is trivial on each disc
component.)
We define $\overline{\rho}(\overline D_b^2;\vec{\overline
p}_b^{\,D},\vec{\overline{p}}_b^{\,S})$ by the same formula for
the disc component $(\overline D_b^2;\vec{\overline
p}_b^{\,D},\vec{\overline{p}}_b^{\,S})$  of $\overline{\Sigma}$.
\par
Let $\{S_{a(b)} \mid b \in \overline I \}$ be a complete set of
representatives of the $\Gamma$-orbits in the set of sphere
components of $((\Sigma, \vec{z}), v)$.
Let $I^D$ and $\overline I^D$ be the sets of disc components of
$\Sigma$ and $\overline{\Sigma}$, respectively. Note that there is a
canonical identification between the two sets.
We define
$$
\rho^{\Gamma}((\Sigma, \vec{z}), v) =
\sum_{b \in \overline I} \rho^{\Gamma_{a(b)}}(S_{a(b)},\vec{p}_{a(b)},o_{a(b)})
+ \sum_{a\in I^D}\rho(D_a^2;\vec
p_a^{\, D},\vec
p_a^{\, S}).
$$
For the reduced model, we define
$$
\overline{\rho}((\overline{\Sigma}, \vec{\overline{z}}),
\overline{v}) = \sum_{b\in\overline I} \rho(\overline{S}_b,
\vec{\overline{p}}_b, \overline{o}_b) +
\sum_{b\in \overline I^D}\overline{\rho}
(\overline D_b^2;\vec{\overline p}_b^{\,D},\vec{\overline{p}}_b^{\,S}).
$$
\par
Then we have the following:
\begin{prop}\label{35.63}
$$
\aligned
& \dim_{\R} \text{\rm Index}(\widetilde C((\Sigma,\vec z),v,\vec x))^{\Gamma}
+ \rho^{\Gamma}((\Sigma, \vec{z}), v) \\
&=
\dim_{\R} \text{\rm Index}(\widetilde C((\overline \Sigma,\vec{\overline z}),\overline v,\vec x))
+ \overline{\rho}((\overline{\Sigma}, \vec{\overline{z}}), \overline{v})
+ \dim_{\R} C(F(\Gamma))^{\Gamma}.
\endaligned
$$
\end{prop}
\begin{proof}
We begin with the following lemma.
\begin{lem}\label{35.64}
$$\dim_{\R} \text{\rm Index}(\widetilde C((\Sigma,\vec z),v,\vec x))^{\Gamma}
=
\dim_{\R} \text{\rm Index}(\widetilde C((\overline \Sigma,\vec{\overline
z}),\overline v,\vec x)).$$
\end{lem}
\begin{proof}
Let us consider the complex $\widetilde{C}((\Sigma,\vec z),v,\vec x)^+$
$$
D_{v}\overline{\partial}
: \widetilde{C}_0((\Sigma,\vec z),v,\vec x)^+
\to \widetilde{C}_1((\Sigma,\vec z),v,\vec x)
$$
where we replace $ \widetilde{C}_0((\Sigma,\vec z),v,\vec x)$ by $
\widetilde{C}_0((\Sigma,\vec z),v,\vec x)^+$. (See (\ref{35.56}).)
\par
We first prove
\begin{equation}
\text{Index}(\widetilde{C}((\Sigma,\vec z),v,\vec x)^+)^{\Gamma} =
\text{Index}(\widetilde{C}((\overline{\Sigma},\vec
{\overline{z}}),\overline{v},\vec {\overline{x}})^+).
\label{35.65}\end{equation}
Note the index of $\widetilde{C}((\Sigma,\vec
z),v,\vec x)^+$ is the sum of indices of its components. Since the
$\Gamma$-action is trivial on disc component, (\ref{35.65}) is trivial
for disc components. The part of sphere components of the left
hand side is
\begin{equation}
\sum_{b\in \overline I}
\text{Index}(D_{v_{a(b)}}\overline\partial)^{\Gamma_{a(b)}}.
\label{35.66}\end{equation}
Here $\{S_{a(b)} \mid b\in \overline I \}$ is the
complete set of representatives of the $\Gamma$-orbit space of the
set of all sphere components of $(\Sigma, \vec{z}, v)$ and the map
$v_{a(b)}$ is the restriction of $v$ to $S_{a(b)}$. The group
$\Gamma_{a(b)}$ is a subgroup of $\Gamma$ consisting of the
elements which preserve $S_{a(b)}$. We remark that $\Gamma_{a(b)}$
is a cyclic group. Hence we can apply Lemma \ref{35.55} to show that
(\ref{35.66}) is equal to the sum of $ \text{Index}(D_{\overline
v_{b}}\overline\partial)$. Here $\overline v_b$ is the restriction
of $\overline v$ to $S_{a(b)}/\Gamma_{a(b)} = \overline S_b
\subset \overline\Sigma$. (\ref{35.65}) follows.
\par
We next observe that there exists an exact sequence
\begin{equation}
0 \to \widetilde{C}((\Sigma,\vec z),v,\vec x)
\to \widetilde{C}((\Sigma,\vec z),v,\vec x)^+
\to
\bigoplus_{x\in \text{Sing }\Sigma} T_{v(x)}M \to 0.
\label{35.67}\end{equation}
Here the $\bigoplus_{x\in \text{Sing }\Sigma} T_{v(x)}M$ is the
sum over all singular points $x$ of $\Sigma$.
\par
We remark that $\Gamma$ action on $\bigoplus_{x\in \text{Sing }\Sigma} T_{v(x)}M$
is by interchanging the factors. It follows that
$$
\dim \left(\bigoplus_{x\in \text{Sing }\Sigma} T_{v(x)}M\right)^{\Gamma}
= 2n \, \# ((\text{Sing }\Sigma)/\Gamma).
$$
Note that $(\text{Sing }\Sigma)/\Gamma \cong \text{Sing }\overline\Sigma$.
Therefore (\ref{35.65}), (\ref{35.67}) and a similar exact sequence for $\overline\Sigma$
imply the lemma.
\end{proof}
We compare
$\rho^{\Gamma_{a(b)}}(S_{a(b)},\vec{p}_{a(b)},o_{a(b)})$ and
$\overline{\rho}(\overline{S}_b,
\vec{\overline{p}}_b,\overline{o}_b)$. Note that $\Gamma_{a(b)}$
is isomorphic to $\Z_{m_{a(b)}}$. If $m_{a(b)}=1$, $\Gamma_{a(b)} =
\{e\}$ and so it is obvious that
\begin{equation}
\rho^{\Gamma_{a(b)}}(S_{a(b)},\vec{p}_{a(b)},o_{a(b)}) =
\overline{\rho}(\overline{S}_b,
\vec{\overline{p}}_b,\overline{o}_b). \label{35.68}\end{equation}
Note that there is no free fixed point on $S_{a(b)}$.
\par
From now on, we assume that $m_{a(b)} > 1$. The marked point
$o_{a(b)}$ is a $\Gamma_{a(b)}$ fixed point. Denote by $q_{a(b)}$
the other fixed point on $S_{a(b)}$. Let $\overline q_{b}$ be its
image of the reduction in $\overline S_b$. There are two cases:
$q_{a(b)} \in \vec{p}_{a(b)}$ or $q_{a(b)} \notin \vec{p}_{a(b)}$.
\par
First consider the case $q_{a(b)} \in \vec{p}_{a(b)}$,
$\overline{q}_b \in \vec{\overline{p}}_b$.   We claim that
\begin{equation}
\rho^{\Gamma_{a(b)}}(S_{a(b)}, \vec{p}_{a(b)}, o_{a(b)}) =
\overline{\rho}(\overline{S}_b, \vec{\overline{p}}_b, \overline{o}_b).
\label{35.69}\end{equation}
\par
If $n_{a(b)} = 1$, then $\overline{n}_b=1$,
$\vec p_{a(b)} = q_{a(b)}$,   $\vec{\overline p}_{b} = \overline q_{b}$.
We have
$$\operatorname{Aut}(S_{a(b)},q_{a(b)}, o_{a(b)})^{\Gamma_{a(b)}} \cong
\operatorname{Aut}(S_{a(b)},q_{a(b)}, o_{a(b)}).
$$
Both
$\operatorname{Aut}(S_{a(b)},q_{a(b)}, o_{a(b)})^{\Gamma_{a(b)}}$
and $\operatorname{Aut}(\overline{S}_b, \overline{q}_b,
\overline{o}_b)$ are isomorphic to $\C^*$ and hence (\ref{35.69})
follows.
\par
In case $n_{a(b)} \geq 2$, we have $\overline n_b \ge 2$ and
$$
\operatorname{Conf}_{n_{a(b)} + 1}^{\Gamma_{a(b)}}(S_{a(b)};
\vec{p}_{a(b)}, o_{a(b)}) \cong
\operatorname{Conf}_{\overline{n}_b+1}(\overline S_b ;
\vec{\overline{p}}_b, \overline{o}_b).
$$
Hence, we obtain
(\ref{35.69}).
\par
We next consider the case $q_{a(b)} \notin \vec{p}_{a(b)}$,
$\overline{q}_b \notin \vec{\overline{p}}_b$. In this case, the position of
$\overline q_{b}$ in the reduced model is not a part of the data
of the reduced model, $(\overline{S}_b, \vec{\overline{p}}_b,
\overline{o}_b)$. We claim
\begin{equation}
\rho^{\Gamma_{a(b)}}(S_{a(b)},\vec{p}_{a(b)}, o_{a(b)}) =
\overline{\rho}(\overline{S}_b, \overline{o}_b) +2. \label{35.70}
\end{equation}
If $n_{a(b)} = 0$, then $\overline{n}_b=0$. Note that $q_{a(b)}$
is a fixed point on $S_{a(b)}$.
We find
$$
\operatorname{Aut}(S_{a(b)},\vec{p}_{a(b)},
o_{a(b)})^{\Gamma_{a(b)}} \cong
\operatorname{Aut}(S_{a(b)},q_{a(b)}, o_{a(b)}),
$$
which is isomorphic to $\C^*$.
On the other hand,
$\operatorname{Aut}(\overline{S}_b, \overline{o}_b)$
is isomorphic to the semi-direct product
$\C^* \times \C$ that is the group of all
affine maps $z \mapsto Az+B$ with $A\ne 0$.
(\ref{35.70}) follows.
\par
Suppose that $\overline{n}_b=1$. Denote by $\overline{p}_b$ the
unique point in $\vec{\overline{p}}_b$. In this case, we find that
$\operatorname{Conf}_{n_{a(b) } +1}^{\Gamma_{a(b)}}(\C P^1)$ is a
point and $n_{a(b)} \ge 2$. On the other hand,
$\operatorname{Aut}(\C P^1, \overline{p}_b, \overline{o}_b)$ is
isomorphic to $\C^*$. (\ref{35.70}) follows.
\par
If $\overline{n}_b >1$, then  $\vec{p}_{a(b)}$ consists of
$\overline{n}_b$ $\Gamma_{a(b)}$-orbits in $S_{a(b)} \setminus \{
q_{a(b)}, o_{a(b)} \}$. Therefore we have
$$
\dim_{\R} \operatorname{Conf}_{n_{a(b)} +1}^{\Gamma_{a(b)}}(\C P^1)
= 2 \overline{n}_b-\dim_{\R} \C^{\ast}
= 2 \overline{n}_b-2,
$$
where $\C^{\ast}\cong \operatorname{Aut}(\C P^1, q_{a(b)}, o_{a(b)})$.
On the other hand,
$$\dim_{\R} \operatorname{Conf}_{\overline{n}_b+1}(\C P^1) = 2(\overline{n}_b+1)-\dim_{\R} \operatorname{Aut} (\C P^1)=
2\overline{n}_b-4.
$$
We also obtain (\ref{35.70}).
\par
Note that $q_{a(b)}$ is a free fixed point if and only if
$q_{a(b)} \notin {\vec{p}}_{a(b)}$.
Therefore (\ref{35.70}) is applied to the component $S_a$ if it contains
a free fixed point. Otherwise (\ref{35.68}),(\ref{35.69}) apply.
\par
We observe that the contribution of disc components to
$\rho^{\Gamma}((\Sigma, \vec{z}), v)$ coincides with that of disc
components to $\overline{\rho}((\overline{\Sigma},
\vec{\overline{z}}), \overline{v})$.
\par
Combining these, we obtain
$$
\rho^{\Gamma}((\Sigma, \vec{z}), v)
=   \overline{\rho}((\overline{\Sigma}, \vec{\overline{z}}), \overline{v}) +
\dim_{\R} C(F(\Gamma))^{\Gamma}.
$$
This formula together with Lemma \ref{35.64} implies Proposition \ref{35.63}.
\end{proof}
We next identify each side of Proposition \ref{35.63} with the dimension of
appropriate moduli space.
\begin{defn}\label{35.71}
We say that $((\Sigma,\vec z),v,\vec x)$ has the {\it same
combinatorial type} as $((\Sigma',\vec z\,'),v',\vec x\,')$ if the
following holds: There exists a homeomorphism $\Sigma \to \Sigma'$
preserving all the marked points together with the order. We also
assume that the restriction of $v$ to each component of $\Sigma$ is
homologous to the restriction of $v'$ to the corresponding component
of $\Sigma'$. We denote by $\mathbb S$ an equivalence class of this
equivalence relation and call $\mathbb S$ the combinatorial type of
$((\Sigma,\vec z),v,\vec x)$.
\par
We denote by
$$\MM_{k+1}^{\text{\rm main}}(\beta; P_1, \dots, P_k)
(\mathbb S,\Gamma)
~(= \MM_{k+
1}^{\text{\rm main}}(\beta; P_1, \dots, P_k)
(\mathbb S) ^{\cong}(\Gamma))
$$ the moduli space of $((\Sigma,\vec z),v,\vec x)$
with  given combinatorial type $\mathbb S$ and its isotropy group
$\Gamma$.
\end{defn}
Let $\mathbb S$ be the combinatorial type of $((\Sigma,\vec z),v,\vec
x)$.
\begin{lem}\label{35.72} The left hand side of Proposition $\ref{35.63}$
is
$$
\dim_{\R}\MM_{k+1}^{\text{\rm main}}(\beta; P_1, \dots, P_k)(\mathbb S,\Gamma).
$$
\end{lem}
\begin{proof} $\rho^{\Gamma}(\Sigma, \vec{z}, v)$ is the number of
deformation parameters of  $(\Sigma, \vec{z}, v)$ keeping the
$\Gamma$-equivariance and the combinatorial type minus the
dimension of the automorphism group of $(\Sigma, \vec{z})$. (Note that
we did not include the parameter to resolve singularities of
$\Sigma$. So this corresponds to the deformation keeping the
combinatorial type.)
Lemma \ref{35.72} follows.
\end{proof}
To study the right hand side of Proposition \ref{35.63} we need some notation.
For each $((\Sigma,\vec z),v,\vec x)$ we consider its reduced
model $((\overline \Sigma,\vec{\overline z}),\overline v,\vec x)$.
We then add the image of elements in $F(\Gamma)$ to the reduced
model as additional (interior) marked points. We denote them by
$\vec z_+$ and the resulting stable map by $((\overline \Sigma,
\vec{\overline z
},\vec z_+),\overline v,\vec x)$.
\par
Note
$$
\# \vec z_+ = \# F(\Gamma)/\Gamma.
$$
We call $((\overline \Sigma, \vec{\overline z},\vec z_+),\overline v,\vec
x)$ the {\it marked reduced model} of $((\Sigma,\vec
z),v,\vec x)$. In this way, we have obtained a natural assignment
$$
((\Sigma,\vec z),v,\vec x\,) \to ((\overline \Sigma,\vec{\overline z},\vec
z_+),\overline v,\vec x).
$$
We note that the reduced model $((\overline \Sigma,\vec{\overline z}),\overline
v,\vec x)$ may be unstable.  On the other hand
the marked reduced  model is always stable.
\par
Let us denote by $\overline{\mathbb S}$ the combinatorial type of
the marked reduced model. We consider $\MM_{k+1,\ell}^{\text{\rm
main}}(\overline\beta; P_1, \dots, P_k)(\overline{\mathbb S})$, the
moduli space of marked reduced models with the combinatorial type
$\overline{\mathbb S}$. Here $\ell$ is the order of
$F(\Gamma)/\Gamma$. We now have:
\begin{lem}\label{35.74}
The right hand side of Proposition $\ref{35.63}$ is equal to
$$
\dim_{\R} \MM_{k+1,\ell}^{\text{\rm main}}(\overline\beta; P_1, \dots, P_k)(\overline{\mathbb S}).
$$
\end{lem}
\begin{proof}
In the same way as in the proof of Lemma \ref{35.72}, we find that
$$
\dim_{\R}\text{\rm Index}(\widetilde C((\overline \Sigma,\vec{\overline z}),\overline v,\vec x))
+ \overline{\rho}((\overline{\Sigma}, \vec{\overline{z}}), \overline{v})
$$
is the virtual dimension of the moduli space of reduced models.
Adding marked points $\vec z_+$ increases the dimension by
$2\# \vec z_+ = 2 \# F(\Gamma)/\Gamma$, which is
equal to
$\dim_{\R} C(F(\Gamma))^{\Gamma}$.
\end{proof}
We next show the following.
\begin{lem}\label{35.75} Suppose that
$(M,J)$ is spherically positive and
$\mathbb S$ contains at least one sphere component. Then we have
$$
\dim_{\R} \MM_{k+1,\ell}^{\text{\rm main}}(\overline\beta; P_1, \dots, P_k)(\overline{\mathbb S})
\le
\dim_{\R} \MM_{k+1}^{\text{\rm main}}(\beta; P_1, \dots, P_k) -2.
$$
\end{lem}
\begin{rem}\label{35.76}
We have not used spherical positivity of $J$ up to this point.
Namely the proof of Lemma \ref{35.75} is the
only place where we use this assumption.
\end{rem}
\begin{proof} Let $J$ be a spherically positive almost complex structure.
We consider a component $S_a$ of an element of
$\MM_{k+1}^{\text{\rm main}}(\beta; P_1, \dots, P_k)({\mathbb S})$.
Let $\vec w_a, \vec z_a$ be the sets of all marked or singular
points on $S_a$, $\partial S_a$, respectively. Let $v_a =
v\vert_{S_a}$. We put $k_a = \# \vec w_a + \# \vec z_a/2$. We put
$s_a =3$ if $S_a$ is a disc component and $s_a =6$ if $S_a$ is a
sphere component. Set
$$
c(a) = 2c_1(M)[v_a]
$$
for a sphere component and
$$
c(a) = \mu_L(v_a)
$$
for a disc component.
\par
We claim that
\begin{equation}
c(a)  + 2k_a - s_a \ge - 2 \label{35.77}\end{equation}
holds for the sphere
components.
\par
In fact, if the map is trivial on this component, then $2k_a \ge
6$ and $c(a) =0$. (\ref{35.77}) holds. If the map is nontrivial on this
component, then we have $c(a) \ge 2$ by the choice of $J$ made in
the beginning of the proof using spherical positivity. Since $2k_a \ge 2$ and $s_a =6$ for a sphere
component, (\ref{35.77}) also holds.
\par
We put $\e(a) = 2$ for a sphere component and $\e(a) = 1$ for a
disc component. It is easy to see from the index theorem that
\begin{equation}\aligned
&\dim_{\R} \MM_{k+1}^{\text{\rm main}}(\beta; P_1, \dots, P_k)({\mathbb S}) \\
&= \sum_{a\in I} (\e(a)n + c(a)  + 2k_a-s_a) \\
& \quad - (2n)\,\#\text{Sing}_{S^2}\mathbb S
- n\,\#\text{Sing}_{D^2}\mathbb S - \sum \deg P_i,
\endaligned\label{35.78.1}\end{equation}
where $I$ is the set of all components of $\mathbb S$,
$\#\text{Sing}_{S^2}\mathbb S$ is the number of singular points
which intersect with sphere components and
$\#\text{Sing}_{D^2}\mathbb S$ is the number of singular points
which do not intersect with sphere components.
We recall $2n = \dim M$.
\par
We have the similar identity for
$\dim_{\R} \MM_{k+1,\ell}^{\text{\rm main}}(\overline\beta; P_1, \dots,
P_k)(\overline{\mathbb S})$.
Namely for each component $S_b$ of $\overline{\mathbb S}$
we define
$\overline k_b$, $\overline
s_b$, $\overline c(b)$,
$\overline{\e}(b)$ and obtain
\begin{equation}\aligned
&\dim_{\R} \MM_{k+1,\ell}^{\text{\rm main}}(\overline\beta; P_1, \dots, P_k)(\overline{\mathbb S}) \\
&= \sum_{b\in \overline I} (\overline{\e}(b)n + \overline c(b)  + 2\overline k_b-\overline s_b)
\\
& \quad -(2n)\,\#\text{Sing}
_{S^2}\overline{\mathbb S}
-n\,\#\text{Sing}
_{D^2}\overline{\mathbb S} - \sum
\deg P_i
\endaligned\label{35.78.2}\end{equation}
where $\overline I$ is the set of all components of $\overline{\mathbb S}$.
\par
For each component $b \in \overline I$ we take
$a(b) \in I$ such that $S_{a(b)}$ is a branched covering of $S_b$. (There may be several
of them. We choose one of them.)
\par
If $S_b$ is a disc component, we have
\begin{equation}
\overline c(b)  + 2\overline k_b - \overline s_{b}= c(a(b))  +
2k_{a(b)}- s_{a(b)}, \label{35.79}\end{equation} since the automorphism group is
trivial on the disc components.
\par
We next prove the following:
\begin{sublem}\label{35.80}
$$
\overline c(b)  + 2\overline k_b - \overline s_{b} \le c(a(b))  + 2k_{a(b)}- s_{a(b)},
$$
if $S_b$ is a sphere component.
\end{sublem}
\begin{proof} We have $\overline s_{b} = s_{a(b)} = 6$. By the
spherical positivity we have $\overline c(b) \ge 0$. It also follows
that $\overline c(b) \le c(a(b))$. If there is no free fixed point
on $S_{a(b)}$, then $\overline k_b \le k_{a(b)}$ and we are done.
\par
Now  assume that there is a free fixed point on $S_{a(b)}$. In
this case the degree, denoted by $\text{deg}$, of the map $S_{a(b)} \to S_b$
is greater than one.
\par
We first consider the case when $[v_b] \ne 0$. Then by the
definition of spherical positivity we have $c(a(b)) \ge 2$.
Therefore
$$
\overline c(b) = \frac{c(a(b))}{\text{deg}} \le c(a(b)) - 2.
$$
(Note that $c(a),\overline c(b)$ are even numbers for the sphere
components.) On the other hand, since there exists at most one
free fixed point on $S_{a(b)}$, it follows that
$$
\overline k_b \le k_{a(b)} + 1.
$$
The sublemma follows in this case.
\par
We next assume $[v_b] = 0$. Then $k_{a(b)} \ge 3$ by stability.
Namely there exist at least 3 singular or marked points. The
component $S_b
$ is identified with the quotient of $S_{a(b)}$ by
the cyclic group $\Gamma_{a(b)}$ of order $\text{deg} \in
\{2,3,\dots\}$. Let $d$ be the distance of $S_{a(b)}$ from the
disc components. Let $o_{a(b)}$ be the singular point on $S_{a(b)}$
where $S_{a(b)}$ is attached with $\Sigma_{d-1}$. Clearly
$o_{a(b)}$ is a fixed point of $\Gamma_{a(b)}$. Since there is a
free fixed point on $S_{a(b)}$, no other singular points are fixed
by $\Gamma_{a(b)}$. In other words the image in $S_b$ of the
singular points of $S_{a(b)}$ consists of $1 + (k_{a(b)}-1)/\deg$
points. Therefore
$$
\overline k_b = 2 + \frac{k_{a(b)}-1}{\deg}.
$$
Since $k_{a(b)} \ge 3$, we derive
$$
\overline k_b \le k_{a(b)}.
$$
The proof of Sublemma \ref{35.80} is now complete.
\end{proof}
It follows from Sublemma \ref{35.80} and (\ref{35.79}) that
\begin{equation}
\sum_{b\in \overline I} (\overline{\e}(b)n+\overline c(b)  + 2\overline k_b
- \overline s_b) \le
\sum_{b\in \overline I} ({\e}(a(b))n + c(a(b))  + 2k_{a(b)}-s_{a(b)}).
\label{35.81}\end{equation}
(\ref{35.77}) implies that for each sphere component $S_a$
($a \in I$) we have
$$
2n+c(a)  + 2k_a-s_a \ge 2n - 2.
$$
Note that the number of sphere components of $\mathbb S$ and of
$\overline{\mathbb S}$ are equal to $\# \text{Sing}_{S^2}\mathbb S$ and
to $\# \text{Sing}_{S^2}\overline{\mathbb S}$, respectively.
Therefore we have
\begin{equation}\aligned
&\sum_{a \in I} ({\e}(a)n + c(a)  + 2k_{a}-s_{a}) - (2n-2) \# \text{Sing}_{S^2}{\mathbb S}\\
&\ge
\sum_{b\in \overline I} ({\e}(a(b))n + c(a(b))  + 2k_{a(b)}-s_{a(b)}) - (2n-2) \#
\text{Sing}_{S^2}\overline{\mathbb S}.
\endaligned\label{35.82}\end{equation}
We remark that
$\text{Sing}_{D^2}\overline{\mathbb S}
= \text{Sing}_{D^2}{\mathbb S}$.
The Formulas
(\ref{35.78.1}), (\ref{35.78.2}), (\ref{35.81}) and (\ref{35.82}) imply
$$\aligned
&
\dim_{\R}\MM_{k+1}^{\text{\rm main}}(\overline\beta; P_1, \dots,
P_k)(\overline{\mathbb S}) \\
&\le
\dim_{\R}\MM_{k+1}^{\text{\rm main}}(\beta; P_1, \dots, P_k)({\mathbb S}) + 2(\#
\text{Sing}_{S^2}\mathbb S  - \# \text{Sing}_{S^2}\overline{\mathbb S}).
\endaligned$$
On the other hand, we have
$$
\aligned
&
\dim_{\R}\MM_{k+1}^{\text{\rm main}}(\beta; P_1, \dots, P
_k) \\
&=
\dim_{\R}\MM_{k+1}^{\text{\rm main}}(\beta; P_1, \dots, P_k)({\mathbb
S}) + 2 \# \text{Sing}_{S^2}\mathbb
S +
\# \text{Sing}_{D^2}\mathbb
S.
\endaligned
$$
Since $\# \text{Sing}_{S^2}\overline{\mathbb S} \ge 1$, we obtain
the lemma.
\end{proof}
Now we are in the position to complete the proof of
Proposition \ref{dimprop}. Let $\Gamma \ne \{1\}$ be an
abstract group.
By Proposition \ref{35.63} and Lemmas \ref{35.72}, \ref{35.74}, \ref{35.75} we have
\begin{equation}\aligned
\dim_{\R}\MM_{k+1}^{\text{\rm main}}(\beta; P_1, \dots, P_k)(\mathbb S,\Gamma)
&= \dim_{\R}\MM_{k+1}^{\text{\rm main}}(\overline{\beta}; P_1, \dots, P_k)(\overline{\mathbb S})\\
&\le \dim_{\R}\MM_{k+1}^{\text{\rm main}}({\beta}; P_1, \dots, P_k) - 2.
\endaligned\label{saigonosi}\end{equation}
By definition \eqref{35.19}, we have
$$
\max_i  d(\MM_{k+1}^{\text{\rm main}}({\beta}; P_1, \dots, P_k);\Gamma;i)
= \max_{\mathbb S}\dim_{\R}\MM_{k+1}^{\text{\rm main}}(\beta; P_1, \dots, P_k)(\mathbb S,\Gamma).
$$
Here the maximum in the right hand side
is taken over $\mathbb S$ such that
$\MM_{k+1}^{\text{\rm main}}(\beta; P_1, \dots, P_k)(\mathbb S,\Gamma)$
is nonempty.
Now Proposition  \ref{dimprop} follows from (\ref{saigonosi}).
\end{proof}
\begin{rem}\label{35.85}
Consider the case when
$\mathbb S$ is a union of a disc $D^2$ and a sphere $S^2$ where
$c_1(M)[v\vert_{S^2 *}([S^2])]= 0$. We assume that $v\vert_{S^2}$
is a cyclic multiple cover  of a map $u$. The reduced model
consists of the same configuration where the map $v\vert_{S^2}$ is
replaced by $u$. The virtual dimension of the reduced model is the
same as the virtual dimension of $\MM_{k+1}^{\text{\rm
main}}(\beta; P_1, \dots, P_k)({\mathbb S})$, the moduli space with
combinatorial type $\mathbb S$. Since there is one sphere bubble, the
virtual dimension of the moduli space $\MM_{k+1}^{\text{\rm
main}}(\beta; P_1, \dots, P_k)({\mathbb S})$ is the virtual
dimension of $\MM_{k+1}^{\text{\rm main}}(\beta; P_1, \dots,
P_k)$ minus 2.
\par
However, since there is one free fixed point on the
sphere bubble, it follows that the virtual
dimension of {\it marked} reduced model
$\MM_{k+1,1}^{\text{\rm main}}(\overline\beta; P_1, \dots, P_k)(\overline{\mathbb S})$
is {\it equal to} the
virtual dimension of $\MM_{k+1}^{\text{\rm main}}(\beta; P_1, \dots, P_k)$.
Namely (\ref{dimform}) does not hold.
\par
This is the reason why we assume spherical positivity in this paper.
\par
If we can use the reduced model in place of the marked reduced
model, this condition would be removed. However it seems rather
difficult to show the compatibility of the normally conical
perturbation with the
process of forgetting extra marked points in the marked reduced
model.
\end{rem}
\section{Filtered $A_{\infty}$ algebra over $\Z_2$}\label{AinftyconstrZ2}
\subsection{Construction of a filtered $A_{n,K}$ structure over $\Z_2$}
\label{subsec:Ankstr}
We now use Theorem \ref{maintechnicalresult} and Proposition
\ref{dimprop} to prove Theorem \ref{theoremA}
(except  Item (4), that is the case of $\Lambda_{0,{\rm nov}}^{\Z}$ coefficient).
Actually we need the
following relative version of Theorem \ref{maintechnicalresult}. We
say a global section $\mathfrak s' = \{s'_p\}$ of a space with
Kuranishi structure to be {\it normally conical} if each of $s'_p$
is normally conical in the sense of Definition \ref{normalcondef}.
Let $X$ be a space with Kuranishi structure with corners. The
boundary of $\partial X$ is decomposed into spaces $\partial _cX$
with Kuranishi structures such that $\partial _cX$ intersects with
$\partial _{c'}X$ along the corner ($= \partial _{cc'}X$) of $X$ with
codimension $\ge 2$. (More precisely, there exists a map from
$\partial_cX$ to the boundary of $X$ so that it is injective outside
the codimension 2 corner $\partial _{cc}X$. See \cite{fooo-book2}
right before Definition A1.30.)
\par
Let $X$ has a Kuranishi sturucture with corners.
We denote by $\overset{\circ}S_kX$ the set of the points of $X$ that
lies on the codimension $k$ corner.
If $p \in \overset{\circ}S_kX$, then
we may take its Kuranishi neighborhood $V_p/\Gamma_p$ so that
$$
V_p = V'_p \times [0,\rho)^{k}.
$$
(Here $V'_p$ is an open set of the Euclidean space.)
The finite group $\Gamma_p$ acts
on $V_p$ and preserves its stratification as a manifold with corners.
So the action of $\Gamma_p$ exchanges the second factor $[0,\rho)^{k}$.
We write
$$
\partial_cV_p = V'_p \times [0,\rho)^{c-1} \times \{0\} \times  [0,\rho)^{k-c}
$$
%To simplify the proof of the next theorem we assume:
\begin{assumption}\label{cornertrivial}
All the elements of $\Gamma_p$ preserve each of $\partial_cV_p$.
\end{assumption}
\par
We remark that $\MM_{k+1}^{\text{\rm main}}({\beta})$ satisfies
Assumption \ref{cornertrivial}.
In fact in the case of $X = \MM_{k+1}^{\text{\rm main}}({\beta})$
the coordinates $[0,1)^m$ that parametrize the direction normal to the
corner, is the glueing parameter of the boundary singular points.
On the other hand, the element of the isotropy group $\Gamma_p$ is the
automorphism that is supported on  sphere bubbles.
Therefore the $\Gamma_p$ action on $[0,1)^m$ is trivial.
\par
Since $\MM_{k+1}^{\text{\rm main}}({\beta}; P_1, \dots,
P_k)$ is a fiber product of $\MM_{k+1}^{\text{\rm main}}({\beta})$ with
$P_i$, Assumption \ref{cornertrivial} is also satisfied
for $\MM_{k+1}^{\text{\rm main}}({\beta}; P_1, \dots,
P_k)$.
\par
Under Assumption \ref{cornertrivial} we can easily prove the following:
\begin{lem}
The Kuranishi neighborhood and obstruction bundle of a neighborhood of codimension $k$ corner
$\overset{\circ}S_kX$ is the product of the induced Kuranshi structure on
$\overset{\circ}S_kX$ and
$[0,\rho)^k$.
\par
Namely we can take the diffeomorphism
\begin{equation}\label{product1}
V_p \cong V'_p \times [0,\rho)^{k}
\end{equation}
so that it is $\Gamma_p$ equivariant and
is compatible with the coordinate change, for each $p\in \overset{\circ}S_kX$.
Moreover we have an equivariant isomorphism of the obstruction bundles
\begin{equation}\label{product2}
E_p \cong E_p\vert_{V'_p} \times [0,\rho)^{k}.
\end{equation}
\end{lem}
We remark that the Kuranishi map $s_p$ may not be
constant in the $[0,\rho)^{k}$ direction (under the identifications
(\ref{product1}),(\ref{product2}).)

\begin{theorem}\label{reltechnicalresult}
Let $X$ be a space with Kuranishi structure with corners and tangent bundle.
We assume that $X$ satisfies Assumption \ref{cornertrivial}.
Suppose that we have global sections
$\mathfrak s'_{c}$ of $\partial _cX$ for each
$c$ such that they coincide with each other at the intersections
$\partial _{cc'}X$.
We assume $\mathfrak s'_{c}$ are
normally conical and satisfy the conclusion of Theorem \ref{maintechnicalresult}.
Then we can extend them to a normally conical  global section of $X$
satisfying the conclusion of Theorem \ref{maintechnicalresult}.
\end{theorem}
\begin{rem}
Actually we can prove Theorem \ref{reltechnicalresult} without assuming Assumption \ref{cornertrivial}.
We assume it only to simplify the proof. In our application this assumption is satisfied.
\end{rem}
\begin{proof}
We put
$$
S_kX = \bigcup_{m\ge k}\overset{\circ}S_mX.
$$
Let $U_{\epsilon}(S_kX)$ be the $\epsilon$ neighborhood of $S_kX$.
(Here we take a metric on $X$ and fix it.)
Let $N$ be the smallest positive integer such that $S_{N+1}X$ is empty.
We will prove the following by induction on $m_0
= 0,1,\dots,N$.
\begin{lem}\label{nearthecorner}
There exists $\epsilon_{k,m}$ for $m\le m_0$ and $k=N-m,\dots,N$
and a normally conical   global section $\frak s_{m}$ on a Kuranishi neighborhood of
\begin{equation}\label{domainskm}
\bigcup_{k=N-m}^{N} U_{\epsilon_{k,m}}(S_kX)
\end{equation}
such that they have the following properties.
\begin{enumerate}
\item
$\frak s_{m}$ satisfies the conclusion of Theorem \ref{maintechnicalresult}.
\item
The restriction of $\frak s_{m}$ to
the intersection of (\ref{domainskm}) and $\partial_c X$
coincides with the restriction of $\frak s'_c$.
\item
$\epsilon_{k,m+1} < \epsilon_{k,m}$ if $k \ge N-m$.
Moreover $\frak s_{m+1} = \frak s_{m}$ on  a Kuranishi neighborhood of
$
\bigcup_{k=N-m}^{N} U_{\epsilon_{k,m+1}}(S_kX).
$
\end{enumerate}
\end{lem}
\begin{proof}
We first prove the first step $m_0 =0$ of the induction.
We need to find $\epsilon_{N,0}$ and
$\frak s_{0}$ on $U_{\epsilon_{N,0}}(S_NX)$.
By assumption $S_NX$ has a
Kuranishi structure {\it without} boundary.
The restriction of various $\frak s'_c$ to $S_NX$
gives a global section on $S_NX$ that satsifies
the conclusion of Theorem \ref{maintechnicalresult}.
We will extend it to a neighborhood of $S_NX$.
\par
We take an isomorphism
\begin{equation}\label{isoneibh}
F : S_NX \times [0,1)^N \to X
\end{equation}
to an open neighborhood. (More precisely
(\ref{isoneibh}) is a family of  isomorphisms
(\ref{product1}),(\ref{product2}) that are $\Gamma_p$ equivariant and
compatible with the coordinate change.
\par
We remark that
\begin{equation}\label{cornerto}
\left([0,1)^{N},\bigcup_{c=1}^N \left( [0,1)^{c-1} \times \{0\}  \times [0,1)^{N-c}\right)\right)
\end{equation}
is PL isomorphic to
\begin{equation}\label{kyokaiha}
\left([0,1) \times B^{N-1},\{0\}\times B^{N-1} \right)
\end{equation}
as a pair.
Here
$B^{N-1} = \{x \in \R^{N-1} \mid \vert x\vert < 1\}$
is an open ball of dimension $N-1$.
\par
We remark that
we are given a family of global sections $\frak s'_c$ on a Kuranishi neighborhood of
$$
S_NX \times \bigcup_{c=1}^N \left([0,1)^{c-1} \times \{0\}  \times [0,1)^{N-c}\right)
$$
so that they are compatible on the overlapped part.
So we use the isomorphism between (\ref{cornerto})
and (\ref{kyokaiha}) to extend the section $\frak s'_c$ to a Kuranishi neighborhood of
$$
S_NX \times [0,1) \times B^{N-1}
\cong
S_NX \times [0,1)^{N}
$$
so that it coicides with $\frak s'_c$ on (a Kuranishi neighborhood of)
$$
S_NX \times \{0\} \times B^{N-1}
\cong
S_NX \times  \bigcup_{c=1}^N \left([0,\rho)^{c-1} \times \{0\}  \times [0,\rho)^{N-c}\right).
$$
Namely we take the extension so that it is constant in the $[0,1)$
factor.
Then it is easy to see that extended section has the required properties.
We have thus proved the case $m_0=0$ of Lemma \ref{nearthecorner}.
\par
We next discuss the induction step for $m_0 < N$.
We remark that
$\overset{\circ}{S}_{N-m_0}X$ has a Kuranishi structure without boundary.
We take $\epsilon_{k,m_0} < \epsilon_{k,m_0-1}$ for $k > N-m_0$.
We put
$$
U(m_0-1)
=
\bigcup_{k=N-m_0+1}^{N} U_{\epsilon_{k,m_0-1}}(S_kX)
$$
and
$$
U^-(m_0-1)
=
\bigcup_{k=N-m_0+1}^{N} U_{\epsilon_{k,m_0}}(S_kX)
$$
and
$$
S_{N-m_0}^-X
=
S_{N-m_0}X \setminus U^-(m_0-1).
$$
We have a piecewise smooth isomorphism between
a Kuranishi neighborhood in $X$ of $S_{N-m_0}^-X$ and the product of
a Kuranishi neighborhood (in $\overset{\circ}{S}_{N-m_0}X$) of $S_{N-m_0}^-X$ and $[0,1)^{m_0}$
in the sense above.
\par
We also have a piecewise smooth isomorphism between
Kuranishi neighborhoods of
\begin{equation}\label{76form}
S_{N-m_0}^-X \times [0,1)^{m_0}
\end{equation}
and of
\begin{equation}\label{77form}
S_{N-m_0}^-X \times [0,1) \times B^{m_0-1}
\end{equation}
so that $S_{N-m_0}^-X \times \{0\} \times B^{m_0-1}$
becomes the intersection with $\bigcup_c\partial_cX$
under the identifications.
\par
By induction hypothesis we have $\frak s_{m_0-1}$.
We take its restriction to (a Kuranishi neighborhood of)
$$
U(m_0-1) \cap
\left(S_{N-m_0}^-X \times [0,1)^{m_0}\right).
$$
We may identify a Kuranishi neighborhood of
$
U(m_0-1) \cap
\left(S_{N-m_0}^-X \times [0,1)^{m_0}\right)
$ with a Kuranishi neighborhood of
$\left(U(m_0-1) \cap
S_{N-m_0}^-X\right)
\times  [0,1)^{m_0}.
$
By construction we also have a piecewise smooth diffeomorphism
\begin{equation}\label{form78}
\aligned
&\left(U(m_0-1) \cap
S_{N-m_0}^-X\right)
\times  [0,1)^{m_0}
\\
&\cong
\left(U(m_0-1) \cap
S_{N-m_0}^-X\right)
\times  [0,1) \times B^{m_0-1}
\endaligned
\end{equation}
such that $\frak s_{m_0-1}$ is constant in the $[0,1)$ factor
by this identification.
\par
By changing the identification between (\ref{76form}) and (\ref{77form})
if necessary we may assume that the identification between (\ref{76form}) and (\ref{77form}) coincides with (\ref{form78})
except in a small neighborhood of
$\partial U(m_0-1) \cap
S_{N-m_0}^-X$.
\par
Thus we define $\frak s_{m_0}$ so that it is constant in the $[0,1)$ factor
of  (\ref{77form}) under the above identifications.
Then it is easy to see that extended section has the required properties.
We remark that we use $m_0 < N$ here.
In fact in case $m_0 = N$,  the stratum $S_{0}^-X$ is not contained in
any of $\partial_cX$ so we can not define $\frak s_{m_0}$ by
taking it independent of $[0,1)$ factor.
\par
We have thus proved the lemma for $m_0 < N$.
So we have defined a global section on a (Kuranishi) neighborhood
of the boundary $\partial X$. We can then extend it by
the following relative version Propositin \ref{maintechnicalresultrel} of Theorem \ref{maintechnicalresult}.
It proves the case of $m_0 = N$.
The proof of Lemma \ref{nearthecorner} is complete.
\end{proof}
By taking $m_0 = N$ Lemma \ref{nearthecorner}
implies Theorem \ref{reltechnicalresult}.
\end{proof}
\begin{prop}\label{maintechnicalresultrel}
Let $X$ be a space with Kuranishi structure that has a tangent bundle, $K\subset X$ a
compact subset and $U$ a Kuranishi neighborhood of $K$.
\par
We take a good coordinate system of $X$.
\par
Suppose we have a sequence of normally conical strongly piecewise smooth global sections $\frak s_{\epsilon}$ on $U$
such that it satisfies Theorem \ref{maintechnicalresult} {\rm (i)(ii)(iii)} and converges to the
Kuranishi map in $C^0$ sense as $\epsilon \to 0$.
\par
Then there exists a sequence of normally conical strongly piecewise smooth global sections $\mathfrak s'_{\epsilon}$
on $X$ such that it satisfies Theorem \ref{maintechnicalresult} {\rm (i)(ii)(iii)},
coincides with $\frak s_{\epsilon}$ on a Kuranishi neighborhood of $K$ and converges to the
Kuranishi map in $C^0$ sense as $\epsilon \to 0$.
\par
Moreover if $f: X \to Y$ is a strongly smooth map from $X$ to a
manifold $Y$, and if we have a triangulation of an intersection of $(\frak s_{\epsilon})^{-1}(0)$
with a Kuranishi neighborhood of $K$,
such that $f$ induces a piecewise smooth map $(\frak s_{\epsilon})^{-1}(0) \to Y$, then
we can extend the triangulation to $(\frak s'_{\epsilon})^{-1}(0)$  so that
$f$ induces a piecewise smooth map $(\frak s'_{\epsilon})^{-1}(0)  \to Y$.
\end{prop}
We remark that Theorem \ref{maintechnicalresult} is proved by induction
using Proposition \ref{35.20'}, which is a relative version of Proposition \ref{35.20}.
(See Section \ref{section35.4}.)
So we can prove Proposition \ref{maintechnicalresultrel} in the same way as
Theorem \ref{maintechnicalresult}.
\par\smallskip
We can use Theorem \ref{reltechnicalresult} in place of Theorem
A1.23 \cite{fooo-book2} (which is a similar result for multisection)
together with Proposition \ref{dimprop} to prove Theorems
\ref{theoremA}, \ref{theoremF} in the same way as in Chapter 7
\cite{fooo-book2}. Below we explain this construction, which is
parallel to Section 7.2 \cite{fooo-book2}.
\par
Now we consider $\MM_{k+1}^{\text{\rm main}}({\beta}; P_1, \dots,
P_k)$. By  Proposition 7.1.2 \cite{fooo-book2}, its boundary is
decomposed into the following:
\begin{subequations}
\begin{equation}\label{boundary1}
\MM_{k+1}^{\text{\rm main}}({\beta}; P_1, \dots, \partial_iP_j,\dots,P_k)
\end{equation}
where $\partial_iP_j$ is the $i$-th factor of the singular simplex $P_j$ and
\begin{equation}\label{boundary2}
(\MM_{k_1+1}^{\text{\rm main}}({\beta_1}) _{ev_0}\times_{ev_i}
\MM_{k_2+1}^{\text{\rm main}}({\beta_2}))
\,\,{}_{ev'}\times_{L^{k}} (P_1 \times \cdots \times P_k)
\end{equation}
\end{subequations}
where $k_1+k_2=k+1$ and $i \in \{1,\dots,k_2\}$. The evaluation map
$$
ev'=(ev'_1,\dots,ev'_k): \MM_{k_1+1}^{\text{\rm main}}({\beta_1})
_{ev_0}\times_{ev_i} \MM_{k_2+1}^{\text{\rm main}}({\beta_2}) \to
L^{k},
$$
which we use to take the fiber product in (\ref{boundary2}), is
defined by
$$
ev'_j(\mathfrak x,\mathfrak y) =
\begin{cases}
ev_j(\mathfrak x)  &\text{if $j < i$},\\
ev_{j-i+1}(\mathfrak y) &\text{if $i \le j < i+k_2$}, \\
ev_{j-k_2+1}(\mathfrak x) &\text{if $ j \ge i+k_2$}.
\end{cases}
$$
Therefore (\ref{boundary2}) is identified with
\begin{equation}\label{boundary3}
\aligned
\MM_{k_1+1}^{\text{\rm main}}({\beta_1})
\times_{L^{k_1}} (&P_1 \times \cdots \times P_{i-1} \\
& \times\MM_{k_2+1}^{\text{\rm main}}({\beta_2};P_{i},\dots,P_{i+k_2-1}) \times P_{i+k_2}
\times \cdots \times P_k).
\endaligned
\end{equation}
\begin{prop}\label{mkconstruct}
Let $E$ and $K$ be fixed integers. Then there is a system of global
sections $\mathfrak s'_{\beta,P_1, \dots, P_k}$ of
$\MM_{k+1}^{\text{\rm main}}({\beta}; P_1, \dots, P_k)$ for each
$k\le K$ and $\beta$ with $\omega(\beta) \le E$ with the following
properties:
\begin{enumerate}[{\rm(i)}]
\item
They are normally conical and satisfy the conclusion of Theorem \ref{maintechnicalresult}.
\item They are compatible at the boundary component (\ref{boundary1}).
\item The zero set of each section $\mathfrak s'_{\beta,P_1, \dots, P_k}$ is
given with triangulation. Each simplex of it is identified with a
smooth singular chain of $L$ by the (strongly smooth extension of)
the map $ev_0$.
\item Let $P'$ be a simplex of
the zero set of $\mathfrak s'_{\beta,P_{i}, \dots, P_{i+k_2-1}}$.
Then the restriction of  $\mathfrak s'_{\beta,P_1, \dots, P_k}$ at
the subset
$$
\MM_{k_1+1}^{\text{\rm main}}({\beta_1})
\times_{L^{k_1}} (P_1 \times \cdots \times P_{i-1} \times
P' \times P_{i+k_2}
\times \cdots \times P_k)
$$
of (\ref{boundary3}) coincides with $\mathfrak s'_{\beta_1,P_1, \dots, P_{i-1},P',P_{i+k_2},\dots ,P_k}$.
\end{enumerate}
\end{prop}
\begin{rem}\label{impreciseobst}
The way how the statement (iv) above is made is slightly imprecise.
Namely we need to describe the relation between obstruction bundles
to make a precise statement. See Proposition 7.2.35
\cite{fooo-book2}.
\end{rem}
\begin{proof}
The proof of Proposition \ref{mkconstruct} is by induction.
We consider the order on the set of pairs $(\beta,k)$
such that $(\beta,k) < (\beta',k')$ if
either (1) $\omega[\beta] < \omega[\beta']$ or
(2) $\omega[\beta] = \omega[\beta']$, $k<k'$.
\par
The conditions (ii) and (iv) describe the global section $\mathfrak
s'_{k,\beta,P_1, \dots, P_k}$ at the boundary. We can check that
they are consistent at the corners of codimension $\ge 2$. (See
 Lemma 7.2.55 \cite{fooo-book2}.) Then by Theorem \ref{maintechnicalresult} we can
extend it to the whole moduli space $\MM_{k+1}^{\text{\rm
main}}({\beta}; P_1, \dots, P_k)$. We refer Section 7.2
\cite{fooo-book2} for the detail.
\end{proof}

Now we fix a system of global sections
$\mathfrak s'_{\beta,P_1,
\dots, P_k}$ as in Proposition \ref{mkconstruct} for each
$k,\beta,P_1,\dots,P_k$ with $k\le K$ and $\omega(\beta) \le E$.
Let $\MM_{k+1}^{\text{\rm main}}({\beta}; P_1, \dots,
P_k)^{\mathfrak s'}$ be the zero set  of
$\mathfrak s'_{\beta,P_1,
\dots, P_k}$. We fixed its triangulation in (iii).
\par
\emph{Up to this point, spherical positivity is not assumed}. Now we
assume that $(M,\omega,J)$ is spherically positive. Then Proposition
\ref{dimprop} implies the following.

\begin{lem}\label{dimensionrest}
If $P'$ is a simplex of $\MM_{k+1}^{\text{\rm main}}({\beta}; P_1, \dots, P_k)^{\mathfrak s'}$
with
$$
\dim P' \ge \dim \MM_{k+1}^{\text{\rm main}}({\beta}; P_1, \dots, P_k) -2
$$
then the interior of
$P'$ does not intersect with
$\MM_{k+1}^{\text{\rm main}}({\beta}; P_1, \dots, P_k)(\Gamma)$
for $\Gamma \ne \{1\}$.
\end{lem}
Let $P'_a$ ($a \in A$) be the set of all simplices of dimension
$\MM_{k+1}^{\text{\rm main}}({\beta}; P_1, \dots, P_k)$.
We put
\begin{equation}\label{virtualfundamental}
ev_{0*}\left[\MM_{k+1}^{\text{\rm main}}({\beta}; P_1, \dots,
P_k)^{\mathfrak s'}\right] = \sum_a P'_a
\end{equation}
this is a smooth singular chain of $L$ with $\Z_2$ coefficients.
We now put
\begin{equation}\label{mkdefinitionprime}
\mathfrak m_{k,\beta}(P_1,\dots,P_k)
=  ev_{0*}[\MM_{k+1}^{\text{\rm main}}(\beta;P_1,\dots,P_k)^{\mathfrak s'}].
\end{equation}
We use Lemma \ref{dimensionrest} and Proposition \ref{mkconstruct}
(i), (iii) to obtain the following formula:
\begin{equation}\label{Ankformula}
\aligned
&\partial \mathfrak m_{k,\beta}(P_1,\dots,P_k)
+ \sum_i \mathfrak m_{k,\beta}(P_1,\dots,\partial P_i,\dots P_k)
\\
&=  \sum_{\beta_1+\beta_2=\beta}\sum_{k_1+k_2=k+1}\sum_{i=1,\dots,k_2}\mathfrak m_{k_1,\beta_1}(P_1,\dots,P_{i-1},\\
&\qquad\qquad\qquad\qquad\qquad\qquad\qquad\mathfrak m_{k_2,\beta_2}(P_i,\dots,P_{i+k_2-1}),
P_{i+k_2},\dots,P_k),
\endaligned
\end{equation}
where we put $\mathfrak m_{1,\beta_0} = \partial$ ($\beta_0 = 0$ is
zero in $H_2(M;L)$). Taking the sum over $\beta$, (\ref{Ankformula})
gives rise to the $A_{\infty}$ formula which defines an $A_{\infty}$ algebra.
More precisely, the formula gives rise to a filtered
$A_{n,K}$ structure in the sense defined in Section 7.2
\cite{fooo-book2}. Here the integer $K$ is as above and $n$ is
determined by $E$ as follows. We consider the subset of $\R$ defined
by
$$
\{ (\omega[\beta],\mu_L(\beta)) \mid \MM_1(\beta) \ne \emptyset \} \subset \R \times \Z.
$$
We take the additive submonoid of $\R \times \Z$ generated by this set
and denote it by $\mathfrak G(L,J)$.
For an element $\beta \in \mathfrak G(L,J)$ we put
$$
\Vert \beta\Vert = \sup \left\{ n \,\left\vert\, \exists \beta_i \in \mathfrak G(L,J),
\,\,\, \sum_{i=1}^n \beta_i = \beta \right\}.\right.
$$
We say $(n,K) < (n',K')$ if either (1) $n+K < n'+K'$ or (2) $n+K =
n'+K'$, $n<n'$. Filtered $A_{n,K}$ structure is by definition a
system of operators $\mathfrak m_{k,\beta}$ which are for $k\le K$
and $\beta$ with $(\Vert \beta\Vert,k) \le (n,K)$ such that it
satisfies (\ref{Ankformula}). (See Definition 7.2.70
\cite{fooo-book2}.)
\begin{rem}
See Subsection 7.2.3 \cite{fooo-book2} for the reason why we need to
restrict to $k\le K$ and $\omega(\beta) \le E$ for the proof of
Proposition \ref{mkconstruct}.
\end{rem}
\subsection{Construction of a filtered $A_{n,K}$ homomorphism over $\Z_2$}\label{subsec:Ankhom}
We can work in the same way as Section 7.2 \cite{fooo-book2} to go
from an $A_{n,K}$ structure (for arbitrary but fixed $n,K$) to
an $A_{\infty}$ structure. For this purpose we should also use the
construction of filtered $A_{n,K}$ homomorphisms which we describe
later in this subsection. See Lemma \ref{extobs}. In this way we
associate a filtered $A_{\infty}$ structure to the singular chain
complex of a Lagrangian submanifold over $\Z_2$ coefficients in the
spherically positive case.
Then it was proved in  Theorem 5.4.2 \cite{fooo-book1} that this
$A_\infty$ structure on the singular chain complex induces a
canonical filtered $A_{\infty}$ structure on the cohomology group,
which we call the {\it canonical model}.
(We use Lemma \ref{canolem} for this purpose.)
To show that the canonical model depends only on the connected
component of $J \in \JJ_{(M,\omega)}^{c_1 > 0}$  up to an isomorphism we
proceed as follows. Let $J_0,\, J_1$ be elements of $\JJ_{(M,\omega)}^{c_1
> 0}$. We use each one of them to construct our filtered
$A_{\infty}$ structure on singular chain complex. (We also fix a
system of global sections, triangulation of their zero sets etc..)
Let $\{J_{\rho}\}_{\rho\in [0,1]}$ be a path in $\JJ_{(M,\omega)}^{c_1 >
0}$ joining them. We now consider the moduli space
$\MM_{k+1}^{\text{main}}
(\{J_{\rho}\}_{\rho}:\beta;\operatorname{top}(\rho);
P_1,\dots,P_k)$ which was introduced by Definition 4.6.10
\cite{fooo-book1}.
\par
We here recall its definition.
Let $(\Sigma,\vec z)$ be a bordered Riemann surface of genus $0$ with
$k$ of boundary marked points.
We consider each of disc component of $\Sigma$ and attach tree of
sphere components rooted on it. We thus decompose
$$
\Sigma = \bigcup_{a=0}^{m} \Sigma_a
$$
where each of $\Sigma_a$ contains exactly one disc components.
Let $\Sigma_0$ be the component which contains $0$-th marked
point $z_0$. We define a partial order
$<$ on the set of components as follows:
$\Sigma_a < \Sigma_b$ if every path joining $\Sigma_a$ to $z_0$
intersect $\Sigma_b$.
We consider $\rho_a \in [0,1]$ for each $a=0,\dots,m$ such that
if $\Sigma_a < \Sigma_b$ then
$\rho_a \le \rho_b$. We call $\{\rho_a\}_{a=0}^{m}$ the
{\it time allocation}.
We consider the set of $((\Sigma,\vec z),\{\rho_a\},v)$
such that
\begin{enumerate}
\item $(\Sigma,\vec z)$  is a bordered Riemann surface of genus $0$ with
$k$ of boundary marked points.
\item $\{\rho_a\}_{a=0}^{m}$ is its time allocation.
\item $v: (\Sigma,\partial \Sigma) \to (M,L)$ is a continuous map.
\item On each component $\Sigma_a$ the restriction of $v$ to
$\Sigma_a$ is $J_{\rho_a}$-holomorphic.
\item We assume the stability in the following sense:
the set of biholomorphic map $\varphi: \Sigma \to \Sigma$ such that
$\varphi(z_i) = z_i$, $v \circ \varphi =v$ is finite.
\end{enumerate}
We say $((\Sigma,\vec z),\{\rho\},v) \sim ((\Sigma',\vec
z'),\{\rho'_{a'}\},v')$ if there exists a biholomorphic map
$\varphi:  \Sigma' \to \Sigma$ such that $\varphi(z'_i) = z_i$,
$v\circ \varphi = v'$ and $\rho_a = \rho_{a'}$, where
$\varphi(\Sigma_{a'}) = \Sigma_a$.
We denote by $\MM_{k+1}^{\text{main}}
(\{J_{\rho}\}_{\rho}:\beta;\operatorname{top}(\rho))$ the set of
$\sim$ equivalence classes of the triple $((\Sigma,\vec
z),\{\rho_a\},v)$ satisfying (1)-(5).
Using the evaluation map
$$
ev = (ev_0,\dots,ev_k): \MM_{k+1}^{\text{main}}
(\{J_{\rho}\}_{\rho}:\beta;\operatorname{top}(\rho)) \to L^{k+1},
$$
we obtain $\MM_{k+1}^{\text{main}}
(\{J_{\rho}\}_{\rho}:\beta;\operatorname{top}(\rho);
P_1,\dots,P_k)$
for smooth singular chains $P_1,\dots,P_k$.
It is  a space with Kuranishi structure with
tangent space.
Its boundary is decomposed into three parts which correspond to
the following 5 possibilities
\smallskip
\begin{enumerate}[{\rm (A)}]
\item The moduli space
\begin{equation}\label{boundaryA}
\MM_{k+1}^{\text{main}}
(\{J_{\rho}\}_{\rho}:\beta;\operatorname{top}(\rho);
P_1,\dots,\partial_iP_j,\dots,P_k)
\end{equation}
where $\partial_iP_j$ is the $i$-th face of $P_j$.
\item One of the time allocation becomes $0$.
\item One of the time allocation becomes $1$.
\item Two time allocation $\rho_a$ and $\rho_b$ coincides,
and corresponding component $\Sigma_a$ and $\Sigma_b$ intersects.
\item One disc component splits into two in the limit.
\end{enumerate}
One can prove (D) and (E) above cancel to each other. (See  Figures
4.6.2 and 4.6.3 \cite{fooo-book1}.) Therefore the actual boundary is
only (A),(B),(C) above.
\par
We next describe (B) and (C) by fiber product. (See Subsection 7.2.9
\cite{fooo-book2} for details.)
\par
(B) is described as the
union of the following fiber product over
$k_1,k_2,\beta_1,\beta_2,i$ with
$k_1+k_2=k-1$, $\beta_1+\beta_2=\beta$, $i =1,\dots,k_2$:
\begin{equation}\label{boundary4}
\aligned
&\MM_{k_1+1}^{\text{\rm main}}(\{J_{\rho}\}_{\rho}:\beta_1;\operatorname{top}(\rho))
\\
&{}\times_{L^{k_1}}  (P_1 \times \cdots P_{i-1} \times
\MM_{k_2+1}^{\text{\rm main}}(J_0,{\beta_2};P_{i},\dots,P_{i+k_2-1}) \times P_{i+k_2}
\times \cdots \times P_k).
\endaligned
\end{equation}
Here we write
$\MM_{k_2+1}^{\text{\rm main}}(J_0,{\beta_2};P_{i},\dots,P_{i+k_2-1})$
in place of
$\MM_{k_2+1}^{\text{\rm main}}({\beta_2};P_{i},\dots,P_{i+k_2-1})$
to clarify the almost complex structure we use for definition.
\par
(C) is described as follows.
We consider $1 \le \ell_1 \le \dots \le \ell_m \le \ell_{m+1} = k$.
Then (C) is a union of the following fiber products:
\begin{equation}\label{boundary5}
\aligned
\MM_{m+1}^{\text{\rm main}}(J_1,{\beta'})
\times_{L^m}
\prod_{i=1}^m
\MM_{\ell_{i+1}-\ell_i+1}^{\text{\rm main}}(\{J_{\rho}\}_{\rho}:\beta_i;\operatorname{top}(\rho);
P_{\ell_i},\dots,P_{\ell_{i+1}-1}).
\endaligned
\end{equation}
Here we take union over $m,\{\ell_1,\dots,\ell_{m+1}\},\beta',\beta_i$
such that $\ell_i$ is as above and $\beta' + \sum \beta_i = \beta$.
Note we include the case $\ell_{i+1} = \ell_i+1$ and $\beta_i = 0\in H_2(M,L;\Z)$.
In that case we put
\begin{equation}\label{becomeidentity}
\MM_{1}^{\text{\rm main}}(\{J_{\rho}\}_{\rho}:0;\operatorname{top}(\rho);
P_{\ell_i})
= P_{\ell_i}
\end{equation}
{\it by definition}.
\par\smallskip
In a way similar to Proposition \ref{mkconstruct} we can prove the
following.
Let $E$ and $K$ be fixed integers. We take and fix
a system of global sections
$\mathfrak s'_{J_0,\beta,P_1, \dots, P_k}$ of
$\MM_{k+1}^{\text{\rm main}}(J_0,{\beta}; P_1, \dots, P_k)$
and
$\mathfrak s'_{J_1,\beta,P_1, \dots, P_k}$ of
$\MM_{k+1}^{\text{\rm main}}(J_1,{\beta}; P_1, \dots, P_k)$
as in Proposition \ref{mkconstruct}.
\par
\begin{prop}\label{fkconstruct}
There is a system of global
sections $\mathfrak s'_{\{J_{\rho}\},\beta,P_1, \dots, P_k}$ of
$$
\MM_{k+1}^{\text{\rm main}}(\{J_{\rho}\}_{\rho}:{\beta};\operatorname{top}(\rho); P_1, \dots, P_k)
$$
for each $k\le K$ and $\beta$ with $\omega(\beta) \le E$
with the following properties.
\begin{enumerate}[{\rm(i)}]
\item
They are normally conical and satisfy the conclusion of Theorem \ref{maintechnicalresult}.
\item They are compatible at the boundary component (\ref{boundaryA}).
\item Each simplex $P'$
of the zero set of $\mathfrak s'_{J_0,\beta_1,P_{i+1}, \dots, P_{i+k_2-1}}$ is identified with a smooth singular chain of $L$ by the (strongly
smooth extention of) the map $ev_0$. Then the restriction of
 $\mathfrak s'_{\{J_{\rho}\},\beta,P_1, \dots, P_k}$ to
\begin{equation}\label{boundary7}
\MM_{k_1+1}^{\text{\rm main}}(\{J_{\rho}\}_{\rho}:\beta_1;\operatorname{top}(\rho);P_1,\dots.P_{i-1},P',
P_{i+k_2-1},\dots,P_k)
\end{equation}
coincides with
 $\mathfrak s'_{\{J_{\rho}\},\beta_1,P_1,\dots.P_{i-1},P',
P_{i+k_2-1},\dots,P_k}$.
Here (\ref{boundary7}) is a part of (\ref{boundary4}).
\item
Each of the
zero set of the global section
$\mathfrak s'_{\{J_{\rho}\}_{\rho},\beta_i,
P_{\ell_i},\dots,P_{\ell_{i+1}-1}}$ of
$$
\MM_{\ell_{m+1}-\ell_m+1}^{\text{\rm main}}(\{J_{\rho}\}_{\rho}:\beta_i;\operatorname{top}(\rho);
P_{\ell_i},\dots,P_{\ell_{i+1}-1})
$$
is given with triangulation.
Let $P'_i$ be a simplex of the zero set of $\mathfrak s'_{\{J_{\rho}\}_{\rho},\beta_i,
P_{\ell_i},\dots,P_{\ell_{i+1}-1}}$.
Then the restriction of  $\mathfrak s'_{\{J_{\rho}\}_{\rho},\beta,P_1, \dots, P_k}$ to
the corresponding subset of
(\ref{boundary5}) coincides with $\mathfrak s'_{J_0,\beta_0,P'_1, \dots, P'_m}$.
\end{enumerate}
\end{prop}
Note the same remark as Remark \ref{impreciseobst} applies to
(iii) and (iv) above.
\par
The proof of Proposition \ref{fkconstruct} is similar to the proof
of Proposition \ref{mkconstruct}. In other words it is similar to
the proof of Proposition 7.2.100 \cite{fooo-book2}. We use global
single valued sections (and Theorem \ref{reltechnicalresult}) here
in place of multisections used in \cite{fooo-book2}.
\begin{lem}\label{dimdropf}
We assume $J_{\rho} \in  \JJ_{(M,\omega)}^{c_1 > 0}$ for each $\rho$. Let
$\Gamma$ be a nontrivial finite group. Then we have
\begin{equation}\label{dimformrho}
\aligned
&\dim \MM^{\text{\rm main}}_{k+1}(\{J_{\rho}\}_{\rho}:\beta;\operatorname{top}(\rho);P_1, \dots, P_k) \\
&- d( \MM^{\text{\rm main}}_{k+1}(\{J_{\rho}\}_{\rho}:\beta;\operatorname{top}(\rho); P_1, \dots, P_k);\Gamma;i) \ge 2.
\endaligned
\end{equation}
\end{lem}
See \eqref{35.19b} for the notation of the second term.
The proof is the same as the proof of Proposition \ref{dimprop}.
\par
It implies that for a global section of $\MM_{k+1}^{\text{main}}
(\{J_{\rho}\}_{\rho}:\beta;\operatorname{top}(\rho);
P_1,\dots,P_k)$ satisfying the conclusion of Theorem
\ref{reltechnicalresult}, the simplex which lies in the fixed point
locus of some non-trivial group $\Gamma$ is codimension $\ge 2$.
Thus we can define its virtual fundamental chain over $\Z_2$. We use it
to define
\begin{equation}\label{defnf}
\mathfrak f_{k,\beta}(P_1,\dots,P_k)
= ev_{0*}[\MM_{k+1}^{\text{main}}
(\{J_{\rho}\}_{\rho}:\beta;\operatorname{top}(\rho);
P_1,\dots,P_k)^{\mathfrak s'}]
\end{equation}
in the same way as (\ref{virtualfundamental}).
We put also
$$
\mathfrak f_{1,\beta_0} = \operatorname{identity}
$$
for $\beta_0 =0 \in H_2(M,L;\Z)$. (Compare (\ref{becomeidentity}).)
Then we have the following formula:
(In the next formula we write $\mathfrak m_{k,\beta}^{J_0}$, $\mathfrak m_{k,\beta}^{J_1}$
to distinguish two filtered $A_{n,K}$ structures which depend almost
complex structures $J_0$, $J_1$ together with other choices made.)
\begin{equation}\label{formulaAinftymap}
\aligned
&\sum \mathfrak m^{J_1}_{\beta',m}\left(
\mathfrak f_{\beta_1,\ell_2-\ell_1+1}(P_1,\dots,P_{\ell_1})
,\dots, \mathfrak f_{\beta_m,\ell_{m+1}-\ell_{m}+1}(P_m,\dots,P_{k})\right) \\
&= \sum \mathfrak f_{k_1,\beta_1}(P_1,\dots,P_{i-1},
\mathfrak m^{J_0}_{k_2,\beta_2}(P_i,\dots,P_{i+k_2-1}),P_{i+k_2},\dots,P_k).
\endaligned
\end{equation}
Here the sum in the right hand side is taken over
$k_1,k_2,\beta_1,\beta_2,i$ with
$k_1+k_2=k-1$, $\beta_1+\beta_2=\beta$, $i =1,\dots,k_2$
and the sum in the left hand side is taken over all
$m,\{\ell_1,\dots,\ell_{m+1}\},\beta',\beta_i$
such that $1 \le \ell_1 \le \cdots \le \ell_m \le \ell_{m+1} = k$ and $\beta' + \sum \beta_i = \beta$.
\par
In fact, the left hand side except those for $\beta' =0$, $m=1$ corresponds
(C) (in other words (iv) Proposition \ref{fkconstruct}).
The right hand side except those for $\beta_2=0$, $k_2=0$
corresponds (B) (in other words (iii) Proposition \ref{fkconstruct}).
The right hand side for $\beta_2=0$, $k_2=0$ corresponds (A)
(in other words (i) Proposition \ref{fkconstruct}).
Compatibility spelled out in Proposition \ref{fkconstruct} and
Lemma \ref{dimdropf} implies that the boundary
of $\mathfrak f_{\beta,k}(P_1,\dots,P_k)$,
that is nothing but the left hand side for $\beta' =0$, $m=1$,
is equal to the sum of the terms corresponding to (A),(B),(C).
It implies (\ref{formulaAinftymap}).
\par
(\ref{formulaAinftymap}) implies that
$$
\mathfrak f_k = \sum_{\beta} T^{\omega[\beta]}e^{\mu_L(\beta)/2}\mathfrak f_{k,\beta}
$$
define a filtered $A_{n,K}$ homomorphism. By definition $\mathfrak
f_1 \equiv $ identity mod $\Lambda_{0,{\rm nov}}^+$, where
$\Lambda_{0,{\rm nov}}^+$ is the set of sums (\ref{eq:nov0}) such
that all $\lambda_i$ are strictly positive. Therefore Theorem 4.2.45
 \cite{fooo-book1} implies that $\mathfrak f_k$ defines a homotopy
equivalence of filtered $A_{n,K}$ structures.
\subsection{From $A_{n,K}$ structure to $A_{\infty}$ structure}\label{subsec:AnktoAinfty}
Now we go back to the construction of
filtered $A_{\infty}$ structure.
We recall $(n,K) < (n',K')$ if either (1) $n+K < n'+K'$ or
(2) $n+K = n'+K'$, $n<n'$.
\begin{lem}\label{extobs}
If $(C,\mathfrak m_{k,\beta})$ (resp. $(C',\mathfrak m'_{k,\beta})$)
is a filtered $A_{n,K}$ (resp. $A_{n',K'}$) algebra with $(n,K) <
(n',K')$. Suppose there exists a filtered $A_{n,K}$ homomorphism
$\mathfrak f: C \to C'$ which is a homotopy equivalence in the sense
of filtered $A_{n,K}$ homomorphism. Then we can extend the filtered
$A_{n,K}$ structure of $C$ to a filtered $A_{n',K'}$ structure without changing
operations which already exist in the filtered $A_{n,K}$ structure. We
then can also extend $\mathfrak f$ to a homotopy equivalence in the
sense of filtered $A_{n',K'}$ homomorphism.
\end{lem}
This is  Theorem 7.2.72 \cite{fooo-book2}. We already explained the
construction of filtered $A_{n,K}$ structure on the smooth singular
chain complex of $L$ for arbitrary but fixed $n,K$. We also proved
that they are homotopy equivalent to each other. So we can use Lemma
\ref{extobs} to construct a filtered $A_{\infty}$ structure. (See Subsection
7.2.9 \cite{fooo-book2} for detail of this construction.)
\par
Then as we mentioned before we obtain a filtered $A_{\infty}$
structure on the cohomology group $H(L;\Lambda_0^{\Z_2})$.
We remark that this is ordinary cohomology group and is not
a Floer cohomology. Namely we take the boundary operator with
respect to the usual boundary operator which does not
include $\mathfrak m_{1,\beta}$ for $\beta\ne 0$.
\par
By construction the filtered $A_{\infty}$ algebra obtained on
$H(L;\Lambda_0^{\Z_2})$ is independent, up to the homotopy
equivalence, of the choices we have made of filtered $A_{n,K}$
algebra for any $n,K$. Then we can use the following.
\begin{lem}\label{shortcut}
Let $(C,\{\mathfrak m\})$ and $(C',\{\mathfrak m'\})$ be filtered
$A_{\infty}$ algebras over $\Lambda_{0,{\rm nov}}^R$ for a finite field
$R$. We assume that they are homotopy equivalent as filtered
$A_{n,K}$ algebras for any $n,K$. We assume that $C$ and $C'$ are
finitely generated free $\Lambda_{0,{\rm nov}}^R$ modules. Then they are
homotopy equivalent as filtered $A_{\infty}$ algebras.
\end{lem}
\begin{proof}
Let $V(n,K)$ be the set of all
filtered  $A_{n,K}$ homotopy equivalence from $(C,\{\mathfrak m\})$ to $(C',\{\mathfrak m'\})$.
Using the fact that $R$ is a finite field we can easily see that this is a finite set.
By assumption it is nonempty. Therefore the projective limit
$$
\underset{\longleftarrow}\lim V(n,K)
$$
with respect to the order $<$ on $(n,K)$ is nonempty.
It is easy to see that this implies that $(C,\{\mathfrak m\})$ is homotopy equivalent to $(C',\{\mathfrak m'\})$
as filtered $A_{\infty}$ algebra.
\end{proof}
\par
We remark that Lemma \ref{shortcut} is related to Lemma 7.2.177 \cite{fooo-book2}.
In other words, we are taking the short cut given in Subsection 7.2.11 \cite{fooo-book2}
using the fact that our ground field $R$ is $\Z_2$.
\par
We can prove that our filtered $A_{\infty}$ algebra has a unit in
the same way as in Section 7.3 \cite{fooo-book2}.
\begin{rem}\label{nogeneration}
In Section 7.2  \cite{fooo-book2}, we introduced the notion of
generation to the smooth singular chains and organized the induction
in the way depending to the generation of $P_i$. We did so in order
to work on the {\it countably generated} subcomplex of smooth
singular chains. (Actually by slightly modifying the argument there
we can work on finitely generated chain complex.)
\par
Actually we can use the whole smooth singular chains and do not need
to introduce generations. (Since our global section $\mathfrak s'_{\beta,P_1, \dots, P_k}$ may depend on $P_i$'s, we can apply
Baire's category theorem, which we need to apply general position
argument, even in case we have uncountably many singular chains.)
That is the way taken in this section.
\par
This may slightly simplify the argument of Section 7.2
\cite{fooo-book2}. On the other hand, some people (including some of
the authors of the present paper) may feel happier to use only
countably many (or finitely many) singular chains, since to make a
choice of perturbations, uncountably many times (for each of the
singular chains $P_1, \dots, P_k$) is rather a wild use of the axiom of
choice.
\end{rem}
\par
The proof of Theorem \ref{theoremF} is similar and parallel to
\cite{fooo-book1}, \cite{fooo-book2} and so we omit it.

\section{The case of $\Z$ coefficients}\label{Zcoefficient}
So far we study the case of $\Z_2$ coefficients.
The case of $\Z$ coefficients goes in mostly a similar way.
In this section we discuss it. We focus on the points which are new in $\Z$ coefficients and
avoid repeating the same material.
\par
We first discuss the construction of a filtered $A_{\infty}$ structure
on the singular chain complex $S(L;\Lambda_{0,{\rm nov}}^\Z)$ of a relatively spin Lagrangian submanifold $L$
of a spherically positive symplectic manifold $(M,\omega)$.
The construction is mostly the same as in Section \ref{AinftyconstrZ2}.
Using the relatively spin structure we can define an
orientation of the moduli space and hence the
simplices in Propositions \ref{mkconstruct} and \ref{fkconstruct} are oriented.
Moreover, the orientation is compatible
with the description of the boundary.
Thus we can define the operator $\frak m_{k,\beta}$ by using the virtual
fundamental chain
(\ref{virtualfundamental}) regarded as a singular chain with $\Z$ coefficients.
We need sign that is the same as in \cite{fooo-book2} Definition 8.4.1.
The formula (\ref{Ankformula})
becomes the following

\begin{equation}\label{Ankformulasign}
\aligned
&\frak m_{1,0}\frak m_{k,\beta}(P_1,\dots ,P_{k})
+\sum_i
(-1)^{*} \frak m_{k,\beta}(P_1,\dots,
\frak m_{1,0}
(P_i),\dots,P_k) \\
& = -
\sum_{  {  {\beta_1+\beta_2=\beta,
~k_1+k_2=k+1 ;}
            \atop  { \beta_{1} \ne 0 \text { or}
~k_{1}\ne 1,  } }
           \atop { \beta_{2} \ne 0\text { or}
~k_{2}\ne 1 }  }
\sum_i
(-1)^{*}
\frak m_{k_1,\beta_1}(P_1,\dots,
\frak m_{k_2,\beta_2}
(P_i,\dots,
P_{i+k_2-1}),\dots,P_k),
\endaligned
\end{equation}
where $* = \sum_{j=1}^{i-1} \deg' P_j$, ($\deg'P_j =\deg P_j + 1$) and
$\frak m_{1,0}=(-1)^{n}\partial$, ($n=\dim L$).
See (3.5.15) and Definition 3.5.6 of \cite{fooo-book1}.
The sign is checked in the same way as in \cite{fooo-book2} Section 8.5. Furthermore
in the same way, (\ref{formulaAinftymap}) holds with an appropriate sign.
(See \cite{fooo-book2} Subsection 8.9.1.)
\par
We thus obtain the following.
\begin{enumerate}
\item We have a pair of sequences $(n(I),K(I)) \to (\infty,\infty)$ as $I \to \infty$ and a
filtered $A_{n(I),K(I)}$ structure
$\{\frak m_{k,\beta}^{I,J}\}$ on $S(L;\Lambda_{0,{\rm nov}}^\Z)$
for each $I$.
(Here we take the whole smooth singular chain complex and do not introduce
generation. See Remark \ref{nogeneration}\footnote{We may
use generation and discuss in the same way as in \cite{fooo-book2} Section 7 as well.}.)
\item
For $I<I'$, we can prove that $(S(L;\Lambda_{0,{\rm nov}}^\Z),\{\frak m_{k,\beta}^{I,J}\})$
is homotopy equivalent to $(S(L;\Lambda_{0,{\rm nov}}^\Z),\{\frak m_{k,\beta}^{I',J}\})$
as filtered $A_{n(I),K(I)}$ structure.
\item
We use an appropriate obstruction theory and Item (2) to enhance the
filtered $A_{n(I),K(I)}$ structure $(S(L;\Lambda_{0,{\rm nov}}^\Z),\{\frak m_{k,\beta}^{I,J}\})$
to a filtered $A_{\infty}$ structure.
\item
For each $I$ the filtered $A_{\infty}$ structure obtained in  Item (3)
is independent of the choice of almost complex structure $J$ (in the given
connected component of $\mathcal
J_{(M,\omega)}^{c_1 > 0}$) and of other choices involved up to
filtered $A_{n(I),K(I)}$ homotopy equivalence.
\end{enumerate}
Now there are some difference between the case of $\Z$ coefficients and the case of Section \ref{AinftyconstrZ2}.
Namely Lemma \ref{shortcut} does not hold over $\Z$.
See \cite{fooo-book2} Remark 7.2.181.
Instead we proceed as follows.
\begin{enumerate}
\setcounter{enumi}{4}
\item Let $(S(L;\Lambda_{0,{\rm nov}}^\Z),\{\frak m_{k,\beta}^{I,J}\})$
and
$(S(L;\Lambda_{0,{\rm nov}}^\Z),\{\frak m_{k,\beta}^{\prime,I,J'}\})$
be filtered $A_{n(I),K(I)}$ structures obtained by different choices of
$J$ etc. Then there exists a pair of sequences
$(n'(I),K'(I)) \to (\infty,\infty)$ such that
$(S(L;\Lambda_{0,{\rm nov}}^\Z),\{\frak m_{k,\beta}^{I,J}\})$
is homotopy equivalent to
$(S(L;\Lambda_{0,{\rm nov}}^\Z),\{\frak m_{k,\beta}^{\prime,I,J'}\})$
as filtered $A_{n'(I),K'(I)}$ algebras.
\item Moreover for $I<I'$ the following diagram commutes
as a diagram of filtered $A_{n'(I),K'(I)}$ homomorphisms.
\begin{equation}
\CD
(S(L;\Lambda_{0,{\rm nov}}^\Z),\{\frak m_{k,\beta}^{I,J}\})
@>>>  (S(L;\Lambda_{0,{\rm nov}}^\Z),\{\frak m_{k,\beta}^{I',J}\})\\
@VVV   @VVV \\
(S(L;\Lambda_{0,{\rm nov}}^\Z),\{\frak m_{k,\beta}^{\prime,I,J'}\})  @>>>
(S(L;\Lambda_{0,{\rm nov}}^\Z),\{\frak m_{k,\beta}^{\prime,I',J'}\})
\endCD
\end{equation}
Here horizontal arrows are ones of Item (2) and the vertical arrows are
ones of Item (5).
\par
We can prove Items (5) and (6) in a way similar to \cite{fooo-book2} Subsection 7.2.10 by modifying in the following manner:
We replace multisections used in \cite{fooo-book2} Subsection 7.2.10
by a single valued normally conical section by using
our assumption (spherically positivity) and Theorem \ref{reltechnicalresult}.
\item
Then we can use \cite{fooo-book2} Lemma 7.2.129 inductively
to show that the filtered $A_{\infty}$ structure
in Item (3) is independent of $J$ (in the given connected component of
$\mathcal
J_{(M,\omega)}^{c_1 > 0}$) and other choices up to filtered $A_{\infty}$
homotopy equivalence.
\end{enumerate}
\begin{rem}
We did not claim that homotopy equivalence obtained in Item (7) above
is independent of the choices up to filtered $A_{\infty}$
homotopy.
This independence was proved in \cite{fooo-book2} Subsections 7.2.12-7.2.13 over $\Q$.
We can actually prove it in the same way over $\Z$.
Namely we define the notion of homotopy of homotopies as in
\cite{fooo-book2} Subsection 7.2.12 and proceed in the same way as in \cite{fooo-book2} Subsections 7.2.12-7.2.13.
The rather cumbersome homological algebra in \cite{fooo-book2} Subsections 7.2.12-7.2.13
is carefully designed so that it works over arbitrary coefficients.
\par
In order to prove the independence of homotopy equivalence
up to  filtered $A_{\infty}$
homotopy
over a finite field, we can again use the short cut used in
\cite{fooo-book2} Subsection 7.2.11.
\end{rem}
\par
We next mention the reason why we assumed that $H(L;\Z)$ is torsion free in
Item (4) of Theorem \ref{theoremA}.
In the last section we first construct a filtered $A_{\infty}$ structure on
the singular chain complex of $L$ and reduce
it to the singular cohomology. For this purpose we use the following lemma.
\par
Let $(C,\{\frak m_k\}_{k=0}^{\infty})$ be a filtered $A_{\infty}$
algebra over $\Lambda_{0,{\rm nov}}^{R}$.
We put
$$
C = \overline C \otimes_R \Lambda_{0,{\rm nov}}^{R},
$$
where $\overline C$ is a graded free $R$ module.
The $R$ reduction $\overline{\frak m}_1 :  \overline C \to  \overline C$  of $\frak m_1$ satisfies
$\overline{\frak m}_1\circ\overline{\frak m}_1 = 0$ and
hence $(\overline C,\overline{\frak m}_1)$ is a chain complex.
\begin{lem}\label{canolem}
Suppose we have a direct sum decomposition
$
\overline C = \overline D \oplus \overline D'
$
of $R$ module such that $(\overline D,\overline{\frak m}_1)$ is a subcomplex of
$(\overline C,\overline{\frak m}_1)$ and the inclusion
$
(\overline D,\overline{\frak m}_1) \to (\overline C,\overline{\frak m}_1)
$
is a chain homotopy equivalence.
Then a filtered $A_{\infty}$ structure is induced on $D=\overline D \otimes_R \Lambda_{0,{\rm nov}}^{R}$
such that it is homotopy equivalent to $C$.
\end{lem}
The proof of Lemma \ref{canolem} is the same as the proof of
\cite{fooo-book1}
Theorem 5.4.2 and so is omitted.
\begin{exm}\label{Examplefree}
(1) When $R$ is a field,
we always have a splitting
\begin{equation}\label{homologysplitting}
\overline C = H(\overline C,\overline{\frak m}_1) \oplus \overline C'
\end{equation}
such that $H(\overline C,\overline{\frak m}_1)$ (with zero boundary operator)
is chain homotopy equivalent to $(\overline C,\overline{\frak m}_1)$.
Therefore by Lemma \ref{canolem}
we obtain a filtered $A_{\infty}$ structure on $H(\overline C,\overline{\frak m}_1)
\otimes_R \Lambda_{0,{\rm nov}}^{R}$.
We used this fact in the last section.
\par
(2) In the case $R= \Z$ (or any Dedekind domain), the splitting
(\ref{homologysplitting}) exists if
$H(\overline C,\overline{\frak m}_1)$ is torsion free.
(Note that any finitely generated torsion free module over a
Dedekind domain is projective.)
Therefore under the assumption that
$H(L;R)$ is torsion free, the filtered $A_{\infty}$ structure
on the singular chain complex induces one on its
singular cohomology.
\end{exm}
\begin{rem}\label{freecomplex}
In general, we can obtain a filtered $A_{\infty}$ structure
on a finite dimensional complex over $\Lambda_{0,{\rm nov}}^{\Z}$
using some additional information on $L$
as follows:
\begin{enumerate}
\item
Let us fix a Morse function $f$ on $L$. We obtain a Morse complex
$C(L;f)$. It is a chain complex over $\Z$ and is torsion free as $\Z$ module.
Using Lemma \ref{canolem} and a filtered $A_{\infty}$ structure on the singular chain
complex we obtain a filtered $A_{\infty}$ structure on
$C(L;f) \otimes \Lambda_{0,{\rm nov}}^{\Z}$.
See \cite{fooo-morse} Theorem 5.1.
\item
Let us take either a simplicial decomposition or a CW decomposition of $L$.
Then we have a finite dimensional chain complex $C(L;\Z)$ which is free over $\Z$ and is
chain homotopy equivalent to the singular chain complex of $L$.
So we obtain a filtered $A_{\infty}$ structure on
$C(L; \Lambda_{0,{\rm nov}}^{\Z})$.
\end{enumerate}
\end{rem}
To prove the last claim in Item (4) of Theorem \ref{theoremA} it suffices to prove the following.
\begin{lem}\label{tensorQheq}
Let $(S(L;\Lambda_{0,{\rm nov}}^{\Z}),\{\frak m_k^{\Z}\})$
be the filtered $A_{\infty}$ algebra in Theorem \ref{theoremA}
and $(S(L;\Lambda_{0,{\rm nov}}^{\Q}),\{\frak m_k^{\Q}\})$
the filtered $A_{\infty}$ algebra in
\cite{fooo-book1} Theorem 3.5.11.
Then
 $(S(L;\Lambda_{0,{\rm nov}}^{\Z}),\{\frak m_k^{\Z}\}) \otimes \Q$
 is homotopy equivalent to $(S(L;\Lambda_{0,{\rm nov}}^{\Q}),\{\frak m_k^{\Q}\})$.
\end{lem}
\begin{proof}
We consider the moduli space
\begin{equation}\label{topmoduli}
\MM_{k+1}^{\text{main}}
(\{J\}_{\rho}:\beta;\operatorname{top}(\rho);
P_1,\dots,P_k)
\end{equation}
defined in Section \ref{AinftyconstrZ2}.
Here we fix $J \in \mathcal J_{(M,\omega)}^{c_1 > 0}$ and use
the constant family $\rho \mapsto J$ of almost complex structures.
We use appropriate perturbation of this moduli space to define
the required homotopy equivalence.
\par
Note that we took a system of single valued normally conical sections
$s^{\Z}$ of the
moduli space $\mathcal M_{k+1}^{\text{main}}(\beta;P_1,\dots,P_k)$ to define $\frak m_k^{\Z}$.
We also took (in \cite{fooo-book1}) a transversal multisection $s^{\Q}$ of the same moduli space
to define $\frak m_k^{\Q}$.
We will find a perturbation of (\ref{topmoduli}) which interpolates these two perturbations.
\begin{defn}
A strongly piecewise smooth multisection $s$ on a space with Kuranishi structure $X$ is
said to be a {\it weakly transversal and weakly normally conical} multisection, if
the following condition is satisfied.
We decompose
$$
s = \bigoplus_i s^{\Gamma_i} \oplus s^{\Gamma} \oplus s^{\perp}
$$
according to the decomposition (\ref{35.41}) (\ref{35.41-2}).
\begin{enumerate}
\item
Each of $s^{\Gamma}$ is of general position to $0$ on $X^{\cong}(\Gamma)$.
\item
Instead of (\ref{35.51}) the following equality is satisfied.
\begin{equation}\label{conicprime}
\aligned
s(p) = &s^{\Gamma}(\pi_{\Gamma}(p)) \\
&+
\sum_{\Gamma'\subsetneq \Gamma}
\left(\left(1 - \exp\left(\frac{1}{d^2} - \frac{1}{\rho_{\Gamma}(p)}\right)\right)s^{\Gamma'}(\pi_{\Gamma}(p))
\right.
\\
&\qquad\qquad\left.
+ \exp\left(\frac{1}{d^2} - \frac{1}{\rho_{\Gamma}(p)}\right) s^{\Gamma'}(r_{\Gamma}(d)(p))\right).
\endaligned
\end{equation}
(Note for each branch of $s^{\Gamma'}(\pi_{\Gamma}(p))$ there is a branch of $ s^{\Gamma'}(r_{\Gamma}(d)(p))$.
We take the sum branch-wise in (\ref{conicprime}).)
\end{enumerate}
\end{defn}
We note that if $s$ is a single valued section then $s^{\Gamma'}(\pi_{\Gamma}(p))$
is automatically $0$.
(Note $s^{\Gamma'}$ is a
component of $s^{\perp}$.) So (\ref{conicprime}) coincides with (\ref{35.51}) in that case.
When $s$ is a multisection $s^{\Gamma'}$ ($\Gamma' \subsetneq \Gamma$) may not be $0$ on
$X^{\cong}(\Gamma)$.
\par
We also note that we may choose our multisection $s^{\Q}$
so that it is weakly transversal and weakly normally conical.
In fact, we may take $(s^\Q)^{\Gamma'}$ to be locally constant
in the normal direction. Namely
$$
(s^\Q)^{\Gamma'}(\pi_{\Gamma}(p))
= (s^\Q)^{\Gamma'}(r_{\Gamma}(d)(p)).
$$
Now in the same way as in Section \ref{AinftyconstrZ2} we can
construct a system of weakly transversal and weakly normally conical multisections $s$
on (\ref{topmoduli}) such that
it is compatible with $s^{\Z}$ on the component of time allocation $0$
and with $s^{\Q}$ on the component of time allocation $1$.
\par
We may choose $s$ so that $s^{-1}(0)$ has a triangulation,  in the following
way.
\begin{enumerate}
\item
As for the component of the obstruction bundle coming from the
component with time allocation $0$, we assume $s^{\perp} = 0$.
\item
For other components, we assume $s$ is transversal to zero.
\item
In the tubular neighborhood of the stratum of the moduli  space defined by the
equation that time allocation $= 0$, the section $s^{\perp}$
is normally conical.
\end{enumerate}
We explain Item (1) above more precisely.
Let $p = (\Sigma,\vec z,u,\{\rho_a\})$ represent an element of
$\MM_{k+1}^{\text{main}}
(\{J\}_{\rho}:\beta;\operatorname{top}(\rho))$,
where $(\Sigma,\vec z)$ is a bordered marked Riemann surface,
$u : (\Sigma,\partial \Sigma) \to (M,L)$ and $\{\rho_a\}$ is a time allocation.
The fiber $E_p$ of the obstruction bundle is a
finite dimensional subspace of $C^{\infty}(\Sigma;u^*TM\otimes \Lambda^{0,1})$,
the space of the smooth sections of the vector bundle $u^*TM\otimes \Lambda^{0,1}$.
(See \cite{fooo-book2} proof of Proposition 7.1.12.)
The support of elements of $E_p$ is disjoint from the singular points.
And $E_p$ is decomposed to the direct sum
$\bigoplus_a (E_p)_a$, where $\Sigma = \cup_{a}\Sigma_a$ is a decomposition
to the components and elements of $(E_p)_a$ are supported on $\Sigma_a$.
The consistency with $s^{\Z}$ requires us to set the component of $s^{\perp}$
in $(E_p)_a$ with $\rho_a=0$ is zero. This is Item (1) above.
\par
We can construct such a multisection by the inducition of the stratum in the
same way as in Section \ref{section35.3}.
It is easy to see that  $s^{-1}(0)$
satisfies the conclusion of Theorem \ref{maintechnicalresult}.
Therefore, we can use it in the same way as (\ref{defnf}) to define the
required homotopy equivalence.
The proof of Lemma \ref{tensorQheq} is complete.
\end{proof}
The proof of Theorem \ref{theoremA} is now complete.
\qed
\par\smallskip
The proof of Theorem \ref{theoremF} over $\Z$ is again similar and parallel to
\cite{fooo-book1}, \cite{fooo-book2}. So we omit it.

\section{How the results of \cite{fooo-book1}, \cite{fooo-book2}
are generalized to $\Z$ or $\Z_2$ coefficients}
\label{whichandhow}

\subsection{Statements}
\label{state}

Throughout this section we always assume $(M,\omega)$ is a
spherically positive symplectic manifold.
We assume its Lagrangian submanifold $L$ is compact, oriented and
relatively spin when the ground ring $R$ is a Dedekind domain (for example $\Z$) or
is a finite field
of odd characteristic (for example $\Z_p=\Z/p\Z$ ($p\ne 2$)). In case when $R$ is a finite field
of even  characteristic (for example $\Z_2$)
we only assume $L$ to be a
compact Lagrangian submanifold.
\par
The various structures we mention in this section depend on the connected
component of $\JJ_{(M,\omega)}^{c_1 > 0}$.
The various well-definedness or functoriality statement
should be understood in the same way as in Theorem \ref{theoremA}
(2).
\par\medskip
The following is a version of Theorem B \cite{fooo-book1}.
Theorem \ref{theoremB} follows from Theorem \ref{theoremA} by a
purely algebraic argument.
\begin{theorem}\label{theoremB}
We can associate a set $\mathcal M_{\text{\rm
weak}}(L;R)$ and a map
$
\frak{PO} : \mathcal M_{\text{\rm weak}}(L;R) \to \Lambda^{+,R}_{0,{\rm nov}},
$
which have the following properties
\begin{enumerate}
\item
There is a Floer cohomology
$HF((L,b_1),(L,b_0);\Lambda^R_{0,{\rm nov}})$ parameterized by
$$
(b_1,b_0)
\in \mathcal M_{\text{\rm weak}}(L;R) \times_{\frak{PO}}
\mathcal M_{\text{\rm weak}}(L;R).
$$
Here $\mathcal M_{\text{\rm weak}}(L;R) \times_{\frak{PO}}
\mathcal M_{\text{\rm weak}}(L;R)$ is the set of pairs
$(b_1,b_0) \in \mathcal M_{\text{\rm weak}}(L;R) \times
\mathcal M_{\text{\rm weak}}(L;R)$ such that
$\frak{PO}(b_0) = \frak{PO}(b_1)$.
\item
There exists a product structure
$$\aligned
\frak m_2 :
HF((L,b_2),(L,b_1);\Lambda^R_{0,{\rm nov}}) &\otimes HF((L,b_1),(L,b_0);\Lambda^R_{0,{\rm nov}}) \\
&\to HF((L,b_2),(L,b_0);\Lambda^R_{0,{\rm nov}})
\endaligned$$
if $(b_1,b_0),  (b_2,b_1)
\in \mathcal M_{\text{\rm weak}}(L;R) \times_{\frak{PO}} \mathcal M_{\text{\rm
weak}}(L;R)$. The product $\frak m_2$ is associative. In particular, $HF((L,b),(L,b);\Lambda^R_{0,{\rm nov}})$ has the ring structure for any $b \in \mathcal M_{\text{\rm weak}}(L;R)$.
\item
If $\psi : (M,L) \to (M',L')$ is a
symplectic diffeomorphism, then
it induces a bijection
$\psi_* : \mathcal M_{\text{\rm weak}}(L;R) \to \mathcal M_{\text{\rm weak}}(L';R)$ such that
$
\frak{PO}\circ \psi_*  =  \frak{PO}.
$
The map $\psi_*$ depends only on the isotopy class of symplectic diffeomorphism $\psi$.
\item
In the situation of {\rm (3)}, we have an isomorphism
$$
\psi_*  : HF((L,b_1),(L,b_0);\Lambda^R_{0,{\rm nov}})
\cong HF((L',\psi_*(b_1)),(L',\psi_*(b_0));\Lambda^R_{0,{\rm nov}}),
$$
and $\psi_*$ commutes with $\frak m_2$.
\item
The isomorphism
$\psi_*$ in {\rm (4)} depends only on the isotopy class of symplectic diffeomorphism $\psi : (M,L) \to
(M',L')$. Moreover $(\psi\circ \psi')_* = \psi_* \circ \psi'_*$.
\end{enumerate}
\end{theorem}
\begin{rem}\begin{enumerate}
\item
We do not include bulk deformations in Theorem \ref{theoremB}.
Namely we consider $\mathcal M_{{\rm weak}}(L;R)$ and not
$\mathcal M_{{\rm weak},\text{\rm def}}(L;R)$.
Including bulk deformations is a bit difficult to work out
in our case for coefficients in a Dedekind domain or a finite field.
For example, we used homotopy theory of filtered $L_{\infty}$ algebra
in \cite{fooo-book2} Section 7.4, for the algebraic formulation of
bulk deformations and of the operator $\frak q$.
Homotopy theory of filtered $L_{\infty}$ algebra
is hard to study over torsion coefficients.
However, it seems very likely that we can go around this problem and
define $\mathcal M_{{\rm weak,def}}(L;R)$ and Floer cohomology parametrized
by it over a Dedekind domain or a finite field coefficients.
We postpone it to future research.
\item
The Maurer-Cartan moduli space $\mathcal M_{\text{\rm weak}}(L;R)$
in Theorem \ref{theoremB}
is one over $\Lambda_{0,{\rm nov}}^{+,R}$.
Namely it is the set of the gauge equivalence classes of the
chains $b\in S(L;\Lambda_{0,{\rm nov}}^{+,R})$ satisfying the equation
\begin{equation}\label{MCequation}
\sum_{k=0}^{\infty} \frak m_k(b^k) \equiv 0  \mod \Lambda_{0,{\rm nov}}^{+,R}\text{\bf e}
\end{equation}
where $\text{\bf e}$ is the unit. (See \cite{fooo-book1} Section 4.3.)
In \cite{fooo:toric1,fooo:bulk,fukaya:cyc} we enhanced it to one over $\Lambda_{0,{\rm nov}}^{R}$
coefficients in case
$R=\R$ or $R=\C$. We will discuss this enhancement in the case when $R$ is a Dedekind domain or a finite field
in Subsection \ref{MClambda0}.
\end{enumerate}
\end{rem}
The next result is a version of Theorem C \cite{fooo-book1}.
The proof is the same as that of Theorem C \cite{fooo-book1} by using Theorem \ref{theoremA}.
\begin{theorem}\label{theoremC}
There exists a series of positive integers $m_k < \dim L/2$ and classes
$$
[o^{2m_k}_k(L;\operatorname{weak})] \in
H^{2m_k}(L;R)
$$
$k=1,2,\dots$, such that,
if $[o^{2m_k}_k(L;\operatorname{weak})]$ are all zero, then
$\mathcal M_{\text{\rm weak}}(L;R)$ is nonempty.
The number $2-2m_k$ is a sum of the Maslov indices of a finite collection of
the homotopy classes in $\pi_2(M,L)$ realized by pseudo-holomorphic discs
$($with respect to a given almost complex structure on $M$$)$.
\end{theorem}
Next, we consider analogs of Theorem D and Theorem E \cite{fooo-book1} on the spectral sequence. Because of convergence issue of the spectral sequence over $\Lambda_{0,{\rm nov}}^R$ we consider the following assumption.

\begin{assumption}\label{assumpRL}
We assume one of the following
conditions.
\begin{enumerate}
\item $R$ is a finite field of characteristic $2$.
\item $R$ is a finite field of odd  characteristic and $L$ is relatively spin.
\item $R$ is a Dedekind domain, $L$ is
relatively spin and rational in the sense of
\cite{fooo-book1} Definition 6.2.1.
In this case, Theorem \ref{theoremD} only applies to $b_1,b_0$
of the form $\sum T^{\lambda_i}e^{\mu_i/2}b_i$
such that the subgroup of $\R$
generated by the set $\{\lambda_i\} \cup \{\beta(\omega) \mid \beta \in \pi_2(M,L)\}$
is isomorphic to $\Z$.
\end{enumerate}
\end{assumption}
We denote by $(\Lambda_{0,{\rm nov}}^{R})^{(p)}$
the degree $p$ part of $\Lambda_{0,{\rm nov}}^{R}$,
where we recall
$\deg(aT^{\lambda}e^{\mu})=2\mu$ for
$aT^{\lambda}e^{\mu} \in \Lambda_{{\rm nov}}^{R}$.
\begin{theorem}\label{theoremD}
Under Assumption \ref{assumpRL}
there exists a spectral sequence
for each $(b_1,b_0)
\in \mathcal M_{\text{\rm weak}}(L;R) \times_{\frak{PO}}
\mathcal M_{\text{\rm weak}}(L;R)$
with the following properties:
\begin{enumerate}
\item
$E_2^{p,q} = \bigoplus_kH^k (L; R) \otimes (T^{q\lambda}\Lambda_{0,{\rm nov}}^R/T^{(q+1)\lambda}
\Lambda^R_{0,{\rm nov}})^{(p-k)}$.
Here $\lambda>0.$
\item
There exists a filtration
$\frak F^*HF((L,b_1),(L,b_0);\Lambda^R_{0,{\rm nov}})$
on the Floer cohomology
$HF((L,b_1),(L,b_0);\Lambda^R_{0,{\rm nov}})$
such that
$$
E_{\infty}^{p,q}
\cong \frac{
\frak F^qHF^p((L,b_1),(L,b_0);\Lambda^R_{0,{\rm nov}})}
{\frak F^{q+1}HF^p((L,b_1),(L,b_0);\Lambda^R_{0,{\rm nov}})}.
$$
\item
Consider the subgroup
$K_r \subset E_r =\bigoplus _{p,q}E_{r}^{p,q}$ defined by
$$
\aligned
K_2^{p,q} & = \bigoplus_k PD(\text {\rm Ker} (H_{n-k} (L;R)
\rightarrow H_{n-k}
(M;R)))\\
&\qquad\qquad\qquad\quad \otimes (T^{q\lambda}\Lambda_{0,{\rm nov}}^R/T^{(q+1)\lambda}\Lambda_{0,{\rm nov}}^R)^{(p-k)},\\
K_{r+1} & = {K_r \cap \text {\rm Ker~} \delta_r \over K_r \cap
\text {\rm Im~} \delta_r}
\subset E_{r+1}. \endaligned
$$
Then we have $\text {\rm Im~} \delta_r \subseteq K_r$ for every $r$,
under the additional assumption that $b_0=b_1$.
\par
In particular, if
the homomorphism
$i_{\ast} : H_{\ast}(L;R)
\to H_{\ast}(M;R)$
induced by the inclusion is injective, the spectral sequence collapses at the
$E_2$ level.
\item
The spectral sequence is
compatible with the ring structure in Theorem \ref{theoremB} {\rm (2)}.
In other words, we have the following.
Each of $E_r$ has a ring structure $\frak m_2$ which
satisfies :
$$
\delta_r(\frak m_2(x,y)) = - \frak m_2(\delta_r(x),y) + (-1)^{\deg x} \frak m_2(x,\delta_r(y)).
$$
The filtration $\frak F$ is compatible with the ring structure in
Theorem \ref{theoremB} {\rm (2)}.
The isomorphisms in Theorem \ref{theoremD} {\rm (1)},{\rm (2)} are ring isomorphisms.
\end{enumerate}
\end{theorem}
By using Theorem \ref{theoremA},
the proof of Theorem \ref{theoremD} is the same as that of Theorem D \cite{fooo-book1}, except the following two points.
Assumption \ref{assumpRL} is used in a way described in (1) below.
\begin{enumerate}
\item
During the construction of the spectral sequence, we used the algebraic material
given in \cite{fooo-book1} Subsections 6.3.1 and 6.3.2. We assumed that the ground ring $R$ is a
field there.
(The authors do not know how to prove the convergence of the spectral sequence
in the case $R=\Z$.)
In the case $L, b_0, b_1$ satisfy Assumption \ref{assumpRL} (2), we can use the argument of \cite{fooo-book1} Section 6.2 and
can construct the spectral sequence and prove its convergence in the case when $R$ is a Dedekind domain also.
\item
In \cite{fooo-book1} we used the operator $\frak p$ and the cyclic cohomology of $L$ to prove
a statement (3) of degeneration of the spectral sequence over $\Q$ (\cite{fooo-book1} Subsection 6.4.2).
It seems hard to study cyclic cohomology over torsion coefficients.
However in Subsection \ref{Hochschild homology} we will use the Hochschild homology instead, and introduce an analogous operator $\frak p'$ to prove (3) over $R$.
\end{enumerate}

The following non-vanishing theorem is an analog of Theorem E \cite{fooo-book1} and
the proof is the same except that of
$PD[pt] \notin \operatorname{Im}(\delta_r)$
in (1), where we used the operator $\frak p$
for the case over $\Q$ in \cite{fooo-book1} Subsection 6.4.3.
We will also use the analogous operator
$\frak p'$ over $R$.
See Subsection \ref{Hochschild homology} for this point.
\begin{theorem}\label{theoremE}
Let us consider the situation of Theorem
\ref{theoremD}.
We assume
$b_0=b_1$. Then, there exists a
cohomology class $PD[L]' \in
H^n(L;\Lambda_{0,{\rm nov}}^R)$ with $PD[L]' \equiv PD[L] \mod \Lambda^{+,R}_{0,{\rm nov}}$ which has the following
properties.
\begin{enumerate}
\item
For each $r$ we have $\delta_r(PD[L]') = 0$ and
$PD[pt] \notin \operatorname{Im}(\delta_r)$, where $\delta_r$
is the differential of the spectral sequence in Theorem \ref{theoremD}.
\item
If the Maslov index of all the pseudo-holomorphic
discs bounding $L$ are non-positive, then
$\delta_r(PD[pt]) = 0$ and
$PD[L]' \notin \operatorname{Im}(\delta_r)$.
\end{enumerate}
\par
The same conclusion holds for $\Lambda_{{\rm nov}}^R$ coefficients.
\end{theorem}
The next theorem is an analog of Theorem G \cite{fooo-book1} and
follows from Theorem \ref{theoremF}.
\begin{theorem}\label{theoremG}
Let $(L_1,L_0)$ be a pair of Lagrangian submanifolds of $M$
of clean intersection.
Then, for each $(b_1,b_0) \in \mathcal M_{\text{\rm weak}}(L_1;R)
\times_{\frak{PO}}
\mathcal M_{\text{\rm weak}}(L_0;R)$,  we can associate a Floer cohomology
$HF((L_1,b_1),(L_0,b_0);\Lambda_{0,{\rm nov}}^R)$
with the following properties. We put
$$
HF((L_1,b_1),(L_0,b_0);\Lambda_{0,{\rm nov}}^R) \otimes_{\Lambda_{0,{\rm nov}}^R}\Lambda_{{\rm nov}}^R
= HF((L_1,b_1),(L_0,b_0);\Lambda_{{\rm nov}}^R).
$$
\begin{enumerate}
\item If $L_0 = L_1 = L$, then
$HF((L_1,b_1),(L_0,b_0);\Lambda^R_{0,{\rm nov}})$ coincides with
the one in Theorem \ref{theoremB}.
\item If $R$ is a field, then we have
$$
\text{\rm rank}_{\Lambda^R_{{\rm nov}}} HF((L_1,b_1),(L_0,b_0);\Lambda^R_{{\rm nov}})
\le \sum_{h,k} \text{\rm rank}_{R}H^{k}(R_h;\Theta_{R_h}^-),
$$
where each $R_h$ is a connected component of $L_0 \cap L_1$ and
$\Theta_{R_h}^-$ is a local system on it.
\par
\item If $\psi : M \to M'$ is a
symplectic diffeomorphism with $\psi(L_i) = L'_i$, $(i=0,1)$,
then we have a canonical isomorphism
$$
\psi_* : HF((L_1,b_1),(L_0,b_0);\Lambda^R_{0,{\rm nov}})
\cong HF((L'_1,\psi_*b_1),(L'_0,\psi_*b_0);\Lambda^R_{0,{\rm nov}})
$$
where $\psi_* : \mathcal M_{\text{\rm weak}}(L_i;R) \to \mathcal M_{\text{\rm weak}}(L'_i;R)$
is as in Theorem \ref{theoremB}. The isomorphism $\psi_*$ depends only
on the isotopy class of symplectic diffeomorphism $\psi$ with $\psi(L_i) = L'_i$.
We also have $(\psi\circ \psi')_* = \psi_*\circ \psi'_*$.
\item If $\psi^s_i : M \to M$ $($$i=0,1$, $s \in
[0,1])$ are Hamiltonian isotopies with $\psi_i^0 = identity$ and
$\psi_i^1(L_i) = L'_i$, then it induces an isomorphism
$$
(\psi^s_1,\psi^s_0)_* : HF((L_1,b_1),(L_0,b_0);\Lambda^R_{{\rm nov}})
\cong HF((L'_1,\psi^1_{1*}b_1),(L'_0,\psi^1_{0*}b_0);\Lambda^R_{{\rm nov}}).
$$
The isomorphism $(\psi^s_1,\psi^s_0)_*$ depends only the isotopy class of the
Hamiltonian isotopies $\psi^s_i : M \to M$ $(i=0,1$, $s \in [0,1])$
with $\psi_i^0 = identity$ and $\psi_i^1(L_i) = L'_i$.
The isomorphism $(\psi^s_1,\psi^s_0)_*$ is functorial with respect to the composition of
the Hamiltonian isotopies.
\item
$HF((L_1,b_1),(L_0,b_0);\Lambda^R_{0,{\rm nov}})$ is a
bimodule over the ring pair
$$
\left(HF((L_1,b_1),(L_1,b_1);\Lambda^R_{0,{\rm nov}}),
HF((L_0,b_0),(L_0,b_0);\Lambda^R_{0,{\rm nov}})\right).
$$
The isomorphisms {\rm (3), (4)} are bimodule isomorphisms.
\end{enumerate}
\end{theorem}
We state Theorem \ref{theoremG} only in the case when $(L_1, L_0)$ has clean intersection.
In the case when $R$ is a finite field, we can use a version of Theorem 6.1.25 \cite{fooo-book1}
in the same way as in \cite{fooo-book1} Subsection 6.5.4 to remove this assumption
over $\Lambda_{0,{\rm nov}}^R$ coefficients.
\par\medskip
\cite{fooo-book1} Theorem I can be generalized as follows.
\begin{theorem}\label{theoremI}
Under Assumption \ref{assumpRL} (1) or (2),
we assume that
$\mathcal M_{\text{\rm weak}}(L;R)$ is non-empty.
Denote $A = \sum_* \operatorname{rank }H_*(L;R)$, $B=
\sum_* \operatorname{rank}\operatorname{Ker}(H_*(L;R)
\to H_*(M;R))$. Then we have
$$
\#(L\cap \phi(L)) \geq A - 2B
$$
for any Hamiltonian diffeomorphism $\phi:M\to M$ such that $L
\pitchfork \phi(L)$.
\end{theorem}
The proof is the same as that of \cite{fooo-book1} Theorem I
using Theorems \ref{theoremD}
and \ref{theoremG}.
\par\smallskip
Theorem J \cite{fooo-book1} is generalized as follows.
We consider the case when Assumption \ref{assumpRL} (1) or (2) is satisfied.
Let $\phi$ be a Hamiltonian diffeomorphism.
We assume $\phi(L_1)$ is transverse to $L_0$ and
$b_i \in \mathcal M_{\text{\rm weak}}(L_i;R)$
with $\frak{PO}(b_0) = \frak{PO}(b_1)$.
Since the algebraic argument in Subsections 6.3.1 and 6.3.2 \cite{fooo-book1} works over an arbitrary field $R$, we can prove
\begin{equation}\label{eq:HF0nov}
 HF((L_1,b_1),(L_0,b_0);\Lambda^R_{0,{\rm nov}})
 \cong
 (\Lambda_{0,\text{\rm nov}}^R)^{\oplus a} \oplus
\bigoplus_{i=1}^b (\Lambda^R_{0,\text{\rm nov}}/T^{\lambda_i}\Lambda^R_{0,\text{\rm nov}})
\end{equation}
in the same way as in \cite{fooo-book1} Theorem 6.1.18 (and also Theorem 6.1.20).
We call $a$ the {\it Betti} number and $\lambda_i$ {\it torsion exponents}.
\begin{theorem}\label{thm:correctedthmJ}
Under Assumption \ref{assumpRL} (1) or (2),
we denote
$$
b(\Vert\phi\Vert) = \#\{ i \mid\lambda_i \ge \Vert\phi\Vert\},
$$
where $\lambda_i$ are the torsion exponents as in \eqref{eq:HF0nov} and
$\Vert\phi\Vert$ is the Hofer norm.
Then we have
\begin{equation}
\label{eq:mainformula} \# (\phi(L_1)\cap L_0) \ge a + 2b(\Vert\phi\Vert).
\end{equation}
\end{theorem}
The original proof of Theorem J \cite{fooo-book1}
contains an error related to an energy estimate.
See \cite{bidisk}.
We now have corrected it in \cite{bidisk}.
Then the proof of Theorem \ref{thm:correctedthmJ}
is the same as that of
Theorem J given in \cite{bidisk}.
Theorem 6.1.25 \cite{fooo-book1}
can be generalized in the same way.

\subsection{Proof of Theorems \ref{theoremKs}{} and \ref{theoremL}}
\label{ProofKSL}

\begin{proof}[Proof of Theorem \ref{theoremKs}]
We prove Theorem \ref{theoremKs} by contradiction.
Suppose that the Maslov class vanishes.
Theorem \ref{theoremC} and the assumption
imply that all the obstruction classes
$[o_k^{2m_k}(L;\text{\rm weak})]$ are in $H^2(L;\Z_2)$.
Therefore by assumption that $H^2(L;\Z_2)=0$,  we have
$b$ such that the Floer cohomology $H((L,b),(L,b);\Lambda_{0,{\rm nov}}^{\Z_2})$
is defined.
\par
By Theorem \ref{theoremE} (1) we have $\delta_r(PD[L]') = 0$.
Moreover by Theorem \ref{theoremE} (2)
$PD[L]' \notin \text{\rm Im}\delta_r$.
Thus we have
$
H((L,b),(L,b);\Lambda_{{\rm nov}}^{\Z_2}) \ne 0.
$
Then Theorem \ref{theoremG} (1),(4) imply that
$$
H((L,b),(\psi(L),\psi_*b);\Lambda_{{\rm nov}}^{\Z_2}) \ne 0
$$
for any Hamiltonian diffeomorphism $\psi$.
Theorem \ref{theoremG} (2) now implies
$$
L \cap \psi(L) \ne \emptyset,
$$
which leads to a contradiction.
\end{proof}
Theorem \ref{theoremL} follows easily from Theorem \ref{theoremKs}.

\subsection{Maurer-Cartan moduli space over $\Lambda_{0,{\rm nov}}^R$}
\label{MClambda0}

We considered the filtered $A_{\infty}$ algebra
$(S(L;\Lambda_{0,{\rm nov}}^R),\{\frak m_k\})$
in Sections \ref{AinftyconstrZ2} and \ref{Zcoefficient}.
For
$
b \in S^1(L;\Lambda_{0,{\rm nov}}^{+,R})
$
we define a deformation of filtered $A_{\infty}$ structure by
\begin{equation}\label{defAinfty}
\frak m_k^b(x_1,\dots,x_k)
=
\sum_{\ell_0=0}^{\infty}\cdots
\sum_{\ell_k=0}^{\infty}
\frak m_{k+\sum\ell_i} (b^{\ell_0},x_1,b^{\ell_1},\dots,x_k,b^{\ell_k}).
\end{equation}
Using the fact that $b\equiv 0 \mod \Lambda_{0,{\rm nov}}^{+,R}$
we can prove that the right hand side of (\ref{defAinfty}) converges in $T$ adic topology.
The Maurer-Cartan equation (\ref{MCequation})
becomes $\frak m_0^b \equiv 0 \mod \Lambda_{0,{\rm nov}}^{+,R}\text{\bf e}$.
\par
In the case $R=\R$ (resp. $\C$), this story is generalized in
\cite{fooo:toric1,fooo:bulk,fukaya:cyc} to the case when
$
b \in S^1(L;\Lambda_{0,{\rm nov}}^{\R})
$ (resp. $b \in S^1(L;\Lambda_{0,{\rm nov}}^{\C})$)
as follows. We recall the case $R=\R$ only, because the case $R=\C$ is similar.
We put
$$
b = \overline b_1 + b_+
$$
where
$$
\overline b_1 \in S^1(L;\R),
\quad
b_+ \in S^1(L;\Lambda_{0,{\rm nov}}^{(0) +,\R}) \oplus
\bigoplus_{k\ge 1} S^{1+2k}(L;\Lambda_{0,{\rm nov}}^{(0) \R})e^{-k}.
$$
Here $\Lambda_{0,{\rm nov}}^{(0) R}$ is the degree $0$ part of
$\Lambda_{0,{\rm nov}}^{R}$, namely, the part which does not contain the indeterminate $e$.
We assume $\partial \overline b_1 = 0$, where
$\partial$ is the usual boundary operator.
(This will follow from the Maurer-Cartan equation.)
We decompose $\frak m_k$ as
$$
\frak m_k = \sum_{\beta \in \Pi_2(M,L)}
T^{\omega(\beta)}e^{\mu(\beta)/2} \frak m_{k,\beta}.
$$
In the situation of \cite{fooo:toric1,fooo:bulk,fukaya:cyc} we take the moduli space
$\mathcal M_{k+1}(\beta)$ and its perturbation so that it is
compatible with the forgetful map of the $1$st,\dots,$k$-th marked points
$\mathcal M_{k+1}(\beta) \to \mathcal M_{1}(\beta)$.
(See \cite{fukaya:cyc} Section 5 for the precise
description of the compatibility.)
We used this fact to show the following formula
for $\overline{b}_1 \in S^1(L;\R)$:
\begin{equation}\label{foget1-k}
\sum_{\ell_0+\dots+\ell_k=\ell}
\frak m_{k+\ell,\beta} (\overline b_1^{\ell_0},x_1,\dots,x_k,\overline b_1^{\ell_k})
=
\frac{(\overline b_1(\partial \beta))^{\ell}}{\ell !}
\frak m_{k,\beta}(x_1,\dots,x_k).
\end{equation}
(See \cite{fooo:bulk} Lemmas 7.2, 9.2 and \cite{fukaya:cyc} Lemma 13.1.)
Using (\ref{foget1-k}) we define
for $b\in S^1(L;\Lambda_{0,{\rm nov}}^{\R})$
\begin{equation}\label{defAinfty2}
\frak m_k^b(x_1,\dots,x_k)
=
\sum_{\beta}
\sum_{\ell_0,\dots ,\ell_k=0}^{\infty}
T^{\omega(\beta)}e^{\mu(\beta)/2}\exp(\overline b_1(\partial\beta))
\frak m_{k+\sum\ell_i,\beta}
(b_+^{\ell_0},x_1,\dots,x_k,b_+^{\ell_k})
\end{equation}
by modifying the definition
\eqref{defAinfty} for the case
$b \in S^1(L;\Lambda_{0,{\rm nov}}^{+,\R})$.
See (11.4) in \cite{fooo:bulk} for toric cases.
Using Gromov's compactness we can show that the right hand side
converges in $T$ adic topology. (See \cite{fooo:bulk} Section 9 and \cite{fukaya:cyc} Lemma 13.3.)
This definition is closely related to the idea of \cite{cho07}
to use nonunitary flat bundles on $L$. Namely, in this case we take
a flat line bundle with monodromy
$$
\gamma \mapsto \exp(\overline b_1(\gamma)).
$$
\par\medskip
In order to generalize this story to the case when $R \ne \R, \C$ but $R$ is a Dedekind domain  or a finite field, there are two points we need to take care of.
\par
Firstly the compatibility of the
perturbation to the forgetful map $\mathcal M_{k+1}(\beta) \to \mathcal M_{1}(\beta)$
is hard to prove when we use single valued perturbations.
(The proof of \cite{fukaya:cyc} does not work.)
So instead of proving (\ref{foget1-k}) we use the right hand side as the
definition.
\par
Secondly there is a denominator $\ell!$ in (\ref{foget1-k})
which may not be invertible in $R$.
So instead of requiring $\overline b_1 \in R$ we require
$\exp(\overline b_1(\partial \beta)) \in R$. (Compare \cite{fooo:bulk} Remark 11.5.)
\par
After these explanations we state our results.
We conside a pair $(\rho,b_+)$ where
$$
\rho \in {\rm Hom} (\pi_1(L), R^*),
\quad
b_+ \in S^1(L;\Lambda_{0,{\rm nov}}^{(0) +,R}) \oplus
\bigoplus_{k\ge 1} S^{1+2k}(L;\Lambda_{0,{\rm nov}}^{(0) R})e^{-k}.
$$
Here $R^*$ is the group of units of the ring $R$.
We put
\begin{equation}\label{MC0def}
\widehat{\mathcal M}_{\text{weak}}
(L;\Lambda_{0,{\rm nov}}^R)
=
\{(\rho,b_+)
\mid
\sum_{\beta,k} T^{\omega(\beta)}e^{\mu(\beta)/2}\rho(\partial\beta) \frak m_{k,\beta}(b_+^k) \equiv 0 \mod \Lambda_{0,{\rm nov}}^{+,R}\text{\bf e}
\}.
\end{equation}
The infinite sum appearing in (\ref{MC0def})
converges in $T$ adic topology.
\begin{rem}
We may use a flat vector bundle in place of a flat line bundle $\rho$.
\end{rem}
\par
We say $(\rho,b_+)$ is {\it gauge equivalent} to $(\rho',b'_+)$
if $\rho$ is conjugate to $\rho'$ as representations of $\pi_1(L)$ to $R^{\ast}$ and $b_+$ is gauge equivalent to $b'_+$
as Maurer-Cartan elements of the filtered $A_{\infty}$ structure $\frak m_k^{\rho}$
in the sense of \cite{fooo-book1} Definition 4.3.19,
where
$$
\frak m_k^{\rho} = \sum_{\beta} T^{\omega(\beta)}e^{\mu(\beta)/2}\rho(\partial\beta) \frak m_{k,\beta}.
$$
We denote the set of gauge equivalence classes of
$\widehat{\mathcal M}_{\text{weak}}
(L;\Lambda_{0,{\rm nov}}^R)$ by
${\mathcal M}_{\text{weak}}
(L;\Lambda_{0,{\rm nov}}^R)$.
\begin{defn}\label{def:deformLambda0}
For $(\rho,b_+) \in {\mathcal M}_{\text{weak}}
(L;\Lambda_{0,{\rm nov}}^R)$
we define a deformed filtered $A_{\infty}$ structure $\frak m_{k}^{(\rho,b_+)}$
by
$$
\frak m_k^{(\rho,b_+)}
(x_1,\dots,x_k) = \sum_{\beta}\sum_{\ell_0,\dots ,\ell_k=0}^{\infty}
 T^{\omega(\beta)}e^{\mu(\beta)/2} \rho(\partial\beta) \frak m_{k+\sum\ell_i,\beta}
 (b_+^{\ell_0},x_1,\dots,x_k,b_+^{\ell_k}).
$$
\end{defn}
We use it in the same way to show the following:
\begin{theorem}
Theorems \ref{theoremB}, \ref{theoremD}, \ref{theoremE}, \ref{theoremG}
hold with $\mathcal M_{\text{\rm weak}}(L;R)$
replaced by ${\mathcal M}_{\text{\rm weak}}
(L;\Lambda_{0,{\rm nov}}^R)$.
\end{theorem}
\begin{rem}
We may take the pair $(\tilde\rho,b_{\text{high}})$
where
$$
\tilde\rho : \pi_1(L) \to
\{
y \in \Lambda_{0,{\rm nov}}^R
\mid
y \equiv \overline y \mod \Lambda_{0,{\rm nov}}^{+,R},
~\overline y \in R^*
\}
$$
and
$$
b_{\text{high}}
\in
\bigoplus_{k\ge 1} S^{1+2k}(L;\Lambda_{0,{\rm nov}}^{(0) R})e^{-k}
$$
in place of $(\rho,b_+)$.
\end{rem}
In \cite{fooo:toric1} we discussed the case of $T^n \cong L \subset X$ where
$X$ is a toric manifold and $T^n$ is an orbit of the $T^n$ action.
There we take
$
b = \sum x_i \text{\bf e}_i, ~ x_i \in \Lambda_{0,{\rm nov}}^{\C}.
$
We then change the coordinate from $x_i$ to $y_i = \exp(x_i)$.
This corresponds to take $\tilde\rho$ as above.
Thus in case $y_i$ is congruent to an element of $R^*$
modulo $\Lambda_{0,{\rm nov}}^{+,R}$ we can apply the argument of this section.
In the toric case, there are many examples where the leading order term
of $y_i$ is in a number field (a finite extension of $\Q$) or its integer ring.
(In particular, $y_i$ is  a unit of its appropriate localization.)
For example, in the case $X = \C P^n$ nonvanishing Floer cohomology
appears when $y_i$ is the $(n+1)$-th root of unity.
Thus we have many examples for which
it seems interesting to study
Floer cohomology over $\Lambda_{0,{\rm nov}}^R$ with
a Dedekind domain $R$.

\subsection{Hochschild homology, operator $\frak p'$ and
degeneration of the spectral sequence}
\label{Hochschild homology}

In this subsection we prove Theorem \ref{theoremD} (3) and the claim $PD[pt] \notin \operatorname{Im}(\delta_r)$ in
Theorem \ref{theoremE} (1).
In their proofs over $\Q$ coefficients
given in \cite{fooo-book1},
we used the operator $\frak p$ whose domain is the cyclic homology of $L$. See Subsections 3.8.1 and 6.4.2 in \cite{fooo-book1}.
Since it is rather difficult to study cyclic homology over torsion coefficients,
we use the Hochschild homology instead in this subsection.
\par
Let $(C,\{\frak m_k\})$ be a filtered $A_{\infty}$ algebra.
We put
\begin{equation}
CH(C[1]) = \widehat{\bigoplus}_{k=0}^{\infty} B_{k+1}(C[1]).
\end{equation}
We define the Hochschild differential $\delta^H$ on it by
\begin{equation}\label{hochboundary}
\aligned
&\delta^H(x_0\otimes \dots \otimes x_k) \\
=
&\sum_{0<i\le j\le k}
(-1)^{*_{1;i}}x_0 \otimes \dots \otimes x_{i-1} \otimes\frak m_{j-i}(x_i,\dots,x_{j-1})
\otimes x_j \otimes \dots \otimes x_k \\
& + \sum_{i=0}^k (-1)^{*_{1;i+1}}x_0 \otimes \dots \otimes
x_i \otimes \frak m_{k-i}(x_{i+1},\dots ,x_k)\\
&+\sum_{0\le i < j < k}(-1)^{*_{2;i,j}}
\frak m_{k+i-j+1}(x_{j+1},\dots,x_k,x_0,\dots,x_i) \otimes x_{i+1}\otimes \dots \otimes x_{j} \\
&+ \sum_{i=0}^k \frak m_{i+1}(x_0,\dots,x_i)\otimes x_{i+1} \otimes \dots \otimes x_k,
\endaligned\end{equation}
where
$*_{1;i} = \deg x_0+\dots + \deg x_{i-1} + i$,
$*_{2;i,j} = (\deg x_0 + \dots + \deg x_{j}+j+1)(\deg x_{j+1}+\dots + \deg x_k+k-j)$.
(Note the sum of terms containing $\frak m_0$ is
$
\pm x_0 \otimes \frak m_0(1) \otimes x_1 \otimes \dots \otimes x_k
\pm \dots \pm x_0 \otimes \dots \otimes x_k \otimes \frak m_0(1)$, each term of which comes  from the first and second lines of \eqref{hochboundary}.
)
It is straightfoward to check
$$
\delta^H \circ \delta^H  = 0.
$$
We thus obtain a chain complex
$(CH(C[1]),\delta^H)$.
\par
We next define a chain map
$$
\frak p' : (CH(S(L;\Lambda_{0,{\rm nov}}^R)[1]),\delta^H)
\to S(M;\Lambda_{0,{\rm nov}}^R).
$$
We use the moduli space $\mathcal M_{k;1}(\beta)$ which we defined in Section
\ref{moduli}. (It is the moduli space of pseudo-holomorphic discs of homology class $\beta$
with $k$ boundary and $1$ interior marked points.)
We have evaluation maps
$$
(ev,ev^+) =
(ev_1,\dots,ev_k,ev^+): \mathcal M_{k;1}(\beta) \to L^{k} \times M.
$$
The boundary of the moduli space is described as follows.
\begin{equation}\label{Mk1boundary}
\aligned
\partial \mathcal M_{k;1}(\beta)
=
&\bigcup_{\beta_1+\beta_2+\beta}\bigcup_{k_1+k_2=k}\bigcup_{i=1,\dots,k_2}
\mathcal M_{k_1+1}(\beta_1) {}_{ev_0}\times_{ev_i}\mathcal M_{k_2+1;1}(\beta_2) \\
&\cup \bigcup_{\beta_1+\beta_2+\beta}\bigcup_{k_1+k_2=k}\bigcup_{i=1,\dots,k_2}
\mathcal M_{k_1+1;1}(\beta_1) {}_{ev_0}\times_{ev_i}\mathcal M_{k_2+1}(\beta_2).
\endaligned
\end{equation}
Let $P_1,\dots,P_k$ be smooth singular chains on $L$.
We put
$$
\mathcal M_{k;1}(\beta;P_1,\dots,P_k)
= \mathcal M_{k;1}(\beta) {}_{ev}\times (P_1 \times \dots \times P_k).
$$
In Sections \ref{AinftyconstrZ2} and \ref{Zcoefficient} we
took a global section of
$$
\mathcal M_{k+1}(\beta;P_1,\dots,P_k)
= \mathcal M_{k+1}(\beta) {}_{(ev_1,\dots,ev_k)}\times (P_1 \times \dots \times P_k),
$$
which is normally conical.
Then we took a triangulation of its zero set
so that the virtual fundamental chain is defined.
(We use our assumption that $(M,\omega)$ is spherically positivity here.)
Now we have:
\begin{lem}\label{1pointsection}
We can take a global section of
$\mathcal M_{k;1}(\beta;P_1,\dots,P_k)$ which satisfies the
conclusion of Theorem \ref{maintechnicalresult} and is
compatible with the
description $(\ref{Mk1boundary})$ of its boundary.
\end{lem}
We clarify the meaning of compatibility with (\ref{Mk1boundary}) during the proof.
\begin{proof}
Once we clarify the meaning of the compatibility with (\ref{Mk1boundary}),
the lemma follows from (relative version of) Theorem \ref{maintechnicalresult}, by
an induction on $\omega(\beta)$ and $k$.
(We define triangulation of its zero set at the same time by induction.)
\par
We first study the first term of the right hand side of (\ref{Mk1boundary}).
We consider
$$
\mathcal M_{k_1+1}(\beta_1;P_i,\dots,P_{i+k_1-1}).
$$
It corresponds to the first factor of  the first term of the right hand side of (\ref{Mk1boundary}).
We fixed a global section of it already.
A triangulation of its zero set is also taken already.
So using $ev_0$ we have smooth singular chains
$P'_1,\dots,P'_a$ of $L$ such that
$$
P'_1 + \dots + P'_a = \frak m_{k_1,\beta_1}(P_i,\dots,P_{i+k_1-1}).
$$
(In case $R$ is not characteristic 2 we need an appropriate sign.)
We next consider
\begin{equation}\label{92}
\mathcal M_{k_2+1;1}(\beta_2;P_1,\dots,P_{i-1},P'_j,P_{i+k_1},\dots,P_k).
\end{equation}
The sum of these spaces over $j=1,\dots,a$ corresponds to the first term
of the right hand side of (\ref{Mk1boundary}).
By induction hypothesis a global section of (\ref{92}) and the triangulation
of its zero set are already taken.
\par
We next study the second term
of the right hand side of (\ref{Mk1boundary}).
We consider
\begin{equation}\label{93}
\mathcal M_{k_2+1}(\beta_2) {}_{ev_1,\dots,ev_{i-1},ev_{i+1},\dots,ev_{k_2}} \times (P_1 \times \dots \times P_{i-1} \times P_{i+k_1} \times \dots
 \times P_k).
\end{equation}
This corresponds to the second factor of the second term
of the right hand side of (\ref{Mk1boundary}).
We rename the boundary marked points such that the $i$-th boundary marked point
becomes the $0$-th marked
point and that
the cyclic order of the boundary marked points is preserved. Then (\ref{93}) is identified with
$$
\mathcal M_{k_2+1}(\beta_2;P_{i+k_1},
\dots, P_k,P_1, \dots, P_{i-1}).
$$
In Sections \ref{AinftyconstrZ2} and \ref{Zcoefficient} we took a
global section of this moduli space and a triangulation of its zero set.
Together with $ev_0$, which is the evaluation map
at the renamed $0$-th marked point on the $\beta_2$-disc, it defines smooth singular chains $P''_1,\dots,P''_{a'}$ of $L$.
We then consider
\begin{equation}\label{94}
\mathcal M_{k_1+1,1}(\beta_1;P''_j,P_i,\dots,P_{i+k_1-1}).
\end{equation}
By induction hypothesis a global section of (\ref{94}) and a triangulation
of its zero set are already taken.
Note (\ref{94}) corresponds to the second term
of the right hand side of (\ref{Mk1boundary}).
\par
Thus we have found the way how to define the global section
of the boundary of
$\mathcal M_{k;1}(\beta;P_1,\dots,P_k)$.
The compatibility at the codimension $\ge 2$ corner
can be checked easily.
The proof of Lemma \ref{1pointsection} is now complete.
\end{proof}
\begin{defn} We put
$$
\frak p'_{k;\beta}
(P_1,\dots,P_k)
= (\mathcal M_{k;1}(\beta;P_1,\dots,P_k),ev_0)
\in S(M;R).
$$
Here $(S(M;R),\delta_M)$ is the smooth singular chain complex with $R$ coefficients.
We also denote by $\delta_M$ its extension over
$\Lambda_{0,{\rm nov}}^R$ coefficients.
We define
$$
\frak p'_{k} = \sum_{\beta} T^{\omega(\beta)}
e^{\mu(\beta)/2}\frak p'_{k;\beta} :
B_{k}(S(L;\Lambda_{0,{\rm nov}}^R)[1]) \to S(M;\Lambda_{0,{\rm nov}}^R)
$$
and
$$
\frak p' = \sum_k \frak p'_{k}
:
CH(S(L;\Lambda_{0,{\rm nov}}^R)[1]) \to S(M;\Lambda_{0,{\rm nov}}^R).
$$
\end{defn}
Let $i_{!} : S^k(L;R) \to S^{k+n}(M;R)$ be the Gysin homomorphism. Then we show the following properties which are analogs of (3.8.10.1) and (3.8.10.2) of \cite{fooo-book1}.
\begin{lem}
\begin{eqnarray}
&\frak p_1^{\prime} \equiv i_{!} \mod \Lambda_{0,{\rm nov}}^{+,R}, \label{p1mod}\\
&\delta_M \circ \frak p' + \frak p' \circ \delta^H = 0.\label{p'delta}
\end{eqnarray}
\end{lem}
\begin{proof}
\eqref{p1mod} is nothing but (3.8.10.1) of \cite{fooo-book1}. Note that $\frak {p}_{1,0}^{\prime}= \frak {p}_{1,0}=i_{!}$.
\eqref{p'delta} is a consequence of Lemma \ref{1pointsection}.
In fact, the first term of the right hand side of (\ref{Mk1boundary})
corresponds to the first and the second terms of (\ref{hochboundary})
and the second term of the right hand side of (\ref{Mk1boundary})
corresponds to the third and the fourth terms of (\ref{hochboundary}).
\end{proof}
We can include the homotopy unit to the story of $\frak p'$ as follows.
Let $(C,\{\frak m_k\})$ be a filtered $A_{\infty}$ algebra.
We put
$C^+ = C \oplus \Lambda_{0,{\rm nov}}^R\text{\bf e}^+ \oplus \Lambda_{0,{\rm nov}}^R{\bf f}$.
If $C$ has a homotopy unit, we can extend the filtered $A_{\infty}$ structure
of $C$ to ones on $C^+$ such that $\text{\bf e}^+$ is the exact unit.
(See \cite{fooo-book1} Section 3.4 for the precise definition.)
Our $A_{\infty}$ algebra $S(L;\Lambda_{0,{\rm nov}}^R)$ has a homotopy unit.
(The proof of this fact is the same as in \cite{fooo-book2} Section 7.3 using (relative version of) Theorem \ref{maintechnicalresult}.)
Now we can use the construction of \cite{fooo-book2} Subsection 7.4.1 to
extend $\frak p'$ to
$$
\frak p^{\prime +} :
CH(S(L;\Lambda_{0,{\rm nov}}^R)^+[1]) \to S(M;\Lambda_{0,{\rm nov}}^R)
$$
that satisfies
\begin{equation}\label{95-1}
\delta_M \circ \frak p^{\prime +} + \frak p^{\prime +} \circ \delta^H = 0.
\end{equation}
Moreover it satisfies
\begin{equation}\label{punital}
\frak p_k^{\prime +}(\dots,\text{\bf e}^+,\dots) = 0
\end{equation}
for $k\ne 1$.
We also have
$
\frak p_1^{\prime +}(\text{\bf e}^+) = PD[L].
$
\begin{rem}
The $R=\Q$ version of this statement is (3.8.10.4), (3.8.10.5) \cite{fooo-book1}.
We note that (3.8.10.6) \cite{fooo-book1} states
$\frak p_2^{+}(x\text{\bf e}^+) = x$. Actually this is an error.
The correct statement is
$\frak p_2^{+}(x\text{\bf e}^+) = 0$.
\end{rem}
\par
Now we use this operator to prove Theorem \ref{theoremD} (2) and
$PD[pt] \notin \operatorname{Im}(\delta_r)$ in
Theorem \ref{theoremE} (1).
Let $b \in \mathcal M_{\text{weak}}(L;R)$.
Namely $b \in S(L;\Lambda_{0,{\rm nov}}^R)^+$ is degree $1$
such that $b \equiv 0 \mod  \Lambda_{0,{\rm nov}}^{+,R}$ and $b$ satisfies
\begin{equation}\label{bweakbdd}
\sum_k\frak m_k(b^k) \equiv 0 \mod \Lambda_{0,{\rm nov}}^{+,R}\text{\bf e}^+.
\end{equation}
We define
$\frak p^{\prime +}_b :
S(L;\Lambda_{0,{\rm nov}}^R)^+[1] \to S(M;\Lambda_{0,{\rm nov}}^R)$
by
\begin{equation}
\frak p^{\prime +}_b(x)
= \frak p^{\prime +}(x e^b).
\end{equation}
\begin{lem}\label{pbprop}
\begin{eqnarray}
&\frak p^{\prime +}_b \equiv i_{!} \mod \Lambda_{0,{\rm nov}}^{+,R}, \label{98}\\
&\frak p^{\prime +}_b \circ \frak m_1^b + \delta_M \circ \frak p^{\prime +}_b = 0.\label{99}
\end{eqnarray}
\end{lem}
\begin{proof} Since $b\equiv 0 \mod \Lambda_{0,{\rm nov}}^{+,R}$,
(\ref{98}) follows from \eqref{p1mod}.
To prove (\ref{99}), we first note that (\ref{95-1}) implies
$$
\delta_M (\frak p^{\prime +} (xe^b))
+
\frak p^{\prime +} (\frak m(e^bxe^b) e^b)
+
\frak p^{\prime +} (xe^b\frak m(e^b) e^b) = 0.
$$
The first term is $\delta_M \circ \frak p^{\prime +}_b$ by definition.
The second term is
$\frak p^{\prime +}_b \circ \frak m_1^b$ by definition.
The third term is zero by (\ref{punital}) and (\ref{bweakbdd}).
Hence the lemma.
\end{proof}
The conclusion of Lemma \ref{pbprop} is the same as \cite{fooo-book1}
Lemma 6.4.5 (for the case $\frak b=0$).
Therefore we can prove Theorem \ref{theoremD} (3) and the claim $PD[pt] \notin \operatorname{Im}(\delta_r)$
in the same
way as in \cite{fooo-book1} Subsection 6.4.2 and
Subsection 6.4.3 respectively.
\qed
\begin{rem}
We do not know how to generalize the formula \cite{fooo-book1} (3.8.10.3)
$$
\frak p_1\circ \frak m_0(1) + \delta_M \circ \frak p_0(1)
+ GW_{0,1}(M)(L) = 0
$$
to our situation.
So we do not know how to generalize \cite{fooo-book1}
Theorem 3.8.11 to $\Z$ or $\Z_p$ coefficients.
In fact, the proof of \cite{fooo-book1} (3.8.10.3)
uses the moduli space $\mathcal M_{0;1}(\beta)$
of pseudo-holomorphic discs without boundary marked point.
An element of this moduli space may have a nontrivial
automorphism even in case it does not have a sphere bubble.
So the argument of Section \ref{dimensionsec} does not apply.
\end{rem}

\bigskip

\noindent Kenji Fukaya\\
Simons Center for Geometry and Physics, State University of New York, Stony Brook, NY 11794-3636, USA \&  Department of Mathematics, Kyoto
University, Kyoto, Japan\\
E-mail: fukaya@math.kyoto-u.ac.jp\\

\medskip

\noindent Yong-Geun Oh\\
Center for Geometry and Physics, Institute for
Basic Sciences (IBS), Pohang, Gyungbuk 790-784, Korea, \\
Department of Mathematics, POSTECH, Pohang, Korea, \& \\
Department of Mathematics, University of Wisconsin,
Madison WI 53705, USA.\\
E-mail: oh@math.wisc.edu\\
\medskip

\noindent Hiroshi Ohta\\
Graduate School of Mathematics, Nagoya University, Nagoya, 464-8602, Japan \& 
Korea Institute for Advanced
Study, Seoul, Korea\\
E-mail: ohta@math.nagoya-u.ac.jp\\

\medskip

\noindent Kaoru Ono\\
Research Institute for Mathematical Sciences, Kyoto University,
Kyoto, 606-8502, Japan \&   Korea Institute for Advanced Study, Seoul, Korea\\
E-mail: ono@kurims.kyoto-u.ac.jp

\end{document}